\documentclass[a4paper,11pt,
 leqno
]{article}
\usepackage{times}
\usepackage{comment}
\usepackage[utf8]{inputenc}
\usepackage[T1]{fontenc}
\usepackage{amsthm}
\usepackage{amssymb}
\usepackage{bbm}
\usepackage{mathrsfs}
\usepackage{amsmath}
\usepackage{amsfonts}
\usepackage{mathtools}
\usepackage{graphicx}
\usepackage{float}
\usepackage{dsfont}
\usepackage{enumerate}
\usepackage{soul,xcolor}
\usepackage{enumitem}
\usepackage{csquotes}
\usepackage{IEEEtrantools}
\usepackage{appendix}
\usepackage{constants}
\usepackage[obeyDraft]{todonotes}
\usepackage{tikz-cd}
\usetikzlibrary{decorations.pathreplacing}
\usepackage{enumerate}
\usepackage{mathabx}
\usepackage[a4paper, left=1.0 in, right=1.0 in, top=0.9 in, bottom=1in]{geometry}

\setstcolor{magenta}

\usepackage
{hyperref}
\hypersetup{linktocpage,
  colorlinks   = true, 
  urlcolor     = blue, 
  linkcolor    = blue,
  citecolor   = blue 
}\usepackage[all]{hypcap}  
\usepackage{breakurl}

\numberwithin{equation}{section}

\usepackage{caption}
\title{
\normalsize
\textbf{{
ON CABLE-GRAPH PERCOLATION BETWEEN DIMENSIONS 2 AND 3
}}}
\author{}
\date{}
\makeatletter
\newcommand{\N}{\mathbb N}
\newcommand{\1}{\mathbbm 1}
\newcommand{\Z}{\mathbb Z}
\newcommand{\R}{\mathbb R}
\renewcommand{\P}{\mathbb{P}}
\newcommand{\E}{\mathbb{E}}

\renewcommand{\epsilon}{\varepsilon}

\newcommand{\capacity}{\text{cap}}

\newcommand{\lO}{\mathcal{O}}

\newtheorem{theorem}{Theorem}[section]
\newtheorem{lemma}[theorem]{Lemma}
\newtheorem{proposition}[theorem]{Proposition}
\newtheorem{corollary}[theorem]{Corollary}
\theoremstyle{definition}

\newtheorem{remark}[theorem]{Remark}

\newcommand{\floor}[1]{\left\lfloor{#1}\right\rfloor}
\newcommand{\ceil}[1]{\left\lceil{#1}\right\rceil}

\newcommand{\cref}[1]{\text{\upshape\ref{#1}}}

\newcommand{\WZ}[1]{{\red{[\small WZ: #1]}}}

\newcommand{\Q}{\mathbb{Q}}

\newcommand{\dX}{\bar{X}}
\newcommand{\cL}{\mathcal{L}}

\newcommand{\dg}{\bar{g}}

\newcommand{\cP}{\mathcal{P}}
\newcommand{\tB}{\tilde{B}}
\newcommand{\Slab}{\mathbb{S}_N}
\newcommand{\tSlab}{\tilde{\mathbb{S}}_N}
\newcommand{\Fbox}{F^{h}_N}
\newcommand{\Th}{\theta^h_N}

\newcommand{\cC}{\mathcal{C}}

\newcommand{\dist}{\text{dist}}

\newcommand{\Ann}{\mathbb{A}}

\newcommand{\bfa}{\mathbf{a}}
\newcommand{\bfb}{\mathbf{b}}
\newcommand{\bfd}{\mathbf{d}}
\newcommand{\bfe}{\mathbf{e}}
\newcommand{\bfG}{\mathbf{G}}
\newcommand{\tP}{\tilde{P}}

\newcommand{\tg}{\tilde{g}}

\renewcommand{\phi}{\varphi}
\renewcommand{\tilde}{\widetilde}
\renewcommand{\hat}{\widehat}
\renewcommand{\epsilon}{\varepsilon}

\newcommand{\Pzd}[1]{P_{#1}^{2}}

\newcommand{\Pz}[1]{P_{#1}^{1}}

\newconstantfamily{c}{symbol=c}

\usepackage{color}
\definecolor{Red}{rgb}{1,0,0}
\definecolor{Blue}{rgb}{0,0,1}
\definecolor{Olive}{rgb}{0.41,0.55,0.13}
\definecolor{Yarok}{rgb}{0,0.5,0}
\definecolor{Green}{rgb}{0,1,0}
\definecolor{MGreen}{rgb}{0,0.8,0}
\definecolor{DGreen}{rgb}{0,0.55,0}
\definecolor{Yellow}{rgb}{1,1,0}
\definecolor{Cyan}{rgb}{0,1,1}
\definecolor{Magenta}{rgb}{1,0,1}
\definecolor{Orange}{rgb}{1,.5,0}
\definecolor{Violet}{rgb}{.5,0,.5}
\definecolor{Purple}{rgb}{.75,0,.25}
\definecolor{Brown}{rgb}{.75,.5,.25}
\definecolor{Grey}{rgb}{.7,.7,.7}
\definecolor{Black}{rgb}{0,0,0}
\def\red{\color{Red}}

\def\black{\color{Black}}

\usepackage{titlesec}
\titleformat{\subsection}[runin]{\normalfont\bfseries}{\thesubsection.}{.5em}{}[.]\titlespacing{\subsection}{0pt}{2ex plus .1ex minus .2ex}{.8em}
\titleformat{\subsubsection}[runin]{\normalfont\bfseries}{\thesubsubsection.}{.5em}{}[.]
\titlespacing{\subsubsection}{0pt}{2ex plus .1ex minus .2ex}{.8em}

\usepackage{titletoc}

\dottedcontents{section}[4em]{}{2.9em}{0.7pc}
\dottedcontents{subsection}[0em]{}{3.3em}{1pc}

\usepackage{multicol}
\usepackage{parcolumns}

\usepackage{tabu}
\usepackage{caption}
\begin{document}
\thispagestyle{empty}
\maketitle
\vspace{0.1cm}
\begin{center}
\vspace{-1.9cm}
Pierre-Fran\c cois Rodriguez$^{1,2}$ and Wen Zhang$^1$

\end{center}
\vspace{0.1cm}
\begin{abstract}
\centering
\begin{minipage}{0.80\textwidth}
We consider the Gaussian free field on two-dimensional slabs with a thickness described by a height $h$ at spatial scale $N$. We investigate the radius of critical clusters for the associated cable-graph percolation problem, which depends sensitively on the parameter $h$. Our results unveil a whole family of new ``fixed points'', which interpolate between recent results from \cite{arXiv:2303.03782} in two dimensions and from \cite{drewitz_critical_2024} and \cite{cai_one-arm_2024} in three dimensions, and describe critical behaviour beyond those regimes. In the delocalised phase, the one-arm decay exhibits a ``plateau'', i.e.~it doesn't depend on the speed at which the variance of the field diverges in the large-$N$ limit. Our methods rely on a careful analysis of the interplay between two- and three-dimensional effects for the underlying random walk, which manifest themselves in a corresponding decomposition of the field.
\end{minipage}
\end{abstract}

\thispagestyle{empty}

\vspace{12.5cm}
\begin{flushright}

\noindent\rule{5cm}{0.4pt} \hfill 5 December 2025 \\
\bigskip
\end{flushright}

\begin{multicols}{2}

\noindent$^1$Imperial College London\\
 Department of Mathematics\\
 %180 Queen's Gate\\
 London SW7 2AZ, UK \\
%  \url{p.rodriguez@imperial.ac.uk} 
  \url{kate.zhang23@imperial.ac.uk} 

\columnbreak
\begin{flushright}
\hfill $^2$ Center for Mathematical Sciences\\
\hfill University of Cambridge\\
%\hfill Wilberforce Road\\
\hfill Cambridge CB3 0WB, UK\\
\hfill\url{pfr26@cam.ac.uk}
\end{flushright}
\end{multicols}

\newpage

\section{Introduction}\label{sec:intro}

For the cable-system Gaussian free field $\varphi$, it was recently proved in \cite{drewitz_critical_2024} and separately in \cite{cai_one-arm_2024} (see also \cite{ding_percolation_2020,drewitz_arm_2023} for precursor results)
that the probability to connect $0$ to distance $N$ within $\{ \varphi \geq 0\}$ is of order $N^{-\frac12}$ on $\Z^3$, up to multiplicative constants. For the continuum analogue in a two-dimensional disk, it was proved in \cite{arXiv:2303.03782} that the corresponding quantity decays like $(\log N)^{-\frac12+o(1)}$ as $N \to \infty$.

One motivation for this article is to unify these results. To this effect we design a rigorous (and refined) version of what physicists often refer to as  ``$\epsilon$-expansion'' (cf.~\cite{WILSON197475}) around the lower-critical dimension of the problem (i.e.~$\epsilon=d-2$), by compactifying one of the three spatial dimensions; i.e.,~we consider the model in appropriately scaled two-dimensional slabs of height $h=h_N$ at spatial scale $N$.

A natural -- but, as it turns out, naive -- choice is $h=N^{\varepsilon}$ for $\varepsilon \in [0,1]$. As will become clear, our results recover the above characteristic decay in the extremal cases $\varepsilon =0$, corresponding to dimension $d=2$ (in fact also improving on the $o(1)$ in that case), as well as $\varepsilon =1$, corresponding to $d=3$. For $\varepsilon \in (0,1)$, our findings are in accordance with results of \cite{drewitz_critical_2023} (picking $\alpha=2+\varepsilon$, $\nu=\varepsilon$ therein) valid on a large family of transient graphs with polynomial volume growth, thus supplying new examples in this class. The polynomial scaling $h=N^{\varepsilon}$ is but a choice, and in fact rather coarse. Indeed it does not properly ``detect'' the lower-critical dimension; i.e.,~the bounds obtained trivialize in the limit $\varepsilon \downarrow 0$.

Our approach deals at once with all possible choices of height $h$ growing at most linearly with $N$.
For generic $h$, they  unveil a whole new family of ``long-range fixed points,'' extending those found in~\cite{drewitz_critical_2023,chalhoub2024universality} beyond the familiar polynomial scaling at criticality. Among others, the resulting one-arm decay displays a ``plateau'' at the onset of the regime for $h$ where $\varphi$ delocalises, and illuminates the logarithmic-to-polynomial transition occurring at criticality between dimensions $2$ and $3$, cf.~Fig.~\ref{F:plateau} below.

\medskip

We now describe our results more precisely, starting with the relevant notion of slabs. Let ${h}: (0,\infty) \rightarrow \R$ be a non-decreasing function such that $1\leq {h}(t)\leq t$ for $t>0$.
For $N \geq 1$, we set  
 \begin{equation}\label{def:slab}
 h_N= \floor{{h}(N)}, \quad \mathbb{S}_N= \mathbb{Z}^2 \times (\Z / h_N \Z),
\end{equation}
and consider the graph with vertex set $\mathbb{S}_N$, a (cylindrical) slab. We frequently use the notation $x=(y,z) \in \mathbb{S}_N$ (and $x'=(y',z')$ etc.) with $y\in \mathbb{Z}^2$ and $z \in (\Z / h_N \Z)$. We sometimes refer to the $y$- and $z$-components as \emph{horizontal} and \emph{vertical}, respectively. The choice of periodic boundary condition in \eqref{def:slab} is a matter of convenience; see Remark~\ref{R:green-additional},\ref{R:bc} below regarding other natural choices.

Attached to $\Slab$ for every $N  \geq 1$ is a continuous-time 
Markov chain, which jumps between neighbors on $\Slab$ at unit rate and gets killed at rate $N^{-2}$; see the beginning of Section~\ref{sec:green-slab} for the precise setup. We write $P_x$ for the canonical law of $X$ starting from $ x \in \mathbb{S}_N$. For $N \geq 1$ and $x,x^\prime \in \mathbb{S}_N$,  the Green's function of $X$ is defined as 
\begin{equation} \label{def:green}
  g_{N}^h(x,x^\prime)= %\frac{1}{\lambda_{x^\prime}} 
  \frac1{\lambda_{x^\prime}}E_{x}\Big[\int_0^\infty \1_{\{X_t=x^\prime\}} \,dt\Big].
\end{equation}
We will often abbreviate $ g_N
= g_N^h$ in the sequel, as with other quantities (such as $\Slab$) that implicitly depend on the height profile $h$. We note that $g_N$ is symmetric, that $g_N(x,x^\prime)=g_N(0,x^\prime-x)$ by translation invariance (where addition is understood mod~$h_N$ in the vertical component), and write $g_N(x)\coloneqq g_N(0,x)$ for all $x\in \mathbb{S}_N$. We will later show, see Theorem~\ref{thm:green_main} below for this and more, that
\begin{equation}  \label{eq:green_estimate0intro}
     g_N(x) 
    \asymp
    \frac{1}{\|x\|\vee 1} +
     \frac{1}{h_N} K_0\left(\frac{|y|\vee h_N}{N}\right), \quad x=(y,z) \in \mathbb{S}_N
  \end{equation}
holds for all $N \geq1$ and $\|x\| \leq N$ (say), where $a_N \asymp b_N$ means that $c \leq (a_N / b_N) \leq C$ for constants $c,C \in (0,\infty)$ (uniform in $x$, $N$ and $h$), and $K_0: (0,\infty) \to (0,\infty)$ with
\begin{equation} \label{eq:K_0}
    K_0(t)\stackrel{\text{def.}}{=} \int_0^{\infty} e^{-t\cosh(r)} \mathrm{d}r, \quad t >0
\end{equation}
\black
denotes the zeroth order modified Bessel function of the second kind (see \cite[page 917, (8.432)]{gradshteyn_table_2014} for this representation of $K_0$). In \eqref{eq:green_estimate0intro} and below for $x = (y,z) \in \mathbb{S}_N$, we write $\|x\| \coloneqq |(y, \hat{z})|$ with $\hat{z}$ the projection of $z$ onto $\mathbb{Z}$ such that $|\hat{z}|=d_{(\Z/h_N\Z)}(0,z) \leq \frac{h_N}{2}$, where $d_G$ is the graph distance on $G$.

It is instructive to highlight the limiting cases $h_N=1$ and $h_N=N$ in \eqref{eq:green_estimate0intro}, corresponding to the choices $h(\cdot)\equiv 1$ and $h=\text{id}$ (i.e.~$h(t)=t$) above \eqref{def:slab}. In particular, since $K_0(t)\sim\log(\frac{1}{t})$ as $t\to 0^+$, see Lemma~\ref{L:Bessel},~\eqref{eq:green_estimate0intro} (and more precisely, Theorem \ref{thm:green_main}) recovers the following familiar boundary cases: for all $\|x \| \leq N$ say,
  \begin{equation*} 
    c \log\big(\tfrac{N}{ |y| \vee 1}\big)
    \leq g_N^h(x) 
    \leq 
    C \log\big(\tfrac{N}{ |y| \vee 1}\big),
    \quad \text{ when } h(\cdot) \equiv1,
  \end{equation*}
  (with $x=(y,0)$) and
  \begin{equation*} 
    {c}(\|x\| \vee 1)^{-1}
    \leq g_N^h(x) 
    \leq 
    {C}(\|x\| \vee 1)^{-1},
    \quad \text{ when } h=\text{id},
  \end{equation*}
corresponding to the $\Z^2$ and $\Z^3$ regimes for the Green's function of the walk, respectively (in the former case, with suitable killing at spatial scale $N$). We refer to Remark~\ref{R:green},\ref{R:green-coro} for a thorough discussion of the various regimes emerging from \eqref{eq:green_estimate0intro}. Loosely speaking, the logarithmic and polynomial terms appearing in (\ref{eq:green_estimate}) correspond to two- and three-dimensional contributions to $g_N^h(x)$. 

%In fact, the relationship between $\|x\|$ and $h_N$ determines whether $g_N(x)$ behaves according to the $\Z^3$ regime or the $\Z^2$ regime.
%When the distance is not large enough for the walk to ``feel the height", i.e. $|x|\ll h_N \frac{1}{K_0(\frac{|y|\vee h_N}{N})}$, the behaviour of $g_N(x)$ is dominated by the polynomial term. On the other hand, when $h_N \frac{1}{K_0(\frac{|y|\vee h_N}{N})}\ll |x|$, the walk ``feels the height" and the logarithmic term dominates the behaviour. 

Associated to the weighted graph $\Slab=(\mathbb{S}_N,\lambda,\kappa)$ is the mean zero Gaussian field $\varphi$ indexed by $\mathbb{S}_N$, with covariance
\begin{equation} \label{eq:GFF-disc}
    \E_N^h[\phi_{x}\phi_{x^\prime}]=g_N^h(x-x^\prime), \quad x,x^\prime\in\mathbb{S}_N,
\end{equation}
with $g_N^h(\cdot, \cdot)=g_N^h(\cdot, \cdot)$ as defined in \eqref{def:green}. The field $\varphi$, whose law is denoted by $\P_N \equiv \P_N^h$ in the sequel, is the (discrete) Gaussian free field on $\Slab$. Depending on the choice of $h_N$ (see \eqref{def:slab}), the free field $\varphi$ can be both {\em localized} or {\em delocalized}. Indeed, it will follow as a corollary of Theorem~\ref{thm:green_main} (see \eqref{eq:green_asymp_macx} and Remark~\ref{R:green},\ref{R:green-coro} below for the details) that, 
  \begin{equation} \label{eq:variance-GFF}
      g_N^h(0) \sim
      \begin{cases*}
         \frac{3}{\pi}\frac{\log {N}}{h_N}, & if $\displaystyle \lim_{N\to\infty} \textstyle\frac{\log N}{h_N}=\infty$ \\
         \Cl[c]{d3_green_on_diag} +\frac{3}{\pi}\lambda , &  if $\displaystyle\lim_{N\to\infty} \textstyle\frac{\log N}{h_N}=\lambda\in [0,\infty)$  
        \end{cases*}
        ,
  \end{equation}
where $\Cr{d3_green_on_diag} \coloneqq  g_{\Z^3}(0)$ denotes the Green's function at $0$ of the random walk on $\Z^3$ and $a_N \sim b_N$ means that the ratio tends to $1$ in the limit $N \to \infty$. In particular, the field $\varphi$ delocalizes whenever $h_N = o(\log N)$, i.e.~the variance $g_N^h(0)$ diverges as $N \to \infty$, whereas the latter case for $h_N$ describes a regime of uniform (in $N$) transience for the walk $X$, in which $\sup_N g_N(0)< \infty$.

\medskip

We now discuss applications of the above random walk results to the bond percolation problem induced by the excursion sets of the metric Gaussian free field on the slab. Various arm events for this model have been recently studied below the mean-field regime, in low dimensions $d \geq 3$ on $\Z^d$ \cite{ding_percolation_2020,drewitz_arm_2023, drewitz_critical_2024, cai_one-arm_2024, cai2025quasimultiplicativityregularitymetricgraph, cai2025heterochromatictwoarmprobabilitiesmetric, cai2025separation}  and more generally on a collection of transient graphs where the Green's function exhibits polynomial decay, see previous references, and also \cite{drewitz_cluster_2022,drewitz_critical_2023}.

To obtain the cable system, $\tSlab^h \equiv \tSlab$, of the slab $\Slab$, we first replace all edges $\{x_1,x_2\}$ by open intervals $I_{x_1,x_2}$ of length $(2\lambda_{x_1,x_2})^{-1}$ and glue them through their endpoints. We then attach a half-open interval $I_x$ of length $(2\kappa_x)^{-1}$ to each $x\in \mathbb{S}_N$ (see \cite{drewitz_critical_2024,lupu_loop_2016} for precise definitions). We endow $\tSlab$ with the natural geodesic distance which assigns length 1 to each interval. The chain $X$ can then be naturally extend to a Markov process on $\tSlab$ with continuous trajectories. We denote by $\tP_x$ the law of the canonical diffusion on $\tSlab$ when starting at $x\in \tSlab$.
 We denote by $\tg_N(\cdot,\cdot)$ the Green's function associated to this diffusion. Informally, one can construct a diffusion with law $\tP_x$ by running a Brownian motion starting at $x_1$ on $I_e$, where $x_1\in I_e$ and $e$ is an edge of the slab, until a vertex $x_2\in \mathbb{S}_N$ is reached. One then chooses uniformly at random an interval glued to the vertex $x_2$ and runs a Brownian excursion on the interval until a vertex is reached. This procedure is repeated until the process reaches the end of the open interval $I_x$ for some $x\in \mathbb{S}_N$. We refer to \cite[Section 2.1]{drewitz_cluster_2022} for a formal definition through its Dirichlet form. Note that by taking $X^{\mathbb{S}_N}$ to be the \emph{trace} process of $X$ on $\mathbb{S}_N$ (see the precise definition of $X^{\mathbb{S}_N}$ around \cite[(2.4))]{drewitz_cluster_2022}), it follows from \cite[Theorem 6.2.1]{fukushima_dirichlet_2010} that the law of $X^{\mathbb{S}_N}$ under $\tP_x$ is the same as the law of $X$ under $P_x$. In particular, this implies that $(\tg_N)_{\vert_{\mathbb{S}_N \times \mathbb{S}_N}}= g_N.$

We consider the Gaussian free field $(\phi_x)_{x\in\tSlab}$ with canonical law $\P_N$, which is the continuous centered Gaussian field with covariance function $\tg_N$. Its restriction to $\Slab$ is the field defined in \eqref{eq:GFF-disc}, which justifies the slight abuse of notation. Our main interest is in the function
\begin{equation} \label{eq:connectivity_function}
\begin{split}
 & \Th(R) \stackrel{\text{def.}}{=} \P_N^h\big(0 \leftrightarrow\partial B_R \big), 
 \quad  R \geq 0
  \end{split}
\end{equation}
where $B_R\subset \mathbb{S}_N$ refers to the ball of radius $R$ around $0$ in the metric $\|\cdot\|$ and for $U,V\subset \tSlab$, we denote by $\{U\leftrightarrow V\}$ the event that $U$ and $V$ are connected by a continuous path in $\tSlab \cap
\{ \varphi \geq 0 \}$. Recall our assumptions on $h(\cdot)$ from  above \eqref{def:slab}, which are in force throughout this paper. { In the sequel, we write $a_N\gg b_N$ and $a_N\gtrsim b_N$ to mean $\lim_N\frac{b_N}{a_N}=0$ and $\limsup_N \black\frac{b_N}{a_N}<\infty$ respectively.} The following theorem yields up-to-constant bounds for the function $ \Th(R)$ that manifest critical behaviour. We refer to the end of this introduction regarding our policy with constants $c,C$ etc. In the sequel $\Cl{C:range}<\infty$ is an arbitrary positive constant and $a \vee b =\max\{ a,b \}$, $a \wedge b =\min \{a,b\}$.

\begin{theorem} \label{thm:critical_connect} Let 
\begin{equation} \label{eq:def_fbox}
    \Fbox(R)\stackrel{\text{def.}}{=}R\wedge \frac{h_N}{K_0(\frac{R\vee h_N }{N})} \quad \text{(see \eqref{eq:K_0} for the function $K_0$).}
\end{equation}
    For all $N, R\geq 1$ such that ${R} \leq \Cr{C:range}N$, we have that
    \begin{equation}\label{eq:critical_one_arm}
        \Cl[c]{arm_lb} \big(g^h_N(0)\Fbox(R)\big)^{-\frac{1}{2}} \leq \Th(R) \leq \Cl{arm_ub} \big(g^h_N(0)\Fbox(R)\big)^{-\frac{1}{2}}.
    \end{equation}
    In particular, for all $N \geq 2$,
    \begin{equation}\label{critical_one_arm-macro}
 \begin{array}{cl}
 c\,  (\log N)^{-1/2} \leq \theta_{N}^h(N) \leq C\,  (\log N)^{-1/2},
    & \text{when }{ h_N \leq \log N},\\
      c \, h_N^{-1/2} \leq \theta_{N}^{h}(N) \leq C\,  h_N^{-1/2},
    & \text{when } {h_N \geq \log N}.
 \end{array}
  \end{equation}
\end{theorem}
We start by making a few comments about Theorem~\ref{thm:critical_connect}; see Remark~\ref{R:final} for more.

\begin{remark}
\begin{enumerate}[label={\arabic*)}]
\item Since \eqref{eq:def_fbox} implies that $\Fbox(N)= c h_N$  (recall that $1 \leq h_N \leq N$), the bounds in \eqref{critical_one_arm-macro} follow immediately by combining \eqref{eq:critical_one_arm} and the on-diagonal behaviour of the Green's function from~\eqref{eq:variance-GFF} (see also Theorem~\ref{thm:green_main}). In the extremal case
$h_N=N$, corresponding to three spatial dimensions (cf.~the discussion following \eqref{eq:K_0}),
the above results match those obtained in \cite{drewitz_critical_2023, drewitz_critical_2024}, and separately in \cite{cai_one-arm_2024}, by which the decay is $N^{-\frac{d-2}{2}}$ on $\Z^d$ for $d=3$. When $h_N=1$, so that $\Slab$ is isomorphic to $\Z^2$ (cf.~\eqref{def:slab}), recent results of \cite{arXiv:2303.03782} yield $(\log N)^{-\frac12 + o(1)}$-behaviour for the corresponding annulus crossing probability by a critical two-dimensional Brownian loop soup cluster \cite{MR2045953}. Thus \eqref{critical_one_arm-macro} for $h_N=1$ corresponds to a strengthened version of \cite{MR2045953} on the cables, with up-to-constant bounds; see also \cite{gao2024percolationdiscretegffdimension,bi2025armeventscriticalplanar} regarding  events involving multiple arms, and \cite{MR2979861,MR3941462,MR4091511} for related CLE results. One could also generalize \eqref{eq:critical_one_arm} in replacing $0$ by $\partial B_r$ for $r < R$ by incorporating ideas from \cite{cai_one-arm_2024} though we won't pursue that here.
For generic choice of $h_N$ satisfying $ \log N \ll h_N \ll N$ as $N \to \infty$, \eqref{eq:def_fbox}-\eqref{critical_one_arm-macro} yield a host of new characteristic one-arm decays, corresponding to distinct long-range universality classes. 

\item The ``plateau'' regime, in which the one-arm probability no longer depends on $h_N$, corresponds to the regime in which $g_N(0)$ delocalizes; cf.~\eqref{eq:variance-GFF} and Fig.~\ref{F:plateau}. Perhaps somewhat surprisingly, the tail behaviour of the one-arm probability does not depend on the strength of the delocalization: for instance the variance $g\coloneqq g_N(0)$ of $\varphi_0$ is of order $(\log N)^\varepsilon$ for $\varepsilon \in (0,1)$ when $h_N= (\log N)^{1-\varepsilon} (\ll \log N)$, of order $(\log \log N)^{\alpha}$ when $h_N= \log N/(\log \log N)^{\alpha}$ for all $\alpha > 0 $ etc.~but this has no effect on connection probabilities. This is the case even though the capacity of the cluster of a point, which has an explicit law, identified in \cite[(3.8)]{drewitz_cluster_2022} (see also \eqref{eq:arctan_calculus} and \eqref{eq:caplaw} below), actually depends on the field precisely through $g$; see Remark~\ref{R:plateau} for more on this.

\end{enumerate}\end{remark}

\begin{figure}[ht] \centering
\includegraphics[width=0.6\textwidth]{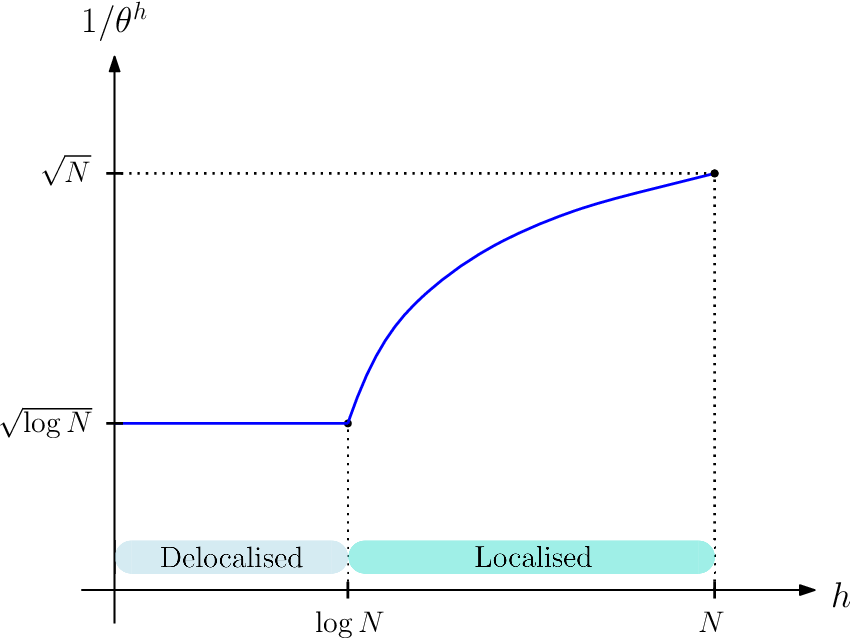} 
\caption{The map $h \mapsto \tfrac{1}{\theta^h}$ with $\theta^h\equiv \theta^h_N(N)$ in Theorem~\ref{thm:critical_connect}.}
\label{F:plateau}
\end{figure}

The proof of Theorem~\ref{thm:critical_connect} is split into two parts.
The lower bound in \eqref{eq:def_fbox} is proved at the end of Section~\ref{sec:lower_bounds}, the upper bounds in Section~\ref{sec:percolation}; see Theorem~\ref{prop:one_arm} for a more general upper bound. The proof does not distinguish between various regimes for $h$. It turns out that ideas from~\cite{drewitz_critical_2023, drewitz_critical_2024} (initially developed in the context of transient graphs with polynomial decay of $g_N$), when suitably extended to account for effects due to recurrence, are sufficiently robust to account for all cases $h, R, N$, including effectively two-dimensional regimes. 

%This is not completely straightforward. For example

Our approach is completely agnostic to the regime of parameters considered. In particular, it does not rely on any planar features available in (near-) two-dimensional regimes (including, among others, access to information about scaling limits of various related objects, whether or not loops ``surround'' a given set, duality considerations etc.). These tools either invariably fail or are at present out of reach when the ``effective dimension'' at scale $R$ becomes larger. Instead, we make use of so-called \emph{obstacles}, which are random and hard to avoid for the cluster cf.~\cite{drewitz_arm_2023, drewitz_critical_2024}; cf.~also \cite{RI-III,RI-I} in the context of sharpness results, where related ideas also appear. Obstacles play an important role in helping the cluster achieve certain features -- for instance, reach a certain capacity. 

Our arguments require a sufficiently fine understanding of the walk $X$ on $\Slab$, which occupies a significant proportion of this article.
 A key question is to understand precisely enough the extent to which two- vs. three-dimensional effects prevail for typical trajectories of $X$ at a given scale $R$. One basic observation rendering the analysis feasible is that this feature can be conveniently encoded \emph{topologically}, in terms of scenarios where trajectories of $X$ until the relevant time scale do or do not wrap around the torus, including possibly many times. This induces a decomposition of the field as
 \begin{equation}\label{eq:intro-decomp}
 \varphi\ \text{``$=$''} \ \varphi^{2d}+\varphi^{3d}
 \end{equation} 
 (cf.~\eqref{eq:projection} below for a precise statement) into corresponding orthogonal, i.e.~independent, components. The dichotomy stemming from the decomposition \eqref{eq:intro-decomp} is a recurrent theme, see for instance Propositions~\ref{thm:greens_function_3d} and~\ref{thm:greens_function_bound}, the splitting into two energy forms  $\mathscr{E}^i(\mu)$, $i=2,3$ in \eqref{eq:line_cap_two_prob},
etc.

\bigskip

One can in fact get more precise information on the function $\theta_N^h$ from \eqref{eq:connectivity_function} via the decomposition~\eqref{eq:intro-decomp} when the slab is sufficiently ``thin". To be more precise,
one should think of the slab at scale $R$ to be ``effectively two-dimensional" when $R\gg\frac{h_N\log R }{\log(N/(R\vee h_N))}$ (see Remark~\ref{rem:arm_condition},\ref{R:example_r} just below for examples of $R,h$ that satisfy this condition). Intuitively, this is the scale where the connection is distant enough so that the height $h_N$ is felt. In this regime, we are able to obtain the precise asymptotics of the function $\theta_N^h$ as $N\to\infty$. Let 
\begin{equation}\label{eq:f-infty}
    f_\infty(s)= \frac{1}{ \pi} \arctan\left[\frac{1}{\sqrt{s-1}}\right], \quad s>1.
\end{equation}
The $\arctan$ appearing in \eqref{eq:f-infty} is no stranger to the present percolation model; it describes for instance the tail of the capacity of a cluster, which was identified in \cite{drewitz_critical_2023, drewitz_cluster_2022} as having an explicit law on a large class of graphs. This is also the reason for its occurrence in the context of the following result.

\begin{theorem}\label{thm:arm_asymp}
If $N,R,h\geq1$ is such that $\lim_NR=\infty$ and $N\gg R\gg\frac{h_N\log R}{\log(N/(R\vee h_N))}$,
    \begin{equation}\label{eq:arm_main}
        \lim_N \frac{\theta_N^h(R) }{f_\infty(s_*)}=1,  
    \end{equation}
    where $ s_*=s_*(N,R,h) \stackrel{\text{def.}}{=}\frac{\pi}{3} \frac{g_N(0)h_N}{\log(N/(R\vee h_N))}$. % \PF{check $s>1$ if $N$ large enough}

\end{theorem}
We refer to Theorem~\ref{prop:arm_general} for a more general version of Theorem \ref{thm:arm_asymp}. We now highlight a few consequences of Theorem \ref{thm:arm_asymp}. In view of the elementary scaling $\arctan(x)\sim x$ as $x\to0$, the following corollary is a direct consequence of Theorem \ref{thm:arm_asymp} and \eqref{eq:variance-GFF}.

\begin{corollary} \label{cor:arm_const}
    Under the assumption of Theorem \ref{thm:arm_asymp}, if one further assumes $s_*\gg 1$ (as $N\to\infty$), then \begin{equation} \label{eq:cor_const}
        \lim_N \sqrt{\frac{g_N(0)h_N}{\log(N/(R\vee h_N))}} \theta^h_N(R) = \sqrt{\frac{3}{\pi^3}}.
    \end{equation}
\end{corollary}
In particular, \eqref{eq:cor_const} entails the following: if $\lambda_{g^2} \stackrel{\text{def.}}{=} \lim_N\frac{\log(N/h_N)}{h_N}\in [0,\infty]$ exists (the slightly enigmatic notation for this limit will become clear in Section~\ref{sec:green-slab}; the subscript $g^2$ refers to the relevant contribution to the Green's function here), then (cf.~\eqref{eq:variance-GFF} regarding $\Cr{d3_green_on_diag}$)
\begin{equation*}
\begin{array}{r@{}ll} 
\lim_N\sqrt{\frac{\log N}{\log(N/R)}}\,\theta^h_N(R) &=\frac{1}{\pi},
& \text{when }\lambda_{g^2}=\infty,\\[1em]
\lim_N\sqrt{\frac{h_N}{\log(N/(R\vee h_N))}}\,\theta^h_N(R)
&=\frac{1}{\pi}\sqrt{\frac{3}{\Cr{d3_green_on_diag}\pi}}(\lambda_{g^2}\frac{3}{\pi \Cr{d3_green_on_diag}}+1)^{-\frac{1}{2}},
& \text{when } \lambda_{g^2}\in[0,\infty).
\end{array}
\end{equation*}
\begin{remark} \label{rem:arm_condition}
\begin{enumerate}[label={\arabic*)}]
\item  If the height $h_N$ is sufficiently large, it is plausible that the slab can never be ``effectively two-dimensional" no matter the choice of $R$. Indeed, if $h_N\gtrsim \frac{N}{\log N} $, the condition $R\gg\frac{h_N\log R }{\log(N/(R\vee h_N))}$ cannot be met for any $R\leq \Cr{C:range}N$.

\item\label{R:_costly} The condition $\frac{g_N(0)h_N}{\log(N/(R\vee h_N))}\gg 1$ (from $s_* \gg 1$) appearing in Corollary \ref{cor:arm_const} corresponds to the regime where the scale $R$ is large enough so that the connection is actually costly. Conversely, if $\frac{g_N(0)h_N}{\log(N/(R\vee h_N))}\lesssim 1$, then $\liminf_N \theta^h_N(R)\geq c.$
A useful picture is that for $R\gg 1$, the connection is always expensive whenever the height $h_N$ is large, but this need not be true when $h_N$ is small.
In fact, $\frac{g_N(0)h_N}{\log(N/(R\vee h_N))}\gg 1$ always holds if $h_N\gg \log N $, whereas when $h_N\lesssim \log N $ one needs $R\gtrsim N^{1-o(1)}$ for the condition to be true. We return to this below (cf.~\eqref{eq:R_c}).

\item\label{R:example_r} We now give some explicit choices of $R$ that meet the conditions of Theorem \ref{thm:arm_asymp} and Corollary \ref{cor:arm_const} given an $h_N$. The intuition is that when $h_N\lesssim \log N$ one needs $R\gtrsim N^{1-o(1)}$ as discussed in \ref{R:_costly} and when $\frac{N}{\log N }\gg h_N\gg \log N $, one can pick $R$ to be larger than $h_N$ by a factor of $\log h_N $ for the slab to be ``effectively two-dimensional". Let $\alpha>1,\beta \in(0,1)$, the following pairs of $h_N,R$ satisfy the conditions of Theorem \ref{thm:arm_asymp} and Corollary \ref{cor:arm_const}:
\begin{equation*}
 \begin{array}{ll}
h_N\lesssim \log N , & R=N/\log N \\
h_N= \log(N)^\alpha, & R=\log(N)^\alpha\log\log N \\
h_N= N^\beta, & R=N^\beta\log N.
 \end{array}
 \end{equation*}
\end{enumerate}
\end{remark}

When the field is delocalized, we witness an intrinsic scale in the recurrent regime,
\begin{equation} \label{eq:R_c}
    R_c(s)=N^{1-\frac{1}{s}} \text{ , }s>1,
\end{equation}
along which $\theta_N^h(R_c(s))$ converges, as $N\to\infty$ to a function of $s$ only; moreover the scaling limit transitions from being order one to having polynomial decay as $s$ varies. In the special case $h_N=1$, results matching those displayed below are found in \cite[Theorem 1.3]{arXiv:2303.03782} for the corresponding continuum limit (to ease comparison observe that $f_{\infty}(s) = \frac1{2\pi}\int_{s-1}^\infty \frac1{\sqrt{t}(t+1)} dt$); noteworthily the precise limiting asymptotics are actually already present at the discrete level. 
\begin{corollary}[Intrinsic scale in the recurrent regime] \label{C:R-c}
  Under the assumption of Theorem \ref{thm:arm_asymp}, if one further assumes $\lim_N\frac{\log N }{h_N}=\infty$ and takes $R=R_c(s)$ as in \eqref{eq:R_c} for $s>1$, then
  \begin{equation} \label{eq:arm_s}  \lim_{s\uparrow\infty}\lim_N\sqrt{s}\,\theta^h_N(R) =\frac{ 1\black}{\pi} \quad\text{and}\quad
  \lim_{s\downarrow1}\lim_N\sqrt{s}\,\theta^h_N(R) =\frac{1}{2}.
    \end{equation}
\end{corollary}
\begin{proof}
By \eqref{eq:arm_main} and \eqref{eq:variance-GFF}
\begin{equation}\label{eq:asymp-2d-easy}
    \theta_N^h(R_c(s))\sim \frac{1 \black}{\pi} \arctan\left[\left(\frac{\log(N)}{\log(N^{s^{-1}})}-1\right)^{-\frac{1}{2}}\right] \stackrel{\eqref{eq:f-infty}}{=}f_\infty(s),
\end{equation} 
from which \eqref{eq:arm_s} follows since $\arctan(x)\sim x$ as $x\to 0$ and $\arctan(x)\sim \frac{\pi}{2}$ as $x\to \infty$.
\end{proof}
%In fact, in \cite{arXiv:2303.03782}, the bounds at scales $R_c(s)$ are used as a stepping stone to derive the bounds at macroscopic scales (which is also at the root of the $o(1)$ term therein. In contrast, in the present article it is rather the opposite, i.e.~Corollary~\ref{C:R-c} (along with Theorem~\ref{thm:arm_asymp} and Corollary~\ref{eq:cor_const}) follows essentially by strengthening the arguments underlying Theorem~\ref{thm:critical_connect} and make the constants $\Cr{arm_lb}, \Cr{arm_ub}$ effective.

We now briefly describe how this article is organized. In Section~\ref{sec:green-slab} we derive effective bounds  on the Green's function $g_N(x)$ from \eqref{eq:green_estimate0intro} (Theorem~\ref{thm:green_main}), along with precise asymptotics at large scales. In Section~\ref{sec:cap-slab} we use these to obtain certain capacity estimates, which play an important role in the sequel. When combined with ideas from \cite{drewitz_critical_2023}, these results are already enough to draw various conclusions about $\theta_N^h(\cdot)$, gathered in Section~\ref{sec:lower_bounds}. They include all of Theorem~\ref{thm:arm_asymp} (see Theorem~\ref{prop:arm_general} for a generalisation) as well as the lower bound of Theorem~\ref{thm:critical_connect}. The outstanding upper bounds are proved in Section~\ref{sec:percolation}. They partly require a more refined understanding of the walk that involves killed estimates.
This is related to the ``obstacles'' mentioned above, that can be very rough and induce a killing on their boundary when explored. These finer results, which are typically more strongly felt the lower the dimension (this aspect requires some care) are gathered separately in Section~\ref{eq:sec:kill}.

\medskip

In the sequel, we write $c,c^\prime,C,C^\prime$ for strictly positive constants which may change from line to line and write $c_i,C_i$ for strictly positive constants which will remain fixed. These constants have no dependence unless we explicitly say so. 

\section{The Green's function $g_N$ on the slab} \label{sec:green-slab}

We start by formally introducing the random walk $X$ whose Green's function is given by \eqref{def:green}. Recall the $\Slab$ from \eqref{def:slab}.
 We introduce non-negative weights $\lambda_{x,x'}=\lambda_{x',x}$ for $x,x'\in \mathbb{S}_N$ with $\lambda_{x,x'}= \tfrac16$ for $x=(y,z)$ and $x'=(y',z')$ if and only if either $|y-y'|=1$ or $z' = z \pm 1 \mod h_N$, where $|\cdot|$ denotes the usual Euclidean distance, and $\lambda_{x,x'}=0$ otherwise. We further introduce the killing measure $\kappa$ on $\mathbb{S}_N$ with $\kappa_x=  N^{-2}$ for all $x\in \mathbb{S}_N$. These choices imply that
$\lambda_{x}=\sum_{x^\prime} \lambda_{x,x^\prime}+\kappa_{x}=1 +N^{-2}$ for all $x\in \mathbb{S}_N$. We refer to as \emph{slab} the weighted graph
$(\mathbb{S}_N,\lambda,\kappa).$ Its edges are the pairs $\{x,x'\}$ for which $\lambda_{x,x'}\neq 0$; they correspond to the natural product graph structure on $\mathbb{S}_N$.  

The random walk $X=(X_t)_{t \geq 0}$ is the continuous-time Markov chain on $\mathbb{S}_N\cup \{\Delta\}$, where $\Delta$ is an absorbing cemetery state, which jumps from $x$ to $y$ at rate $\lambda_{x,y}$, with $\lambda_{x,\Delta}= \kappa_x$. 
We denote by $P_x, x \in \mathbb{S}_N$, the canonical law of $X$ starting from $x$, and by $\dX=(\dX_n)_{n \geq 0}$ its discrete-time skeleton. The killing time $\tau=\tau_N\in (0,\infty)$ is such that $X_t\in \mathbb{S}_N$ for all $0\leq t < \tau$ and $X_t\in \Delta$ for all $t\geq \tau$.

 For reference, we write $g_{\mathbb{Z}^3}(x)= \int_0^\infty \overline{P}_0(\overline{X}_t=x)$, $x \in \mathbb{Z}^3$, for the Green's function of the simple random walk $\overline{X}_{\cdot}$ with unit jump rate and canonical law $\overline{P}_x$ when $\overline{X}_0=x$. One classically knows that
 \begin{equation}\label{eq:green3d-normal}
 g_{\mathbb{Z}^3}(x) \sim \frac{3}{2\pi |x|} , \quad \text{as }  |x| \to \infty
 \end{equation}
 (see \cite[Theorem 1.5.4]{lawler_intersections_2013} for a proof), where $|\cdot|$ refers to the Euclidean norm on $\Z^3$.  With a slight abuse of notation, if $x=(y,z) \in \mathbb{S}_N$ we set $g_{\mathbb{Z}^3}(x) \coloneqq g_{\mathbb{Z}^3}((y,\hat{z}))$ (see below \eqref{eq:K_0} regarding $\hat{z}$).

  Our main result of this section is the following estimate on the Green's function $g_N^h(x,y)\equiv g_N^h(x-y)$ from \eqref{def:green}. Item~{(i)} below yields matching upper and lower bounds (up to multiplicative constants) uniform in $N \geq 1$ and all $x \in \mathbb{S}_N$ of interest. Item~{(ii)} yields precise asymptotics (including pre-factors) in the limit as $N \to \infty$. The restriction to a macroscopic range parametrized by the (arbitrary) constant $\Cr{C:range} \in [1,\infty)$ below allows for uniform estimates in the sequel. Recall that $h: (0,\infty) \rightarrow \R$ is any function such that $1\leq h(t)\leq t$ for $t>0$, and $h_N= \floor{h(N)}$ below. 

\begin{theorem} \label{thm:green_main} For all $N\geq 1$ and $x=(y,z) \in \mathbb{S}_N$, the following hold:
 \begin{enumerate}[label={(\roman*)}]
\item  For all $\Cr{C:range} \geq 1$ and $N,x$ such that ${\|x\|}\leq \Cr{C:range} N$, 
  \begin{equation}  \label{eq:green_estimate}
    \frac{\Cl[c]{d3_lb_green}}{\|x\|\vee 1} +
    \frac{\Cl[c]{d2_lb_green}}{h_N}
    K_0\left(\frac{|y|\vee h_N}{N}\right)
    \leq g_N(x) 
    \leq 
    \frac{\Cl{d3_ub_green}}{\|x\|\vee 1} +
     \frac{\Cl{d2_ub_green}}{h_N} K_0\left(\frac{|y|\vee h_N}{N}\right),
  \end{equation}
with $K_0(\cdot)$ as in \eqref{eq:K_0}.
\item 
If $\lim_N\frac{h_N}{N}=0$, $\lim_N \|x\|\in[0,\infty]$ and  $\limsup_N\frac{\|x\|}{N}<\infty$, then as $N\to\infty$, 
 \begin{equation} \label{eq:green_asymp_macx}
     g_N(x) \sim  \frac{3}{\pi} \frac{1}{h_N} K_0\left(\sqrt{6}\frac{|y|\vee h_N}{N}\right) + g_{\Z^3}(x)  \exp\bigg(-\sqrt{6}\frac{\|x\|}{N}\bigg) .    
 \end{equation}
\end{enumerate}
\end{theorem}

We first make a few observations about the previous theorem. We refer to \eqref{eq:g3-asymp-unif} and \eqref{eq:g2-asymp-unif} (in combination with \eqref{eq:projection}) below for stronger forms of \eqref{eq:green_asymp_macx} yielding asymptotics uniform in $x$ in appropriate regimes. We further refer to Remark~\ref{R:green-additional} at the end of this section for more comments about Theorem~\ref{thm:green_main}, in particular, concerning large-$N$ asymptotics that complement the regime covered by \eqref{eq:green_asymp_macx}, i.e.~when $h_N$ is asymptotically of order $N$. 

\begin{remark}\label{R:green}
\begin{enumerate}[label={\arabic*)}]
\item Item (i) above simplifies under the additional assumption {$\frac{h_N}{N}>c$.}
In this case,  for each $N\geq 1$ and $x=(y,z) \in \mathbb{S}_N$ satisfying $\frac{|y|}{N}\leq \Cr{C:range}$, one has that
  \begin{equation} \label{eq:green_edbound_hnn}
    {c}h_N^{-1} 
    \leq  h_N^{-1} K_0\textstyle\big(\frac{|y|\vee h_N}{N}\big) 
    \leq {C}h_N^{-1},
  \end{equation}
  which, in view of \eqref{eq:green_3d_bound} and \eqref{eq:green3d-normal}, yields a sub-leading contribution to $g_N(x)$ unless $\Vert x \Vert \asymp N$, and altogether yields a large-distance behavior (until macroscopic scale $N$) comparable to that of simple random walk on $\Z^3$.
  
\item \label{R:green-coro} The ``height effect" only switches on once the random walk has travelled a distance of order $\frac{h_N}{K_0(h_N/N)}$, and this naturally splits \eqref{eq:green_asymp_macx} into three regimes:
\begin{itemize}
\item ($2d$ regime). When $\displaystyle \lim_{N}   \textstyle  \frac{K_0(h_N/N)}{h_N  (\|x\|\vee 1) ^{-1}}=\infty$, one has
  \begin{equation*} 
        g_N(x) \sim \textstyle \frac{3}{\pi} \frac{1}{h_N} K_0\big(\sqrt{6}\tfrac{|y|\vee h_N}{N}\big).
      \end{equation*} 
\item(Intermediate~regime). When $\displaystyle \lim_{N}   \textstyle\frac{K_0(h_N/N)}{h_N (\|x\|\vee 1)^{-1}}=\lambda \in(0,\infty)$ and $ \ell_x \coloneqq \displaystyle \lim_{N}   \textstyle \|x\|  \in [0,\infty]$, one has
    \begin{equation*}
        g_N(x) \sim
        \begin{cases*}  %\Cr{d3_green_const}(\ell_x) 
        g_{\Z^3}(x)+\frac{3\lambda}{\pi}  \frac{1}{\|x\| \vee 1}, & $\ell_x<\infty$\\
         \big(\frac{3}{2\pi}\exp\big(-\sqrt{6}\frac{\|x\|}{N}\big) 
         +\frac{3\lambda}{\pi}\big)
         \frac{1}{\|x\|} , & $\ell_x=\infty$ 
        \end{cases*},
    \end{equation*}
    where we used \eqref{eq:green3d-normal}.
    \item ($3d$ regime). When $\displaystyle \lim_{N}   \textstyle\frac{K_0(h_N/N)}{h_N (\|x\|\vee 1)^{-1}}=0$ and $ \ell_x \coloneqq \displaystyle \lim_{N}   \textstyle \|x\|  \in [0,\infty]$, one has
    \begin{equation*}
        g_N(x) \sim
        \begin{cases*}
        %  \Cr{d3_green_const}(\ell_x) 
        g_{\Z^3}(x), & $\ell_x<\infty$ \\
           \frac{3}{2\pi \|x\|} \exp\big(-\sqrt{6}\frac{\|x\|}{N}\big) , & $\ell_x=\infty$. 
        \end{cases*}
    \end{equation*}
\end{itemize}

In particular, these asympstotics are readily seen to imply the large-$N$ behaviour for $g_N(0)$ asserted in \eqref{eq:variance-GFF}, with $\Cr{d3_green_on_diag} =  g_{\Z^3}(0)$: note to this effect that when $x=0$, owing to the fact that $K_0(t)\sim\log(\frac{1}{t})$ as $t\to 0^+$, the ratio defining the three above regimes simplifies to $ \textstyle\frac{\log N}{h_N}$.
\end{enumerate}
\end{remark}

The proof of Theorem \ref{thm:green_main} crucially relies on measuring the joint effects of the projections of $(X_t)_{t\geq0}$
onto ``horizontal'' (planar) and ``vertical'' directions. For this we introduce some further notation. We denote by $\Pzd{y}$ the canonical law of the simple random walk on $\Z^2$ with jump rate $\frac{2}{3}$ starting at $y\in\Z^2$ and denote by $\Pz{z}$ the canonical law of the simple random walk on $\Z$ with jump rate $\frac{1}{3}$  starting at $z\in\Z$. 
We write $(Y_t)_{t\geq 0}$ and $(Z_t)_{t\geq 0}$ for the canonical process respectively. Let $\pi_N: \mathbb{Z} \to \Z/h_N\Z$ denote the canonical projection. % If $\tau$ is an exponential variable of mean $N^{-2}$ independent of $(Y,Z)$,
In view of our setup above \eqref{def:green}, one readily obtains that for all $(y,z) \in \mathbb{S}_N$, with $\hat z \in \pi_N^{-1}(\{ z\}) $ (a point in $\mathbb{Z}$) as defined below \eqref{eq:K_0}, that
\begin{equation}
\label{eq:X-decomp}
\text{$(Y_{t }, \pi_N(Z_{t }))_{0 \leq t < \tau}$ under $\Pzd{y} \otimes \Pz{\hat{z}}$ has the same law as $(X_t)_{0 \leq t <  H_\Delta}$ under $P_{(y, z)}$,}
\end{equation} 
where $\tau$ in the previous display is an exponential variable of mean $N^{-2}$ independent of $(Y,Z)$ and $H_\Delta$ is the hitting time of $\Delta$ (cf.~above \eqref{eq:green3d-normal}). Applying \eqref{eq:X-decomp} to \eqref{def:green}, the Green's function $g_N$ naturally splits into a ``topologically trivial'' part and one that witnesses the (periodic boundary) at height $h_N$, as 
%Let $\Pd{x}$ be the canonical law of the simple random walk on $\Z^{3}$ with jump rate $1$ starting form $x \in \Z^3$. By projecting $\Z/h_N\Z$ onto $\Z$ and using Proposition 1.2.2 of \cite{lawler_random_2010}, we have
  \begin{align} \label{eq:projection}
    \begin{split}
    g_N(x)={\lambda}_x^{-1}(g_N^2(x) +g_N^3(x)),  \quad  x \in \mathbb{S}_N,
    \end{split}
  \end{align}
  where, abbreviating $P= \Pzd{0} \otimes \Pz{0}$, for all $x=(y,z) \in \mathbb{S}_N$, we have set
    \begin{align} \label{eq:def_g3}
      &g_N^3(x) \stackrel{\text{def.}}{=}  \int_0^\infty P \big((Y_t,Z_t)=(y,\hat{z})\big) \exp\{-{t}/{N^2}\} \,dt, \text{ and}\\
& \label{eq:def_g2}
      g_N^2(x) \stackrel{\text{def.}}{=}  \int_0^\infty \Pzd{0} (Y_t=y) \exp\{-{t}/{N^2}\}
      \sum_{k\in \mathbb{Z} \setminus\{0\}}  \Pz{0}(Z_t=\hat{z}+  kh_N)  
     % + \sum_{k=1}^\infty \Pz{0}(Z_t=\hat{z}-kh_N) \Big] 
     \,dt.
  \end{align}
We will soon see that the two terms in (\ref{eq:green_estimate}) come exactly from the decomposition in \eqref{eq:projection}, in particular the $g^3_N$ term corresponds to the polynomial contribution while the $g_N^2$ term corresponds to the logarithmic contribution. For ease of reading, some of the computations below have been gathered in Appendix~\ref{A:HK}. The presence of the factor ${\lambda}_x^{-1}$ in \eqref{eq:projection} is insignificant since our choice of normalization for the weights below \eqref{def:slab} imply that $\lambda_x = 1+ N^{-2} \sim 1$ as $N \to \infty$. To start with, it is straightforward to see that $g_N^3(x)$ in \eqref{eq:def_g3} is simply the Green's function for simple random walk on $\Z^3$ with an independent exponential killing. We have:

\begin{proposition}  \label{thm:greens_function_3d}
For all $N\geq 1$ and $x\in \mathbb{S}_N$, the following hold:
  \begin{enumerate}[label={(\roman*)}]
\item when $\|x\|\leq \Cr{C:range}N$,
  \begin{equation}  \label{eq:green_3d_bound}
    \frac{\Cr{d3_lb_green}}{\|x\|\vee 1} 
    \leq g_N^3(x) 
    \leq 
    \frac{\Cr{d3_ub_green}}{\|x\|\vee 1};
  \end{equation}
 \item  if {$\lim_{N}\|x\| \in [0,\infty]$ exists}, then as $N\to\infty$,
    \begin{equation} \label{eq:green_3d_asymp}
        g_N^3(x) \sim 
        g_{\Z^3}(x)\exp\big\{-\sqrt{6}{\|x\|}/{N}\big\} \stackrel{\textnormal{def.}}{=} \bar g_N^3(x).
    \end{equation}
   Moreover, for all $R \geq 1$,
    \begin{equation}\label{eq:g3-asymp-unif}
%\lim_R
 \sup_{\Vert x \Vert \geq R}\left| \frac{g_N^3(x)}{\bar g_N^3(x)}-1\right| \le C R^{-c}.
  %=0,
\end{equation}
\end{enumerate}
\end{proposition}

\begin{proof}
We start by recalling two useful identities involving Bessel functions. For $\nu\geq 0,\beta,\gamma>0,$ one has the following integral representation of $K_\nu$, the $\nu$-th order modified Bessel function of the second kind (see \cite[page 368, (3.471.9)]{gradshteyn_table_2014}),
\begin{equation} \label{eq:bessel_integral}
    K_\nu(2\sqrt{\beta\gamma})=\frac{1}{2}\left(\frac{\beta}{\gamma}\right)^{-\frac{\nu}{2}}\int_0^\infty s^{\nu-1}\exp{\{-\beta/s-\gamma s\}}\,ds.
\end{equation}
In the special case of $\nu=\frac{1}{2}$, one has the following explicit formulation (see \cite[page 925, (8.469.3)]{gradshteyn_table_2014}),
\begin{equation} \label{eq:bessel_half}
    K_{1/2}(x)=\sqrt{{\pi}/{2x}}\exp{\{-x\}}.
\end{equation}

For $x\in\mathbb S_N$, let $p(x,t)\coloneqq \big(\frac{3}{2\pi t}\big)^{3/2} \exp\big\{-3\frac{\|x\|^2}{2t}-\frac{t}{N^2}\big\}$. Using the substitution $s=\frac{3\|x\|^2}{2t}$, we have that
\begin{align} \label{eq:g3_hk}
\begin{split}
    \int_0^\infty p(x,t) \,dt 
    = &\frac{3^{1+1/4}}{2^{1/4}(\pi^3\|x\|N)^{1/2}} \times \frac{1}{2} (\frac{3\|x\|^2}{2N^2})^{-\frac{1}{4}}\int_0^\infty  s^{-1/2} \exp\left\{-s-\frac{3\|x\|^2}{2N^2s}\right\} \,ds \\  
    \stackrel{(\ref{eq:bessel_integral})}{=} &\frac{3^{1+1/4}}{2^{1/4}(\pi^3\|x\|N)^{1/2}} \times K_{1/2}\bigg(2\sqrt{\frac{3\|x\|^2}{2N^2}}\bigg) 
    \stackrel{(\ref{eq:bessel_half})}{=} \frac{3}{2\pi \|x\|} \exp\left\{-\sqrt{6}\frac{\|x\|}{N}\right\} .
\end{split}
\end{align}
Hence in order to deduce \eqref{eq:green_3d_bound} for $x \neq 0$ and \eqref{eq:green_3d_asymp} in the case $\lim_{N}\|x\|=\infty$ \black 
at once (note in the latter case that $\tfrac{3}{2\pi \|x\|}$ corresponds to the asymptotics of $g_{\Z^3}(x)$ as $\|x\| \to \infty$ on account of \eqref{eq:green3d-normal}), it suffices to show there exists $C \in (0,\infty)$ such that, for all $N\geq 1$ and $x=(y,z) \in \mathbb{S}_N \setminus \{ 0\}$, abbreviating $q(x,t) \coloneqq P \big((Y_t,Z_t)=(y,\hat{z})\big) e^{-t/N^2}$,
\begin{equation} \label{eq:g3_claim}
  \int_0^\infty \big\lvert
   p(x,t) 
  - q(x,t)
  \big\rvert  \,dt \leq 
  \frac{C}{\|x\|^3},
\end{equation}
which follows using (\ref{eq:lclt_approx_y}) with the choices $d=3,r=1,M=0$ upon noting that $(Y,Z)$ introduced above \eqref{eq:X-decomp} has the law of a unit rate simple random walk on $\Z^3$ started at the origin (this law is $P_0^{3,1}$ in the language of Appendix~\ref{A:HK}). As for \eqref{eq:green_3d_asymp} in the case $\lim_{N}\|x\| <\infty$
, one simply observes from \eqref{eq:def_g3} using {dominated convergence} that $g_N^3(x) \to g_{\mathbb{Z}^3}(x)$ as $N \to \infty$ (see below \eqref{eq:green3d-normal} for notation). 

It remains to explain \eqref{eq:green_3d_bound} when $x=0$. The lower bound is immediate by focusing on $t\in [0,1]$ in \eqref{eq:def_g3} and deriving a straightforward lower bound, uniform in $N$, on the probability that the walk hasn't jumped or been killed by time $t=1$. For the upper bound, Proposition~\ref{prop:lclt_approx_y} (see below (\ref{eq:lclt_approx_y})) further yields that $ \int_1^\infty \lvert p(0,t) - q(0,t) \rvert dt \leq C$. The claim now follows by bounding $g_N^3(0) \leq C + \int_1^\infty q(0,t) dt$, applying this bound, and noting that $t^{-3/2}$ is integrable at infinity. The bound \eqref{eq:g3-asymp-unif} follows by inspection of the above proof of \eqref{eq:green_3d_asymp}  (see, in particular, \eqref{eq:g3_claim}) and \eqref{eq:green3d-normal}.\black
\end{proof}

We turn our focus to the more interesting term $g_N^2$ in \eqref{eq:projection}, which comprises trajectories with a non-trivial winding number (i.e.,~with $|k| \geq 1$ below). The following key lemma controls the contribution of the 
integral in $g_N^2$ when $t$ is ``not too large.'' Although we are ultimately interested in setting $h=h_N$ (cf.~\eqref{def:slab}), the following result is best stated in terms of the (scalar) height parameter $h \geq 1$.

\begin{lemma} \label{prop:green_error_term}
 For $ M\geq 2,h\geq 1$ and $(y,z)\in \mathbb{Z}^2 \times (\Z / h \Z)$, letting  
\begin{equation*}
    E_ M\stackrel{\text{def.}}{=} \int_0^{Mh^2} \sum_{k\in \mathbb{Z} \setminus\{ 0 \}}   P \big((Y_t,Z_t)=(y,\hat{z}+ k h)\big)  \,dt,
\end{equation*}
we have that
\begin{equation} \label{eq:en_prop}
  {E_M}\leq \Cl{en_error}\bigg(\frac{M}{h\vee |y|} { + \frac{1}{h^2\sqrt{h\vee |y|}}} +     M \black\bigg(\frac{M}{h}\vee 1\bigg)\exp\left\{-c\frac{h\vee |y|}{M}\right\}\bigg).
\end{equation}
\end{lemma}
\begin{proof}
Abbreviating $I_k= (y,\hat{z}+ kh)$, we will first show that for all $M,h$ and $(y,z)$ as above, 
\begin{align} \label{eq:green_error_tsmall}
\begin{split}
    \int_0^{{h\vee |y|}} &\sum_{k\in \mathbb{Z} \setminus\{ 0 \}}  P \big((Y_t,Z_t)=I_k\big) \,dt
    \leq C  M \black \left(\frac{M}{h}\vee 1\right)\exp\left\{-c\frac{h\vee |y|}{M}\right\}. 
\end{split}
\end{align}
%\WZ{before: $\frac{1}{\sqrt{2}}$ was just $c$}\PF{maybe $1/\sqrt{3}$? [That's definitely ok]}
For all $|\sigma|=1$ and $k \geq 1$, using the inequalities $|\hat{z}|\leq h/2$ and $\sqrt{a^2+D}\geq { \Cl[c]{C:lbE}}(a+b)$ {for $a,b,D>0$ and $\sqrt{D}>b$} and some $\Cr{C:lbE}>0$, one has that $\vert I_k \vert= \sqrt{|y|^2+|\hat{z}+\sigma k h|^2}\geq {\Cr{C:lbE}}(|y|+(k-1/2) h).$ Applying the short-time estimate (\ref{eq:exit_time}) for $t\leq %\|(y,\hat{z}+\sigma kh)
\Cr{C:lbE}^{-1}\vert I_k \vert M$ yields that for all $t \leq h\vee |y|$ (such $t$ satisfy $t\leq %\|(y,\hat{z}+\sigma kh)
\Cr{C:lbE}^{-1}\vert I_k \vert M$ since $M \geq { 2}$ and $\vert I_k \vert \geq { \Cr{C:lbE} (|y| \vee h)/2}$ by the previous lower bound), 
\begin{equation} \label{eq:error_transition_tsmall}
   \sum_{k = 1}^\infty P \big((Y_t,Z_t)=(y,\hat{z}+\sigma k h)\big) 
    \leq Ce^{-c|y|/M}  \sum_{k=1}^\infty \exp\left\{-\frac{c(k-1/2)h}{M}\right\}. 
\end{equation}
By considering separately the cases $\tfrac h M \leq 1$ and  $\tfrac h M > 1$, one finds that the sum in \eqref{eq:error_transition_tsmall} is bounded by $C \left(\frac{M}{h}\vee 1\right)\exp\left\{-\frac{ch}{2M}\right\}$, and
(\ref{eq:green_error_tsmall}) readily follows.

With \eqref{eq:green_error_tsmall} at hand, it suffices to bound 
 $E_M^\prime \coloneqq   \sum_{k\in \mathbb{Z} \setminus\{ 0 \}} \int_{h\vee |y|}^{M h^2}  P ((Y_t,Z_t)=I_k)  \,dt$ instead of $E_M$. By (\ref{eq:prop_lclt_main_hbig}) applied with $d=2$ and $r=1$, and in view of \eqref{eq:BM-kernel}, we can rewrite $E_M^\prime$  as
\begin{equation} \label{eq:en_step2_heat_kernel}
   \sum_{k\in \mathbb{Z} \setminus\{ 0 \}} \int_{h\vee |y|}^{M h^2}  
    \Big(\frac{3}{2\pi t}\Big)^{3/2} \exp\left\{-3\frac{|y|^2+|\hat{z}+ kh|^2}{2t}\right\}  \,dt
\end{equation}
with an error bounded by ${C}/({h^2\sqrt{h\vee |y|}})$, accounted for by the {second} term on the right of \eqref{eq:en_prop}. When $|y|\geq h$ (so $|y| \geq 1$ and in particular, the term $\exp\{-3\tfrac{|y|^2}{2t}\}$ below does not trivialize), we have
\begin{align*}
\begin{split}
    (\ref{eq:en_step2_heat_kernel}) 
    &\leq 
    \int_{0}^{Mh^2} \frac{C}{t^{3/2}}
    e^{-3\frac{|y|^2}{2t}} 
    \sum_{k=1}^{\infty}  e^{- \frac{3(k-1/2)^2 h^2}{2t}} \,dt \\
    &\leq \int_{0}^{\infty} \frac{C}{t^{3/2}}
    e^{-3\frac{|y|^2}{2t}} 
   \,dt\times  \sum_{k=1}^\infty  e^{- \frac{3(k-1/2)^2}{2M}}  
    \leq \frac{C}{|y|}\times \frac{1}{1-e^{-c/M}}\leq \frac{CM}{|y|},  
\end{split}
\end{align*}
which concludes the proof of when $|y|\geq h$. When $|y|<h$, using the substitution $s=\frac{3h^2}{2t}$ we get
\begin{align*}
\begin{split}
    (\ref{eq:en_step2_heat_kernel}) \leq 
    \int_{h}^{M h^2} \frac{
    C}{t^{3/2}}
    \sum_{k=1}^{\infty}e^{- \frac{3(k-1/2)^2 h^2}{2t}} \,dt 
    &\leq \frac{C}{h} \int_{\frac{3}{2M}}^\infty s^{-1/2} \sum_{k=1}^\infty e^{-(k-1/2)^2s}\,ds \\
    &\leq \frac{C'\sqrt{M}}{h}\times \sum_{k=1}^\infty e^{-\frac{3(k-1/2)^2}{2M}}\leq \frac{CM}{h},
\end{split}
\end{align*}
where the last line follows by applying Fubini, bounding $s^{-1/2}$ by its maximal value bound, performing the integral over $s$ and using a straightforward Riemann sum argument involving the estimate $\int_0^{\infty}  \exp\{-\frac{3(x-1/2)^2}{2M}\}  \,dx \leq C\sqrt{M}$. Overall, this yields \eqref{eq:en_prop}.
\end{proof}

Following is the counterpart to Proposition~\ref{thm:greens_function_3d} for $g_N^2$ defined in \eqref{eq:def_g2}. Recall $K_0$ from \eqref{eq:K_0}. 

\begin{proposition}\label{thm:greens_function_bound}
  For all $N\geq 1$ and $x=(y,z) \in \mathbb{S}_N$, the following hold:
    \begin{enumerate}[label={(\roman*)}]
  \item if $|y|\leq\Cr{C:range}N$, we have
  \begin{equation} \label{eq:green_bound}
    \frac{\Cr{d2_lb_green}}{h_N} K_0\left(   \black \frac{|y|\vee h_N}{N} \right)\leq g^2_N(x) \leq \frac{\Cr{d2_ub_green}}{h_N} K_0\left(   \black \frac{|y|\vee h_N}{N} \right);
  \end{equation}
   \item  if $\limsup_{N}\frac{h_N}{N}=0$ and $\limsup_{N}\frac{|y|}{N}<\infty$, then as $N\to\infty$,
   \begin{equation} \label{eq:green_asymptotics_2d}
    g^2_N(x) \sim  \frac{3}{\pi} \frac{1}{h_N} K_0\left(\sqrt{6}\,\frac{|y| \vee h_N}{N}\right)\stackrel{\textnormal{def.}}{=} \bar g_N^2(x).
  \end{equation}
  Moreover, for all $\varepsilon \in (0,1)$, and $N, h_N$ satisfying $ N \geq C(\varepsilon) $ and $h_N \leq c(\varepsilon) N$,
    \begin{equation}\label{eq:g2-asymp-unif}
%\lim_R
 \sup_{ \Vert x \Vert \leq \Cr{C:range}N }\left| \frac{g_N^2(x)}{\bar g_N^2(x)}-1\right| < \varepsilon.
  %=0,
\end{equation}
  \end{enumerate}
\end{proposition} 
\begin{proof} We start with some preliminary reduction steps.
Combining the estimate (\ref{eq:en_prop}) with $h=h_N$ (cf.~around \eqref{def:slab}) and (\ref{eq:prop_lclt_main_hbig}) applied with {$d=2$} and $r=1$, we obtain in view of \eqref{eq:BM-kernel} that for all $M\geq 2$, $N \geq 1$ and $x=(y,z)\in \mathbb{S}_N$ (all tacitly assumed in the sequel),
\begin{multline} \label{eq:green_func_main_term}
   - \frac{C }{h_N^3 \sqrt{M}} \leq g_N^2(x) -  \frac{1}{h_N}\int_{M h_N^2}^\infty  \frac{3}{2\pi t} 
   e^{-3\frac{|y|^2}{2t}-\frac{t}{N^2}} \left[\sum_{k\in \mathbb{Z} \setminus \{0\}} h_N (\frac{3}{2\pi t})^{1/2} e^{ - \frac{3|\hat{z}+ kh_N|^2}{2t}}\right] \,dt 
   \\
   \leq C \left(\frac{M^{ 2 }}{h_N\vee|y|} {+\frac{1}{h_N^2\sqrt{h_N\vee |y|}}} +\frac{1 }{h_N^3 \sqrt{M}}\right).
\end{multline}
Note that the series appearing inside the square brackets in (\ref{eq:green_func_main_term}) can be
thought of as approximating the area under the Gaussian density and is uniformly bounded as the variance term is bounded away from $0$. More specifically, we note that this series converge to $1$ as the variance increases. Indeed, for suitable $\Cl{C:gauss-approx} \in [1,\infty)$ {and $\Cl[c]{c:gauss-lb}\in(0,1)$}, we have that for $M\geq 2$, $h_N \geq 1$, and $t\geq M h_N^2$, 
\begin{equation} \label{eq:approx_sum_gaussian}
     0<\Big(1- \frac{\Cr{C:gauss-approx}}{\sqrt{M}}\Big){ \vee \Cr{c:gauss-lb}} \leq \sum_{ k \in \mathbb{Z} \setminus \{0\} \black} h_N (\frac{3}{2\pi t})^{1/2} \exp\left\{ - \frac{3(\hat{z}+ kh_N)^2}{2t}\right\} \leq 1;
\end{equation}
 the lower bound {by $1- \frac{\Cr{C:gauss-approx}}{\sqrt{M}}$} in \eqref{eq:approx_sum_gaussian} essentially arises from the approximation error to the Gaussian integral near the maximum, which is bounded by $Ct^{-1/2} \times h_N$, where the first factor bounds the density and $h_N$ is the interval length considered; the other intervals yield summable corrections (in $k$) of the same order. The uniform lower bound in \eqref{eq:approx_sum_gaussian} can be obtained from the observation that the series is larger than $P(Z\geq \sqrt{3/t}(|\hat{z}|+h_N))\geq P(Z\geq \frac{3}{2}\sqrt{3/M})\geq\Cr{c:gauss-lb}$, where $Z$ is a standard normal variable. 

We now focus on estimating the integral $ I(M) \equiv \int_{M h_N^2}^\infty  \frac{3}{2\pi t} 
   \exp\{-3\frac{|y|^2}{2t}-\frac{t}{N^2}\} \,dt$ appearing in \eqref{eq:green_func_main_term} when neglecting the term in square brackets. Rescaling by $N$, expressing the exponential $\exp\{-\frac{3 |y|^2/N^2}{2t}\}$ as an integral and applying Fubini yields that for all $M \geq 1$,
 \begin{equation}
 \label{eq:green_hn_ub-general}
I(M)=    \frac{3}{2\pi} \int_{(\frac{\sqrt{M}h_N}{N})^2}^\infty \frac{1}{ t}
    e^{-\frac{3 |y|^2/N^2}{2t}-t} \,dt 
    =  \frac{3}{2\pi} \int_{0}^\infty  e^{-s}
     \int_{(\frac{\sqrt{M}h_N}{N})^2\vee\frac{3}{2s}(\frac{|y|}{N})^2}^\infty\frac{e^{-t}}{ t}   \,dt \,ds. 
 \end{equation}
  We now proceed to show the desired bounds. Let $\Cl{g2_hN} \in [1,\infty)$ be a large constant, soon to be chosen suitably. We start with item (i), under the additional hypothesis that {$h_N>\tfrac{1}{\Cr{g2_hN}}N$,} whence the factor $K_0(\cdot)$ in \eqref{eq:green_bound} can effectively be neglected (cf.~\eqref{eq:green_edbound_hnn}). For this we pick $M=2$ and note that \eqref{eq:green_hn_ub-general} yields that
% \WZ{before:$\wedge \frac{1}{2}$}
\begin{align}  \label{eq:green_hn_big_ub}
\begin{split}
    I(2)
    &\leq  \frac{3}{2\pi} \int_0^\infty e^{-s}\int_{(\frac{\sqrt{2}h_N }{N})^2 \wedge \frac{1}{{ 4}}}^\infty
    \frac{e^{-t}}{t} \,dt \,ds \leq \frac{3}{\pi} \log\left(\frac{N}{\sqrt{2}h_N}\vee 2\right) + \frac{3}{2\pi} \leq C,
  \end{split}
\end{align}
where we also used 
\begin{equation} \label{eq:exp_inequality}
    \int_x^{\infty} \frac{e^{-t}}{ t}  \,dt 
    \leq
    \begin{cases*}
        \log(1/x) +1 & \text{if } $0<x\leq 1$\\
        1 & \text{if } $x> 1$.
    \end{cases*}
\end{equation}
For the lower bound, simply note since $\frac{h_N}{N} (\leq 1)$ and $\frac{|y|}{N} (\leq \Cr{C:range})$ are uniformly bounded in $N$, we have
\begin{equation*}
     I(2) \stackrel{\eqref{eq:green_hn_ub-general}}{\geq}  \frac{3}{2\pi} \int_{1}^\infty  e^{-s}  
  \int_{2\vee\frac{3}{2}\Cr{C:range}^2}^\infty\frac{e^{-t}}{ t}   \,dt \,ds  \geq c.
\end{equation*}
Together, \eqref{eq:green_func_main_term}, \eqref{eq:approx_sum_gaussian} and the fact that $c \leq   I(2) \leq C $ conclude the proof of \eqref{eq:green_bound} when
{$h_N>\tfrac{1}{\Cr{g2_hN}}N$.}
%$\limsup_{N}\frac{h_N}{N}>0$.

We now focus on the case that {$h_N\leq \tfrac{1}{\Cr{g2_hN}}N$ and $|y|\leq h_N$}
%$\limsup_{N}\frac{h_N}{N}=0 = \lim_{N\to\infty}\frac{h_N}{N}$
in \eqref{eq:green_bound}, and in doing so will derive bounds precise enough to deduce (\ref{eq:green_asymptotics_2d}) as well. %We first focus on the case $\lim_{N\to\infty}\frac{|y|}{h_N}<\infty$. Following the calculation in (\ref{eq:green_hn_ub}) We have
In view of \eqref{eq:green_hn_ub-general}, and using \eqref{eq:exp_inequality} again, we have for all $N, M \geq 1$ and $y \in \mathbb{Z}^2$ that
\begin{align} \label{eq:xi-M-UB}
  \begin{split}
  I(M)
    &\leq  \frac{3}{2\pi} \int_0^\infty e^{-s}\int_{(\frac{h_N }{N})^2}^\infty
    \frac{e^{-t}}{t} \,dt \,ds 
    \leq \frac{3}{\pi} \log\Big(\frac{N}{h_N}\Big) + \frac{3}{2\pi}.
  \end{split}
\end{align}
and for the lower bound that, whenever $\frac{\sqrt{M}h_N}{N} \leq 1$,
\begin{align} \label{eq:xi-M-LB}
  \begin{split}
   I(M) &\geq  \frac{3}{2\pi} \int_{\frac{3}{2}(\frac{|y|}{\sqrt{M}h_N})^2}^\infty e^{-s}\int_{(\frac{\sqrt{M}h_N}{N})^2}^{1}
    \frac{1-t}{t} \,dt \,ds \\
%    &\geq
%      \frac{3}{2\pi}\left(2\log\Big(\frac{N}{h_N}\Big)- \log(M) - 1\right)
%      - \Cl{C:gf-error-M}\frac{|y|^2}{Mh_N^2}\left(2\log\Big(\frac{N}{h_N}\Big)-\log(M) -1\right)\\
    & \geq
      \frac{3}{2\pi}\left(2\log\Big(\frac{N}{h_N}\Big)- \log(M) - 1\right)
      - \frac{\Cl{C:gf-error-M}}{M}\left(2\log\Big(\frac{N}{h_N}\Big)-
      \log(M) -1\right),
  \end{split}
\end{align}
where the last line follows from the fact that $  1- \int_{\alpha}^\infty e^{-s} \,ds \leq \alpha$ with $\alpha= \tfrac{3}{2}(\tfrac{|y|}{\sqrt{M}h_N})^2 \leq \tfrac{\Cr{C:gf-error-M}}{M}$. Now, picking $M= 2 \vee 4\pi \Cr{C:gf-error-M}/3$ and combining \eqref{eq:xi-M-UB}, \eqref{eq:xi-M-LB} with \eqref{eq:green_func_main_term} and \eqref{eq:approx_sum_gaussian} as before, the claim \eqref{eq:green_bound} follows under the assumptions that 
{$h_N\leq \frac{1}{\Cr{g2_hN}}N$}
%$\lim_{N}\frac{h_N}{N}=0$
and $|y| \leq h_N$, using the fact that $K_0(x)\sim \log(1/x)$ as $x\to 0$. 
{Furthermore, with a view towards \eqref{eq:green_asymptotics_2d} and \eqref{eq:g2-asymp-unif}, we also obtain, for any $\epsilon\in(0,1)$, by taking $M=M(\epsilon
)$ and $\Cr{g2_hN}=\Cr{g2_hN}(\varepsilon)$ large enough,}
%Furthermore, under the assumption that $\lim_{N}\frac{h_N}{N}=0$ and $\lim_N \frac{|y|}{h_N}$ exists and belongs to $[0,\infty)$, 
that
\begin{equation} \label{eq:g^2-sharpbd}
    \frac{3}{\pi} \frac{(1-\epsilon)}{h_N} K_0\left(\frac{ h_N}{N}\right) \leq g^2_N(x)\leq \frac{3}{\pi} \frac{(1+\epsilon)}{h_N} K_0\left(\frac{ h_N}{N}\right),
\end{equation}
whenever $|y| \leq h_N$ and $h_N\leq \frac{1}{\Cr{g2_hN}(\epsilon)}N$. Note that the latter of the two conditions will hold for large enough $N$ in the context of item (ii) since it is a requirement for \eqref{eq:g2-asymp-unif} and since $h_N=o(N)$ as $N \to \infty$ by assumption in the context of \eqref{eq:green_asymptotics_2d}.

To deal with the remaining case {$h_N\leq \frac{1}{\Cr{g2_hN}}N$ and $|y|> h_N$ in \eqref{eq:green_bound}}  (and in fact also in \eqref{eq:green_asymptotics_2d}-\eqref{eq:g2-asymp-unif} as we shall explain momentarily), we proceed as follows. For the upper bounds, using (\ref{eq:bessel_integral}) with $\nu=0$ we find that for all $N, M \geq 1$ and $y \in \mathbb{Z}^2$,
\begin{align} \label{eq:green_rn_ub}
  \begin{split}
    I(M)= \int_{Mh_N^2}^\infty \frac{3}{2\pi t}
    e^{-\frac{3 |y|^2}{2t}-\frac{t}{N^2}} \,dt \leq \int_{0}^\infty \frac{3}{2\pi t}
    e^{-\frac{3 |y|^2}{2t}-\frac{t}{N^2}} \,dt
    = \frac{3}{\pi} K_0\left(\sqrt{6}\frac{|y|}{N}\right),
  \end{split}
\end{align}
which gradually improves over \eqref{eq:xi-M-UB} whenever {$\tfrac{N}{|y|}= \tfrac{N}{h_N} \cdot \tfrac{h_N}{|y|} \leq \frac{N}{h_N}$.} In particular, using \eqref{eq:green_rn_ub} with say $M=2$ completes the verification of the upper bound in \eqref{eq:green_bound} upon combining the resulting estimate on $I(2)$ with (\ref{eq:green_func_main_term}) and (\ref{eq:approx_sum_gaussian}). Moreover, the right-hand side in \eqref{eq:green_rn_ub} is also a lower bound for $I(M)$ for any $N, M \geq 1$ and $y \in \mathbb{Z}^2$ up to an error
\begin{align*}
  \begin{split}
    \int_0^{Mh_N^2} \frac{3}{2\pi t}
    e^{-\frac{3 |y|^2}{2t}} \,dt 
    \leq \frac{3}{2\pi}  \int_{\frac{3|y|^2}{2Mh_N^2}}^{\infty} \frac{1}{s}
    e^{-s}\,ds
    \leq C\frac{Mh_N^2}{|y|^2}\exp\left\{-\frac{3|y|^2}{2Mh_N^2}\right\}.
  \end{split}
\end{align*}
To deduce the outstanding lower bound in \eqref{eq:green_bound} (and soon in \eqref{eq:green_asymptotics_2d}-\eqref{eq:g2-asymp-unif}) under the assumption {$h_N\leq \frac{1}{\Cr{g2_hN}}N$} and $|y| > h_N$, one distinguishes two cases as follows: for $(h_N <) |y| \leq \sqrt{h_N N} $, 
the error for any $M\geq 2$ is bounded by $C(M)$ uniformly in such $y$ and $N \geq 1$ whereas the right-hand side of \eqref{eq:green_rn_ub} is bounded from below by $ c (1\vee \log (N/h_N))$, which in particular can be made arbitrarily large by taking $\Cr{g2_hN}=\Cr{g2_hN}(M)$ large.
\black 
If on the other hand $ |y| > \sqrt{h_N N}$, then the right-hand side of \eqref{eq:green_rn_ub} is $ \geq c$ but the error is bounded by ${ C(M)} e^{-cN/h_N}$, {which can be made arbitrarily close to $0$ by taking $\Cr{g2_hN}{=\Cr{g2_hN}(M)}$ large.}
Putting things together and combining the above calculations with (\ref{eq:green_func_main_term}) and (\ref{eq:approx_sum_gaussian}) yields the desired lower bound and with it completes the proof of \eqref{eq:green_bound}. The proof of the corresponding lower bound for \eqref{eq:g2-asymp-unif}, i.e.~the analogue of the lower bound in \eqref{eq:g^2-sharpbd} with $\sqrt{6}|y|/N$ in the argument of $K_0$, valid for {$h_N\leq \frac{1}{\Cr{g2_hN}(\varepsilon)}N$} and $|y| > h_N$, is obtained by the same reasoning.

As to the outstanding upper bound akin to \eqref{eq:g2-asymp-unif} when $\Cr{g2_hN}(\varepsilon) h_N\leq N$ and $|y| > h_N$, one still distinguishes the above two cases: for $|y| \leq \sqrt{h_N N} $, 
the error from \eqref{eq:green_func_main_term} is bounded by $C(M)h_N^{-1}$ uniformly in such $y$ and $N \geq 1$ whereas by \eqref{eq:green_rn_ub} $\frac{I(M)}{h_N}$ is bounded from below by $ h_N^{-1} (1\vee \log (N/h_N)-C(M))$, which in particular can be made arbitrarily large by taking $\Cr{g2_hN}=\Cr{g2_hN}(M, \varepsilon)$ large.
If on the other hand $ |y| > \sqrt{h_N N}$, then by \eqref{eq:green_rn_ub}, $ \frac{I(M)}{h_N}\geq \tfrac{c}{ h_N}$ but the error from \eqref{eq:green_func_main_term} is bounded by $ C h_N^{-1}(\tfrac{M^2\sqrt{h_N}}{\sqrt{N}}{+\frac{1}{N^{1/4}}} + \tfrac{1}{\sqrt{M}})$, and the expression inside the brackets \black can be made arbitrarily close to $0$ by taking $M=M(\varepsilon)$ and $N \geq C(\varepsilon)$ large. All in all, we thus obtain the analogue of \eqref{eq:g^2-sharpbd} with $K_0(\tfrac{\sqrt{6}|y|}{N})$ replacing $K_0(\tfrac{ h_N}{N})$ in the regime $|y| > h_N$ and whenever $\Cr{g2_hN}(\varepsilon) h_N\leq N$ and $N \geq C(\varepsilon)$, from which \eqref{eq:g2-asymp-unif} follows; \eqref{eq:green_asymptotics_2d} is a direct consequence of \eqref{eq:g2-asymp-unif}. \black
\end{proof}

\begin{proof}[Proof of Theorem \ref{thm:green_main}]
   Item (i) of Theorem \ref{thm:green_main} is an immediate consequence of the decomposition $(1+N^{-2}) g_N=g_N^2+ g_N^3$  in \eqref{eq:projection}, Proposition \ref{thm:greens_function_3d},(i) and Proposition \ref{thm:greens_function_bound},(i). As to item~(ii), \eqref{eq:green_asymp_macx} follows immediately from the decomposition \eqref{eq:projection}, together with Proposition \ref{thm:greens_function_3d},(ii) and Proposition \ref{thm:greens_function_bound},(ii). 
   \end{proof}

\begin{remark}\label{R:green-additional}
 \begin{enumerate}[label={\arabic*)}]
 \item{(Asymptotics beyond \eqref{eq:green_asymp_macx}).} \label{R:green-2}  Although we won't need this in the sequel, we record for future reference the following complementary asymptotics to those stated in Theorem~\ref{thm:green_main},~(ii)  in the regime $c_h\stackrel{\text{def.}}{=} \lim_{N}\frac{h_N}{N}\in (0,1]$.
 
   Suppose $\lim_N \frac{\|x\|}{N} =0$, then
\begin{equation}\label{eq:green_asymp_hnn_micro}
       g_N(x)\sim g_{\Z^3}(x).
   \end{equation}
   This is because if in addition $\lim_N \|x\| < \infty$, Proposition \ref{thm:greens_function_bound},(i) implies that $g_N^2({x}) \leq Ch_N^{-1}=o(1)$ as $N \to \infty$, whereas $g_N^3(x) \to g_{\Z^3}(x)>0$, yielding \eqref{eq:green_asymp_hnn_micro} in this case. If instead $\lim_N \|x\| = \infty$, then $g_N^2(x) = o(\tfrac1{\|x\|})$ by combining Proposition \ref{thm:greens_function_bound},(i) with our standing assumptions, which imply that $h_N^{-1} \leq CN^{-1}= o(\tfrac1{\|x\|})$. Hence, \eqref{eq:green_asymp_hnn_micro} follows on account of \eqref{eq:green3d-normal}.
\black

 If $\lim_{N}\frac{\|x\|}{N}\in(0,\infty)$, then 
\begin{equation} \label{eq:green_asymp_hnn_main}
    g_N(x)\sim   \frac{3}{\pi} \frac{1}{h_N}K(c_h,\ell_y,\ell_z) + \frac{3}{2\pi \|x\|}\exp\bigg\{-\sqrt{6}\frac{\|x\|}{N}\bigg\} ,
\end{equation}
where $\ell_y\stackrel{\text{def.}}{=} \lim_{N\to\infty}\frac{|y|}{N}$, $\ell_z\stackrel{\text{def.}}{=}\lim_{N\to\infty}\frac{|\hat{z}|}{N}$ (which are implicitly assumed to exist) and
\begin{align*}
\begin{split}
    K(c_h,\ell_y,\ell_z)= \frac{c_h}{2}
    \sum_{k\in \mathbb{Z} \setminus \{0\}} \Big(\ell_y^2+(\ell_z+ c_h k)^2\Big)^{-1/2}\exp\Big\{-\sqrt{6}\Big(\ell_y^2+(\ell_z+ c_h k)^2\Big)^{1/2}\Big\}
    .
    \end{split}
    \end{align*}
The claim~\eqref{eq:green_asymp_hnn_main} is obtained by bounding $g_N^2$ in \eqref{eq:projection} in a similar manner as in the proof of  Proposition \ref{thm:greens_function_bound}. However, one keeps all terms in the sum over $k$ in the analogue of \eqref{eq:green_func_main_term}, which leads to an effective integral $  \sum_{k\in \mathbb{Z} \setminus \{0\}} \int_{Mh_N^2{/N^2}}^\infty 
    (\frac{3}{2\pi t})^{3/2} \exp\big\{-\frac{3}{2t}(\frac{|y|^2}{N^2}+(\frac{|\hat{z}|}{N}+ \frac{h_N}{N} k)^2) - t\big\}  dt$ that plays the role of $I(M)$, and leads to $g_N^2(x) \sim \frac{3}{\pi} \frac{1}{h_N}K(c_h,\ell_y,\ell_z)$, yielding \eqref{eq:green_asymp_hnn_main}.
    
Although the quantity $K(c_h,\ell_y,\ell_z)$ cannot be expected to have a closed form in general, one can compute it directly in the case of $\ell_y=\ell_z=0$. In this case one finds that  $K(c_h,0,0)=-\log(1-\exp\{-\sqrt{6}c_h\}).$
This is consistent with the first term in (\ref{eq:green_asymp_macx})  in the sense that if ${\ell_y=\ell_z=0}$ and ${c_h=\lim_{N}\frac{h_N}{N}=0}$, then 
\begin{equation*}
  K(c_h,\ell_y,\ell_z)=  -\log\big(1-\exp\{-\sqrt{6}c_h\}\big) \sim \log({N}/{h_N}) \sim K_0\big(\sqrt{6}{h_N}/{N}\big).
\end{equation*}
In light of this and (\ref{eq:green_asymp_macx}), one might naively expect at least when $\ell_y>c_h$, one would get $K(c_h,\ell_y,\ell_z)=K_0(\sqrt{6}\ell_y)$, however this is not the case. The reason is that the first term of (\ref{eq:green_asymp_macx}) arises from a simple random walk in $\Z^2$, whereas $K(c_h,\ell_y,\ell_z)$ arises from a simple random walk in $\Z^3$. This is because when $h_N=o(N)$, one can afford to separate the cost of wrapping around the periodic direction (see (\ref{eq:en_prop}) and (\ref{eq:approx_sum_gaussian})) from the rest of the walk, which is a simple random walk in $\Z^2$ (see, for instance, (\ref{eq:green_rn_ub})). However, when $h_N\asymp N$, this separation is no longer pertinent (see again (\ref{eq:en_prop})).

\item{(Boundary conditions).}\label{R:bc}
    The choice in \eqref{def:slab} is technically convenient, in particular, with regards to obtaining the precise asymptotics in \eqref{eq:green_asymp_macx}. We expect (with some work) results analogous to those of Theorem~\ref{thm:green_main} to remain true if 
one modifies the setup by any combination of the following: (i) replacing the presence of the killing measure by a Dirichlet boundary condition in the first two (corresponding to $\Z^2$ in \eqref{def:slab}) coordinate directions at spatial scale $N$, with suitable restriction on $x$ (at macroscopic distance from the boundary); (ii) replacing periodic by free boundary conditions in the ``vertical'' direction.
    \end{enumerate}
\end{remark}

\section{Capacity estimates on slabs} \label{sec:cap-slab}

Recall the random walk $X$ on $\mathbb{S}_N$ introduced above \eqref{def:green}, with canonical law denoted by $P_x$, $x \in \mathbb{S}_N$.  For $A\subset \mathbb{S}_N$, we introduce the equilibrium measure of $A$, 
\begin{equation}\label{eq:equil}
    e_{A,N}(x)\stackrel{\text{def.}}{=} \lambda_x P_x(\tilde{H}_A=\infty)\1_{x\in A}, \quad x\in \mathbb{S}_N,
\end{equation}
where $\tilde{H}_A=\tilde{H}_A(X_{\cdot})\coloneqq \inf\{t>0:X_t\in A 
\text{ and there exists } s\in (0,t) \text{ with } X_s\neq X_0\}$ is
the hitting time of $A$. The total mass of the equilibrium measure of $A$ is the capacity of $A$, 
\begin{equation} \label{def:capacity}
    \capacity_N(A)\stackrel{\text{def.}}{=} \sum_{x\in \mathbb{S}_N}e_{A,N}(x).
\end{equation}
The capacity can also be expressed via the variational formula (see \cite[(1.61)]{sznitman_topics_2012})
\begin{equation} \label{eq:cap_var}
    \capacity_N(A) =\left(\inf_\mu  \sum_{x_1,x_2}g_N(x_1,x_2)\mu(x_1)\mu(x_2)\right)^{-1},
\end{equation}
where the infimum is over all probability measures $\mu$ supported on $A$.
By a version of \cite[(1.57)]{sznitman_topics_2012} on infinite graphs, one has the last-exit decomposition, valid for all finite $A\subset \mathbb{S}_N$ and $x\in \mathbb{S}_N$,
\begin{equation} \label{eq:last_exit}
    P_x(H_A< \infty)=\sum_{x'\in \mathbb{S}_N}g_N(x,x')e_{A,N}(x'),
\end{equation}
where $H_A\coloneqq \inf\{t\geq0:X_t\in A \}$ is
the entrance time of $K$. For $y \in \Z^2$ and $R>0$, let $B^{2}(y,R)=\{y'\in \Z^2: |y-y'|<R\}$ (recall $|\cdot|$ denotes the Euclidean norm) and $B_{(\Z/h_N\Z)}(z,R)=\{z'\in (\Z/h_N\Z): d_{(\Z/h_N\Z)}(z,z')< R \}$. For $x\in \mathbb{S}_N,R\geq1$, we consider balls of the form (with $x=(y,z)$ below)
\begin{equation}\label{eq:def_box}
B(x,R) = \{x\in\mathbb{S}_N:\|x\|<R\}
\quad\text{and}\quad
D(x,R) =  B^{2}(y,R) \times B_{(\Z/h_N\Z)}(z,R) 
\end{equation}
(recall $\|\cdot\|$ from below \eqref{eq:K_0})
and line (segments) of the form
\begin{equation*}
  {\ell_R}= ([0,R-1]\cap \Z)\times\{0\}^2.
\end{equation*}
We use $B_R,D_R$ to denote $B(0,R),D(0,R)$. The set $D(x,R)$ corresponds to an ($R$-thickened) two-dimensional disk. For some of the precise estimates we have in mind, $D(x,R)$ will be
more pertinent at large scales $R$, see for instance \eqref{eq:ball-cap-const} below.
This is essentially because the function $g_N^2(\cdot)$ from Proposition~\ref{thm:greens_function_bound}, in regimes of $R$ where it supplies the leading contribution to the Green's function $g_N(\cdot)$, follows asymptotics that depend on $x=(y,z)$ only through $|y|$; see \eqref{eq:green_asymptotics_2d}.

In what follows, we seek estimates for the capacities of lines and balls, with up-to-constant upper and lower bounds uniform in both $R $ and $N$ and
with explicit constants in special cases. 
Recall the function $\Fbox(\cdot)$ from \eqref{eq:def_fbox} and $\Cr{C:range} \in [1,\infty)$ from above Theorem~\ref{thm:green_main}, which is arbitrary.

\begin{proposition}[Capacity of a line]\label{prop:line_cap}
  For all $\epsilon\in(0,1)$ and $N\geq 1$, the following hold:
   \begin{enumerate}[label={(\roman*)}]
   \item  For all $R\geq 1$ such that $R\leq\Cr{C:range}N$,\label{thm:line_cap}
    \begin{equation} \label{eq:line-cap}
      \Cl[c]{lb_line} \, \Big( \frac{R}{\log(R\vee 2)}\wedge \frac{h_N}{K_0(\frac{R\vee h_N}{N})} \Big)
      \leq \capacity_N({\ell_R}) \leq \Cl{ub_line} \Big( \frac{R}{\log(R\vee2)}\wedge \frac{h_N}{K_0(\frac{R\vee h_N}{N})} \Big).
    \end{equation}
  In particular in the case of $R=N \geq 2$, we have that
  \begin{equation}\label{eq:line-cap-macro}
  \begin{array}{cl}
      \Cr{lb_line} \,h_N \leq \capacity_N(\ell_N) \leq \Cr{ub_line} h_N,
      & \text{ when } \frac{N}{h_N\log N }\in [1,\infty  ) \black \\[0.6em]
     \displaystyle    \frac{ \Cl[c]{lb_line_b} N}{\log N }\leq \capacity_N(\ell_N) \leq  \frac{\Cl{ub_line_b} N}{\log N},
      & \text{ when } \frac{N}{h_N\log N} \in [0,1).
      \end{array}
    \end{equation}
\item There exists $\Cl{line_r}(\epsilon),\Cl{line_rn}(\epsilon),\Cl{line_hnn}(\epsilon)<\infty$ such that for all $N,R,h_N$ satisfying $\Cr{line_r}\leq R\leq \frac{1}{\Cr{line_rn}}N$ and $h_N\leq \frac{1}{\Cr{line_hnn}}N$,
\begin{equation} \label{eq:line-cap-const}
\frac{\pi(1-\epsilon)}{3}\left(\frac{\log(R)}{R}+\frac{K_0\big(\frac{R\vee h_N}{N}\big)}{h_N}\right)^{-1}
\leq 
\capacity_N(\ell_R) \leq  \frac{\pi(1+\epsilon)}{3}\left(\frac{\log(R)}{R}+\frac{K_0\big(\frac{R\vee h_N}{N}\big)}{h_N}\right)^{-1}.
\end{equation}
%    In particular, if one further assumes $\lambda=\lim_N \frac{RK_0(\frac{R\vee h_N}{N})}{h_N\log(R)}\in [0,\infty]$ exists then as $N\to\infty$,
%    \begin{equation} 
%     \begin{array}{cl}
%    \capacity_N(\ell_R) \sim \frac{\pi}{3}\frac{h_N}{K_0\left(\frac{R\vee h_N}{N}\right)},
%      & \text{ when } \lambda=\infty\\
%    \capacity_N(\ell_R) \sim \frac{\pi}{3}\frac{R}{\log(R)},
%      & \text{ when } \lambda=0.
%      \end{array}
%    \end{equation}
   \end{enumerate}
\end{proposition}

\begin{proposition}[Capacity of a ball]\label{prop:ball_cap}
  For all $\epsilon\in(0,1)$ and $N\geq 1$, the following hold:
   \begin{enumerate}[label={(\roman*)}]
\item For all $R\geq 1$ such that $R\leq\Cr{C:range}N$ \label{thm_box_cap}, 
 \begin{equation} \label{eq:box-cap}
      \Cl[c]{lb_box} \,\Fbox(R) \leq \capacity_N(B_R) \leq \capacity_N(D_R)\leq \Cl{ub_box} \Fbox(R).
    \end{equation}
  In particular in the case of $R=N$, we have that
  \begin{equation} \label{eq:box-cap-macro}
      \Cr{lb_box} \,h_N \leq \capacity_N(B_N)\leq \capacity_N(D_N) \leq \Cr{ub_box} h_N.
    \end{equation}
    \item  There exists $\Cl{ball_r}(\epsilon),\Cl{ball_rn}(\epsilon),\Cl{ball_hnn}(\epsilon)<\infty$ such that for all $N,R,h_N$ satisfying $\Cr{ball_r}\leq R\leq \frac{1}{\Cr{ball_rn}}N$ and $h_N\leq \frac{1}{\Cr{ball_hnn}}N$,
    \begin{equation} \label{eq:ball-cap-const}
    \begin{array}{cl}
        \frac{\pi(1-\epsilon)}{3}\frac{h_N}{K_0\big(\frac{R\vee h_N}{N}\big)}\leq 
        \capacity_N(D_R) \leq  \frac{\pi(1+\epsilon)}{3}\frac{h_N}{K_0\big(\frac{R\vee h_N}{N}\big)},
    & \text{ when } \frac{1}{ \epsilon R}\leq \frac{K_0(\frac{R\vee h_N}{N})}{h_N}\\[1.5em]
        \frac{2\pi(1-\epsilon)}{3}\left(\frac{1}{R}+ \frac{2K_0\big(\frac{ h_N}{N}\big)}{h_N}\right)^{-1} 
        \!\!\!\!
        \leq 
        \capacity_N(B_R) \leq  \frac{2\pi(1+\epsilon)}{3}\left(\frac{1}{R}+ \frac{2K_0\big(\frac{ h_N}{N}\big)}{h_N}\right)^{-1}
        \!\!\!,
    & \text{ when } \frac{1}{\epsilon R}> \frac{K_0(\frac{R\vee h_N}{N})}{h_N} .
    \end{array}
    \end{equation}
   \end{enumerate}
\end{proposition}

\begin{remark}\label{R:cap}
\begin{enumerate}[label={\arabic*)}]

\item The capacity of lines and boxes have been computed by \cite{drewitz_geometry_2023} in the setting of transient graphs with Green's functions that decay as a polynomial of order $\nu>0$. For such graphs, \cite[(3.11),(3.14)]{drewitz_geometry_2023} suggest $\Z^3$, which corresponds to $\nu=1$, marks the threshold for the ``mismatch" between line and box capacity. In particular, the capacity of a line and a box of diameter $r$ are both of order $r^\nu$, when $\nu\in(0,1)$; whereas when $\nu=1$, the order of capacity of a line and a box is $ \frac{r}{\log r }$ and $r$ respectively. Propositions~\ref{prop:line_cap} and~\ref{prop:ball_cap} (cf.~in particular \eqref{eq:line-cap-macro} and \eqref{eq:box-cap-macro}) refine this threshold as this mismatch appears when $h_N\asymp \frac{N}{\log N}$. See Figure \ref{fig:h_regimes} for a summary of the different phenomena as the height $h_N$ of the slab changes.

\begin{figure}[ht]
\centering
\begin{tikzpicture}[>=stealth, thick]

% Define some placeholder values for demonstration
\def\N{10}            % Total length
\def\logn{2}          % Position representing log(n)
\def\NOverLogN{8}     % Position representing N/log(n)
\def\inset{0.05}

% Main horizontal line
\draw[->] (0,0) -- (\N+0.5,0) node[right] {$h_N$};

% Mark positions: 0, log(n), N/log(n), N
\draw (0,0) -- (0,0.1) node[below=5pt] {0};
\draw (\logn,0) -- (\logn,0.1) node[below=5pt] {$\log N $};
\draw (\NOverLogN,0) -- (\NOverLogN,0.1) node[below=5pt] {$\frac{N}{\log N }$};
\draw (\N,0) -- (\N,0.1) node[below=5pt] {$N$};

% Title above
%\node[above] at (\N/2,2) {\large order of $h_N$};

% z^2 regime (left)
\draw [decorate,decoration={brace,amplitude=5pt,raise=5pt}] (0,0.3) -- (\NOverLogN-0.05,0.3);
\node[above=15pt] at ({(0+\NOverLogN)/2},0.3) {$\capacity_N(\ell_N)\asymp\capacity_N(B_N)$};

% z^3 regime (right)
\draw [decorate,decoration={brace,amplitude=5pt,raise=5pt}] (\NOverLogN+0.05,0.3) -- (\N,0.3);
\node[above=15pt] at ({(\NOverLogN+\N)/2},0.3) {$\capacity_N(\ell_N)\ll\capacity_N(B_N)$};

\fill[blue!10, rounded corners=6pt]
  (0,0.12) rectangle (\logn,0.42);

% Transience band: log N .. N
\fill[cyan!10, rounded corners=6pt]
  (\logn,0.12) rectangle (\N,0.42);

% Labels inside bands (also above axis)
\node at (\logn/2,0.27) {\small recurrence};
\node at ({(\logn+\N)/2},0.27) {\small transience};

\end{tikzpicture}
\caption{Behaviour of the random walk on $\mathbb{S}_N$ for different $h_N$}
\label{fig:h_regimes}
\end{figure}
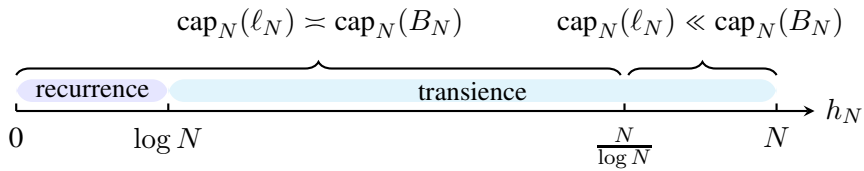
\item As for the Green's function, cf.~Remark~\ref{R:green},\ref{R:green-coro}, the behaviour of the capacity functionals follows either the $\Z^2$ or $\Z^3$ regime depending on the relationship between $r$ and $h_N$. Taking $\capacity_N(B_R)$ for example, one can interpret the result of Proposition~\ref{prop:ball_cap} by distinguishing the following regimes:
\begin{itemize}
    \item (2d regime). If ${RK_0(\frac{R\vee h_N}{N})}>{h_N}$ , then by \eqref{eq:box-cap} we have
        \begin{equation*} 
          \frac{\Cr{lb_box} h_N}{K_0(\frac{R\vee h_N}{N})} \leq \capacity_N(B_R) \leq  \frac{\Cr{ub_box} h_N}{K_0(\frac{R\vee h_N}{N})}.
        \end{equation*}
    \item (3d regime). If $R K_0(\frac{r\vee h_N}{N})<{h_N}$, then by \eqref{eq:box-cap} (cf.~also \eqref{eq:ball-cap-const}) we have
        \begin{equation*} 
        \Cr{lb_box} R\leq \capacity_N(B_R) \leq \Cr{ub_box} R.
        \end{equation*}
\end{itemize}
 
\end{enumerate}
\end{remark}

Before looking into the capacity functionals, we isolate the following estimates for the sum of Bessel functions. Recall $K_0(\cdot)$ from \eqref{eq:K_0}. 

\begin{lemma} \label{prop:bessel_approx}
There exists $\Cl{bessel_ub}<\infty$ such that for all $N,R,k,h\geq1$, we have
\begin{equation} \label{eq:bessel_approx2}
    \sum_{j=1}^k (2^j)^2 K_0\left(\frac{2^j\vee h}{N}\right) \leq C \,(2^k)^2 \left(K_0\left(\frac{{2^k} \vee h}N\right)\vee 1\right),
\end{equation}
and
\begin{equation} \label{eq:K_0_bound}
    \sum_{k=1}^{R} K_0\left({k}/{N}\right) \leq \Cr{bessel_ub} \,R \left(K_0(R/N)\vee 1\right).
\end{equation}
Further, for any $\epsilon\in(0,1)$, there exists $\Cl{N_big}(\epsilon)<\infty$ such that for all $N,R\geq 1$ with $N \geq \Cr{N_big} R$,
\begin{equation} \label{eq:K_0_asymp}
    \sum_{k=1}^{R} K_0\left({k}/{N}\right) \leq (1+\epsilon) \,R K_0(R/N).
\end{equation}
\end{lemma}

Lemma~\ref{prop:bessel_approx} is proved in Appendix~\ref{A:K_0}.

\subsection{Line} In this section we focus on Proposition~\ref{prop:line_cap}, which will follow from Lemmas~\ref{lem:tube_3d} and \ref{lem:tube_2d}. In view of \eqref{eq:projection} and \eqref{eq:cap_var}, we consider two minimisation problems separately to gain some insight. They are % for finite $A\subset \mathbb{S}_N$,
associated to the energy forms
\begin{equation} \label{eq:line_cap_two_prob}
 \mathscr{E}^i(\mu) \stackrel{\text{def.}}{=}\sum_{x_1,x_2\in \mathbb{S}_N} g_N^i(x_1,x_2)\mu(x_1)\mu(x_2), \quad i\in \{2,3\},
  %  \quad\text{ and }\quad
 %   \inf_{\mu\in\mathcal{P}(\ell_R)} \sum_{x_1,x_2\in \ell_R} g_N^2(x_1,x_2)\mu(x_1)\mu(x_2)
\end{equation}
where $\mu$ refers to a probability measure with finite support in $\mathbb{S}_N$.  In the sequel for $A \subset \mathbb{S}_N$ a finite set we write $\mathcal{P}(A)$ for the set of probability measures supported on $A$. We first consider the case $i=3$.

%We now tackle the first minimisation problem in (\ref{eq:line_cap_two_prob}). The proof is similar to that of \cite[Lemma 2.1]{prevost_first_2024} when $\nu=1$, with some additional care needed for the extra exponential factor in $g_N^3$.
\begin{lemma} \label{lem:tube_3d}
For all $\epsilon\in(0,1)$
there exists $\Cl{3d_line_R}(\epsilon)<\infty$ such that for all $ \Cr{3d_line_R}\leq R\leq \Cr{C:range}N$, %and any $\mu^\prime \in \mathcal{P}(\ell_R)$ of the form $\mu^\prime(x)= \frac{1}{R} h\big(\tfrac{\|x\|}{R-1}\big)\1_{\{x\in \ell_R\}}$ with $h$ a bounded and continuous function on $[0,1]$,
    \begin{equation} \label{eq:3d_tube}
        \frac{3(1-\epsilon)}{\pi}\frac{\log R}{R} \leq\inf_{\mu\in\mathcal{P}(\ell_R)} \mathscr{E}^3(\mu)
%\leq  \mathscr{E}^3(\mu')
        \leq \frac{3(1+\epsilon)}{\pi}\frac{\log R}{R}.    \end{equation}
 %   Moreover, for all $ \Cr{3d_line_R}\leq R\leq \Cr{C:range}N$ any     \begin{equation} \label{eq:3d_green_unif_max}
 %      \mathscr{E}^i(\mu^\prime)
 %       \leq 
 %       \frac{3(1+\epsilon)}{\pi}\frac{\log R}{R}.
 %   \end{equation}
\end{lemma}

\begin{proof} We start by showing the upper bound on the infimum in \eqref{eq:3d_tube}. Take $\mu$ to be the uniform measure of $\ell_R$ and $t\in(0,1)$. Combining \eqref{eq:green_3d_bound},  \eqref{eq:g3-asymp-unif} \black and \eqref{eq:green3d-normal}, 
 we get that for all $C(\varepsilon)\leq R\leq \Cr{C:range}N$, 
 \begin{align} \label{eq:3d_tube_unif}
\begin{split}
     \mathscr{E}^3(\mu) &=\frac{1}{R^2} \bigg(\sum_{\substack{x_1,x_2\in \ell_R\\ \|x_1-x_2\|>R^t}} g_N^3(x_1,x_2)+\sum_{\substack{x_1,x_2\in \ell_R\\ \|x_1-x_2\|\leq R^t}} g_N^3(x_1,x_2)\bigg)\\[0.3em]
    &\leq \frac{1}{R^2} \bigg( \frac{3(1+\epsilon)}{2\pi}\sum_{\substack{i,j\in \{0,1,\dots,R-1\}\\ |i-j|>R^t}} \frac{1}{|i-j|}+\sum_{\substack{i,j\in \{0,1,\dots,R-1\}\\ |i-j|\leq R^t}} \frac{\Cr{d3_ub_green}}{|i-j|\vee 1} \bigg)\\[0.3em]
    &\leq \frac{1}{R} \left( \frac{3(1+\epsilon)}{\pi}\log R+ t\Cr{d3_ub_green}\log R+1 \right) \leq \frac{3(1+2\epsilon)}{\pi}\frac{\log R}{R},
\end{split}
\end{align}
where the last inequality follows upon taking $t(\epsilon)\in(0,1)$ to be small enough.

For the lower bound on the infimum in \eqref{eq:3d_tube}, we follow the strategy of the proof of \cite[Lemma 2.2]{goswami_radius_2022}, with some additional care owing to the killing and the presence of two parameters $R$ and $N$.
Let $e_{\cdot,N}^{3}$ be the equilibrium measure of the simple random walk on $\Z^3$ killed by an independent exponential clock $\tau$ of rate $N^{-2}$ and let $\capacity_N^{3}(\cdot)$ be the corresponding capacity, its total mass. That is, using the notation from the beginning of Section~\ref{sec:green-slab},  $e_{A,N}^{3}(x)=\overline{P}_x(\widetilde{H}_A(\overline{X}_{\cdot}) > \tau)1_{x\in A}$, with $\widetilde{H}_A$ as introduced below \eqref{eq:equil}. In view of \eqref{eq:def_g3} and \eqref{eq:line_cap_two_prob}, one has by the same reasoning as in \eqref{eq:cap_var} that $\inf_{\mu\in\mathcal{P}(\ell_R)}   \mathscr{E}^3(\mu) = \capacity_N^{3}(\ell_R)^{-1}$.
Now letting
  \begin{equation*}
      \ell_R^{-}=\ell_R^{-}(\epsilon)=\ell_R\setminus \big(\big(\big([0,R^{1-\epsilon/2}]\cup[R-R^{1-\epsilon/2},R)\big)\cap \Z \big)\times\{0\}^2\big),
  \end{equation*}
one has by analogues of \eqref{def:capacity} and \eqref{eq:last_exit} for the process $X_{\cdot \wedge \tau}$ that
\begin{equation}\label{eq:cut_line_strat}
  \capacity_N^{3}(\ell_R)\leq |\ell_R\setminus \ell_R^-|+
  \sum_{x\in\ell_R^-} e^3_{\ell_R,N}(x)
  \leq
  2(1+R^{1-\epsilon/2}) + \frac{R}{\min_{x_1\in \ell_R^{-}}\sum_{x_2\in \ell_R}g_N^3(x_1,x_2)}.    
\end{equation}
Now by definition of $\ell_R^-$,  \eqref{eq:g3-asymp-unif} and \eqref{eq:green3d-normal}, we have that for all $\Cr{3d_line_R}\leq R\leq \Cr{C:range}N$ with $\Cr{3d_line_R}=\Cr{3d_line_R}(\varepsilon)$ large enough that the minimum on the right-hand side of \eqref{eq:cut_line_strat} is bounded from below by\begin{equation*} 
    2  \sum_{k=\lceil R^{\epsilon/2}\rceil}^{\floor{R^{1-\epsilon/2}}} \frac{3(1-\epsilon)}{2\pi} \frac{1}{k}e^{-\sqrt{6}k/N} 
   \geq \frac{3(1-\epsilon)^2}{\pi}\log(R)e^{-\sqrt{6}R^{1-\epsilon/2}/N} \geq \frac{3(1-\epsilon)^3}{\pi}\log R ,
\end{equation*}
where in the last inequality we also used the fact that $e^{-\sqrt{6}R^{1-\epsilon/2}/N}\geq (1-\epsilon)$ for all $\Cr{3d_line_R}\leq R\leq \Cr{C:range}N$ since $R^{1-\epsilon/2}/N \leq \Cr{C:range}R^{-\epsilon/2}$. Hence combining the last display with \eqref{eq:cut_line_strat} we find that ${\capacity_N^{3}(\ell_R)}/{\frac{\pi}{3}\frac{R}{\log R}}$ is bounded by $C\frac{\log R}{R^{\epsilon/2}}+\frac{1}{(1-\epsilon)^3}\leq \epsilon+\frac{1}{(1-\epsilon)^3}$ for $R \geq C(\varepsilon)$, which concludes the proof of \eqref{eq:3d_tube}.
\end{proof}

We now deal with matters concerning $\mathscr{E}^2(\cdot)$ in \eqref{eq:line_cap_two_prob}. 
\begin{lemma} \label{lem:tube_2d}
%Assume $\limsup_N\frac{h_N}{N}=0$.
For all $\epsilon\in(0,1)$
there exists $\Cl{2d_line_N}(\epsilon),
\Cl{2d_line_h_NN}(\epsilon),
\Cl{2d_line_cr}(\epsilon)<\infty$ such that for all $\Cr{2d_line_N}\leq R\leq \Cr{C:range} N$  
the following hold:
\begin{enumerate}[label={(\roman*)}]
\item If $h_N> \frac{1}{\Cr{2d_line_h_NN}}N$ or $R> \frac{1}{\Cr{2d_line_cr}}N$, then  
    \begin{equation}\label{eq:2d_line_micro_easy_bd}
        \frac{c}{h_N}
        \leq\inf_{\mu\in\mathcal{P}(\ell_R)}  \mathscr{E}^2(\mu)       \leq  \frac{C(\epsilon)}{h_N}
        ;
    \end{equation}
\item If $R\leq \frac{1}{\Cr{2d_line_cr}}N$ and $h_N\leq \frac{1}{\Cr{2d_line_h_NN}}N$, then
    \begin{equation} \label{eq:2d_line_micro}
         \frac{3(1-\epsilon)}{\pi} \frac{K_0(\frac{R\vee h_N}{N})}{h_N}
         \leq\inf_{\mu\in\mathcal{P}(\ell_R)}  \mathscr{E}^2(\mu)
        \leq \frac{3(1+\epsilon)}{\pi} \frac{K_0(\frac{R\vee h_N}{N})}{h_N}.
    \end{equation}
\end{enumerate}
\end{lemma}

\begin{proof}
First recall that $K_0(\cdot)$ is decreasing, as follows plainly from \eqref{eq:K_0}, hence by \eqref{eq:green_bound} and using that $h_N \leq N$ and $R \leq  \Cr{C:range} N$, we have that $g^2_N(x_1,x_2)\geq \frac{c}{h_N}$  for all $x_1,x_2\in \ell_R$, from which the lower bound in \eqref{eq:2d_line_micro_easy_bd} easily follows. For the upper bound in \eqref{eq:2d_line_micro_easy_bd}, if $h_N> \frac{1}{\Cr{2d_line_h_NN}}N$, the factor $K_0(\cdot)$ in \eqref{eq:green_bound} can effectively be neglected by the same reasoning as in \eqref{eq:green_edbound_hnn}. Hence the upper bound in this case follows, as we have that $g^2_N(x_1,x_2) \leq C(\epsilon)h_N^{-1}$  for all $x_1,x_2\in \ell_R$.

Now assume that $R> \frac{1}{\Cr{2d_line_cr}}N$. By \eqref{eq:K_0_bound}, \eqref{eq:green_bound} and taking $\mu$ to be the uniform measure over $\ell_R$, one has that
\begin{equation}\label{eq:line_easy_ub_r}
  \mathscr{E}^2(\mu) =   \frac{1}{R^2}\sum_{x_1,x_2\in\ell_R}g^2_N(x_1,x_2) \leq Ch_N^{-1}(K_0(R/N)\vee 1) \leq C(\epsilon)h_N^{-1} ,
\end{equation}
which concludes the proof of \eqref{eq:2d_line_micro_easy_bd}.

We now show \eqref{eq:2d_line_micro} in the case $h_N\geq R$. 
By \eqref{eq:g2-asymp-unif} (or \eqref{eq:g^2-sharpbd}) \black one has that for all $x_1,x_2\in \ell_R$, if $h_N\geq R (\geq \Vert x_i \Vert)$, $h_N\leq  \frac{1}{\Cr{2d_line_h_NN}}N$ \text{ and }$\Cr{C:range}N\geq R \geq \Cr{2d_line_N}$ with $\Cr{2d_line_N}$ and $\Cr{2d_line_h_NN}$ sufficiently large, the bound \eqref{eq:g^2-sharpbd} holds with $x=x_1-x_2$. Since the resulting upper and lower bounds on $g_N^2(x_1, x_2)=g_N^2(x_1- x_2)$ are uniform in $x$, the claim \eqref{eq:2d_line_micro} easily follows in this case.

It remains to show \eqref{eq:2d_line_micro} when $h_N<R$. 
%\PF{In this case I had to make some adjustments to discard small $i,j$ where we may not have uniformity of the asymptotics (but don't need it). This is why $\ell_R^-$ is back in play.}
We start with the upper bound. Taking $\mu$ to be the uniform measure over $\ell_R$, %, and let $\ell_R^{-}=\ell_R^-({1})$ as defined above \eqref{eq:cut_line_strat}. Bounding $g_N^2(x_1-x_2) \leq C \log N$ using \eqref{eq:green_bound} and \eqref{eq:bessel_ub}, we immediately infer that the contribution to the sum for $\mathscr{E}^2(\mu)$ in \eqref{eq:line_cap_two_prob} stemming from pairs $(x_1,x_2)$ with $x_i \in \ell_R \setminus \ell_R^{-}$ for at least one $i\in \{1,2\}$ are negligible compared to the order of magnitude of the bounds in \eqref{eq:2d_line_micro}. Thus it is enough to consider the case $x_1,x_2 \in \ell_R^-$. For such $x_i$, 
combining \eqref{eq:g2-asymp-unif}, \eqref{eq:K_0_asymp} and monotonicity of $K_0(\cdot)$, one obtains that for all $R,N,h_N$ with $\Cr{2d_line_N}\leq R\leq \frac{1}{\Cr{2d_line_cr}}N $ and $h_N\leq \frac{1}{\Cr{2d_line_h_NN}}N$,
\begin{multline}\label{eq:2d_line_cr0_unif}
   % \frac{1}{R^2}\sum_{x_1,x_2\in \red\ell_R^-\black} g_N^2(x_1,x_2)
   \mathscr{E}^2(\mu)\leq \frac{1}{R^2}\frac{3(1+\epsilon)}{\pi h_N}\sum_{i,j=0}^{R-1}K_0\left(\sqrt{6}\frac{|i-j|\vee 1}{N}\right)
    \leq \frac{1}{R}\frac{3(1+\epsilon)}{\pi h_N}
    \cdot 2\sum_{k=0}^{\lceil R/2\rceil }K_0\left(\sqrt{6}\frac{k \vee 1}{N}\right) \\ \leq \frac{3(1+\epsilon)}{\pi h_N}\left((1+\epsilon)K_0\left(\frac{\sqrt{6} R}{2N}\right) + 
\frac{1}{R}K_0\left(\frac{\sqrt{6}}{N}\right)\right),
\end{multline}
where the second bound follows using the fact that $K_0(\cdot)$ is decreasing to conclude that the function $i \mapsto \sum_{j=0}^{R-1}K_0(\sqrt{6}\frac{|i-j|\vee 1}{N})$ is maximized for $i$ closest to $(R-1)/2$.
Now by \eqref{eq:bessel_ub} we have  for all $R\geq \Cr{2d_line_N}$ and $N \geq 2R$ that
\begin{equation*} 
    \frac{1}{R} \frac{K_0\left({\sqrt{6}}/{N}\right)}{K_0(R/N)}\leq \frac{1}{R}\frac{\log(N/\sqrt{6})+C}{\log(N/R)} 
    =\frac{1}{R}\left(1+\frac{\log(R/\sqrt{6})}{\log(N/R)}\right) + \frac{C}{R\log(N/R)}\leq \epsilon,
\end{equation*}
and similarly using \eqref{eq:bessel_ub} that $K_0(\tfrac{\sqrt{6} R}{2N}) \leq K_0(\tfrac{R}{N})(1+\varepsilon)$ whenever $ N \geq \Cr{2d_line_cr} R$. Feeding these bounds into \eqref{eq:2d_line_cr0_unif}, the upper bound in \eqref{eq:2d_line_micro} follows. 

For the lower bound, by \eqref{eq:g2-asymp-unif} and the fact that $K_0(x)$ is decreasing one has that for all $\mu\in \mathcal{P}(\ell_R)$, $R,N,h_N$ with $\Cr{2d_line_N}\leq R\leq \frac{1}{\Cr{2d_line_cr}} N$ and $h_N\leq \frac{1}{\Cr{2d_line_h_NN}}N$,
\begin{align*}
\begin{split}
    \mathscr{E}^2(\mu)
     &\geq \frac{3(1-\epsilon)}{\pi h_N}
     \sum_{x_1,x_2\in \ell_R} K_0\left(\sqrt{6}\frac{|y_1-y_2|\vee h_N}{N}\right)\mu(x_1)\mu(x_2)\\
     &\geq \frac{3(1-\epsilon)}{\pi h_N}K_0\left(\sqrt{6}\frac{R\vee h_N}{N}\right)\geq \frac{3(1-\epsilon)}{\pi h_N}K_0\left(\sqrt{6}\frac{R}{N}\right),
\end{split}
\end{align*}
where the last inequality holds true since $R>h_N$.
To conclude the lower bound in \eqref{eq:2d_line_micro}, it suffices to note that by \eqref{eq:bessel_ub}, $K_0(\sqrt{6}R/N)\geq K_0(R/N)-C \geq (1-\epsilon)K_0(R/N)$,
where the last inequality follows since $ N \geq \Cr{2d_line_cr} R$.
\end{proof}

\begin{proof}[Proof of Proposition \ref{prop:line_cap}]
In view of \eqref{eq:projection} we have
\begin{align*}
      ( 1+ N^{-2} )\inf_{\mu\in\mathcal{P}(\ell_R)} \sum_{x_1,x_2\in \ell_R} g_N(x_1,x_2)\mu(x_1)\mu(x_2)
    \geq 
    \inf_{\mu\in\mathcal{P}(\ell_R)}  \mathscr{E}^3(\mu)
    +\inf_{\mu\in\mathcal{P}(\ell_R)} \mathscr{E}^2(\mu)
\end{align*}
Hence if $\Cr{line_r}\leq R\leq \Cr{C:range}N$, $R\leq \frac{1}{\Cr{line_rn}}N$ and $h_N\leq \frac{1}{\Cr{line_hnn}}N$, the upper bound in \eqref{eq:line-cap-const} is an immediate consequence of \eqref{eq:3d_tube},  \eqref{eq:2d_line_micro} and \eqref{eq:cap_var}. For the upper bound in \eqref{eq:line-cap}, it suffices to take $\epsilon=1/2$ and note that we can w.l.o.g assume $ \Cr{3d_line_R}(1/2)\vee\Cr{2d_line_N}(1/2)\leq R\leq \Cr{C:range}N$, hence the upper bound in \eqref{eq:line-cap} is an immediate consequence of \eqref{eq:3d_tube}, \eqref{eq:2d_line_micro_easy_bd}, \eqref{eq:2d_line_micro}
and \eqref{eq:cap_var}. 

We now turn to the lower bounds. 
For \eqref{eq:line-cap}, we similarly take $\epsilon=1/2$ and w.l.o.g assume $ \Cr{3d_line_R}(1/2)\vee\Cr{2d_line_N}(1/2)\leq R\leq \Cr{C:range}N$. Combining \eqref{eq:3d_tube_unif}, \eqref{eq:line_easy_ub_r} and \eqref{eq:2d_line_cr0_unif}, and the bound $g^2_N(x_1,x_2) \leq Ch_N^{-1}$  for all $x_1,x_2\in \ell_R$ valid when $h_N \geq c N$, we obtain that
\begin{equation} \label{eq:line_final_lb_micro}
    \frac{1}{R^2}\sum_{x_1,x_2\in \ell_R} \big(g_N^3(x_1,x_2)+g_N^2(x_1,x_2)\big) \leq C\left(\frac{\log R}{R}+\frac{K_0(\frac{R\vee h_N}{N})}{h_N}\right),
\end{equation}
from which \eqref{eq:line-cap} follows via \eqref{eq:cap_var} upon choosing $\mu$ to be uniform.
For \eqref{eq:line-cap-const}, the lower bound is an immediate consequence of \eqref{eq:cap_var} together with \eqref{eq:3d_tube_unif}, \eqref{eq:2d_line_cr0_unif}, and using \eqref{eq:g2-asymp-unif} to deal with the case $h_N \geq R$ complementary to \eqref{eq:2d_line_cr0_unif}, similarly as with the proof of the upper bound for \eqref{eq:2d_line_micro} in that case (see the paragraph following \eqref{eq:line_easy_ub_r}). 
\end{proof}

\subsection{Ball} We now prove Proposition~\ref{prop:ball_cap}, which is simpler.
\begin{proof}[Proof of \eqref{eq:box-cap}]
Note the bound $\capacity_N(B_R)\leq\capacity_N(D_R)$ is an easy consequence of monotonicity of the capacity, as $B_R \subset D_R$ (recall \eqref{eq:def_box}).
We start with the upper bound. For all $x \in D_R$,
  \begin{align*}
    \begin{split}
      1\stackrel{\eqref{eq:last_exit}}{=} \sum_{x' \in D_R} g_N(x,x')e_{D_R,N}(x') 
       \geq \min_{x,x' \in D_R} g_N(x,x') \, \capacity_N(D_R) 
      \stackrel{\eqref{eq:green_estimate}}{\geq}  \capacity_N(D_R)  \times \frac{1}{\Cr{ub_box}\Fbox(R)}
        ,
    \end{split}
  \end{align*}
  where we recall $F_N^h$ from \eqref{eq:def_fbox} for the last inequality.

  For the lower bound, we construct a measure $\nu$ supported on $B_R$ such that
  \begin{equation}\label{eq:nu-condition}
  \sum_{x \in B_R} g_N(\cdot,x)\nu(x) \leq 1, \text{ on $B_R$.}
  \end{equation}
   By the `Principle of domination' (see for instance Proposition 7.6 and Theorem 7.8 in 
  \cite{barlow_random_2017}), any measure $\nu(\cdot)$ satisfying \eqref{eq:nu-condition} yields the lower bound $\capacity_N(B_R) \geq \nu(B_R)$.  We choose $\nu$ to be the uniform measure such that $\nu(x)=|B_R|^{-1}\mu>0$ for all $x \in B_R$ and some scalar $\mu \in (0,\infty)$. In order to satisfy \eqref{eq:nu-condition} we need to pick $\mu$ such that
  \begin{equation}\label{eq:mu-condition}
     \mu \times \max_{x\in B_R} \frac1{|B_R|}\sum_{x' \in B_R} g_N(x,x') \leq 1.
  \end{equation}
  We now determine the largest value of $\mu$ such that \eqref{eq:mu-condition} holds, thus implying $\capacity_N(B_R) \geq \mu$.  For this we consider for a given $x=(y,z)\in B_R$ annuli of the form
  \begin{equation*}
    A_j = \{x'=(y',z') \in B_R: 2^j < \lvert y-y'\rvert \leq  2^{j+1}\}.
  \end{equation*} 
  By picking $k$ such that $R/2 \leq 2^k <R $,
   we get using Theorem~\ref{thm:green_main} that
  \begin{equation}\label{eq:capbox-pf1}
    \sum_{x \in B_R} g_N(x,x') 
    \leq C (R\wedge h_N)\bigg(1+\frac{K_0({h_N}/{N})}{h_N}\bigg) 
    +\sum_{j=1}^k \sum_{x' \in A_j} g_N(x,x') .
  \end{equation}
  First note that
  \begin{equation*}
    \lvert A_j \rvert \leq C (2^j)^2 \times (R\wedge h_N),
  \end{equation*}
  and for $x' \in A_j$, using the upper bound in \eqref{eq:green_estimate},
  \begin{equation*}
    g_N(x,x') \leq \frac{\Cr{d3_ub_green}}{2^j} 
    +\frac{\Cr{d2_ub_green}}{h_N} K_0\left(\frac{h_N\vee 2^j}{N}\right) .
  \end{equation*}
  Hence we obtain using Lemma~\ref{prop:bessel_approx} that
  \begin{multline}\label{eq:capbox-pf2}
      \sum_{j=1}^k \sum_{x' \in A_j} g_N(x,x') 
      \leq C \,(R\wedge h_N) 
      \bigg[ \Cr{d3_ub_green} \sum_{j=1}^k 2^j
      + \frac{\Cr{d2_ub_green}}{h_N}\sum_{j=1}^k (2^j)^2 K_0\left(\frac{h_N\vee 2^j}{N}\right) 
      \bigg] \\
      \stackrel{\eqref{eq:bessel_approx2}}{\leq} \,C'(R\wedge h_N) R^2  \times \bigg[\frac{1}{R}  + \frac{1}{h_N }K_0\left(\frac{R\vee h_N}{N}\right)\bigg].
  \end{multline}
  Feeding this into \eqref{eq:capbox-pf1}, it readily follows that \eqref{eq:capbox-pf2} is the dominating term in the upper bound for $\sum_{x \in B_R} g_N(x,x')$. Now using that $|B_R| \geq c (R\wedge h_N) R^2 $ and observing that the term in square brackets in the last line of \eqref{eq:capbox-pf2} is bounded by $2 F_N^h(R)$, we deduce that
\eqref{eq:mu-condition} is satisfied  by picking
  \begin{equation*}
    \mu = c \Fbox(R), 
  \end{equation*}
  and the ensuing lower bound $\capacity_N(B_R) \geq \mu$ completes the proof of \eqref{eq:box-cap}.
\end{proof}

\begin{proof}[Proof of \eqref{eq:ball-cap-const}]
We start by making two elementary observations that will be useful below. First, using the property that $K_0(\cdot)$ is decreasing and applying \eqref{eq:bessel_ub}, one has that 
\begin{equation} \label{eq:ball_error}
    \frac{R^2(h_N\wedge R)}{N^2}\left(\frac{1}{R}  + \frac1{h_N} {K_0\left(\frac{R\vee h_N}{N}\right)}{ }\right) \leq \frac{R^2}{N^2} \left(1 + K_0(R/N) \right)\to 0 \quad\text{as } N/R\to\infty .
\end{equation}
Second, letting $T_1$ be an independent exponential random variable with rate $1$, for any $B\subset \mathbb{S}_N$ with $\text{int}(B)\coloneqq (B \setminus \partial B)\neq \emptyset$ (recall that $\partial B \subset B $ denotes the inner vertex boundary of $B$), it's easy to see that for all $x\in \text{int}(B)$,
\begin{equation} \label{eq:int_em}
    e_{B,N}(x)= ({ 1}+N^{-2})P_x(T_1 >H_{\Delta})=\frac{{ 1}+N^{-2}}{N^2+1}.
\end{equation}

Now, pick $R\leq \frac{1}{\Cr{ball_rn}}N$ and $h_N\leq \frac{1}{\Cr{ball_hnn}}N$ with $\Cr{ball_rn}$ and $\Cr{ball_hnn}$ large enough such that $\epsilon K_0(\tfrac{R\vee h_N}{N})\geq 1$, hence if $\frac{1}{R}>\epsilon{K_0(\frac{R\vee h_N}{N})}/{h_N}$, this in particular implies that $R\leq h_N$.
%It follows from the above calculation and \eqref{eq:green_asymp_macx}.
As we now explain, upon possibly adapting the value of $\Cr{ball_rn}, \Cr{ball_hnn}$, one has for all $R,N,h_N$ with
$\Cr{ball_r}\leq R\leq \frac{1}{\Cr{ball_rn}}N$ and $h_N\leq \frac{1}{\Cr{ball_hnn}}N$ that
\begin{equation} \label{eq:green_asymp_ball}
     \begin{array}{cl}
    \frac{3(1-\epsilon)}{\pi}\frac{K_0\left(\frac{R\vee h_N}{N}\right)}{h_N}
    \leq g_N(x)\leq  \frac{3(1+\epsilon)}{\pi}\frac{K_0\left(\frac{R\vee h_N}{N}\right)}{h_N},
    & \forall x\in \partial D_R \text{ when }\frac{1}{R}\leq  \epsilon\frac{K_0(\frac{R\vee h_N}{N})}{h_N}\\[0.3em]
    \frac{3(1-\epsilon)}{2\pi} \left(\frac{1}{R}+\frac{2K_0(\frac{h_N}{N})}{h_N}\right)\leq g_N(x)\leq  \frac{3(1+\epsilon)}{2\pi} \left(\frac{1}{R}+\frac{2K_0(\frac{h_N}{N})}{h_N}\right) ,
    &  \forall x\in \partial B_R \text{ when } \frac{1}{R}> \epsilon\frac{K_0(\frac{R\vee h_N}{N})}{h_N} .
    \end{array}   
\end{equation}
Both sets of inequalities in \eqref{eq:green_asymp_ball} essentially follow by a combination of \eqref{eq:projection}, \eqref{eq:g3-asymp-unif} and \eqref{eq:g2-asymp-unif}. The condition on $R$ in the first line of \eqref{eq:green_asymp_ball} ensures on account of \eqref{eq:g3-asymp-unif} that the contribution stemming from $g_N^3$ can be absorbed into the error term involving $\varepsilon$ in \eqref{eq:green_asymp_ball}, and the resulting bounds for $g_N(x)$ follow effectively from \eqref{eq:g2-asymp-unif} alone in this regime. Note that the argument for $K_0(\cdot)$ uses the fact that $R-1 < |y| \leq R$ for $x=(y,z) \in \partial D_R$ (recall $D_R$ from below \eqref{eq:def_box}, which is a natural choice in this regime because the relevant asymptotic quantity $\bar{g}_N^2$ in \eqref{eq:green_asymptotics_2d} only depends on $x$ through $|y|$). 

In the second line of \eqref{eq:green_asymp_ball}, the term $g_N^3$ does contribute, and the term proportional to $\tfrac1R$ stems from \eqref{eq:g3-asymp-unif} and \eqref{eq:green3d-normal}. Here the relevant argument of $K_0(\cdot)$ uses the fact that the regime of $R$ considered implies that $R\leq h_N$, as observed above.

Additionally, one obtains following the same argument as for \eqref{eq:capbox-pf2} that for $B\in\{B_R,D_R\}$ and all $R \leq \Cr{C:range}N$ and $1 \leq h_N \leq N$,
\begin{equation}\label{eq:int_sum_green}
    \sum_{x\in \text{int} (B)}\!g_N(x) \leq CR^2(h_N\wedge R)\left(\frac{1}{R}  + \frac{K_0(\frac{R\vee h_N}{N})}{h_N }\right) .      
\end{equation}

We now start the proof of the upper bounds in \eqref{eq:ball-cap-const}, each of which involves one of the two estimates from \eqref{eq:green_asymp_ball}. By \eqref{eq:last_exit} and \eqref{eq:int_em}, one has that for $B\in\{B_R,D_R\}$ and all $N,R \geq 1$,
\begin{align*}
\begin{split}
    1
    \geq \sum_{x\in B}g_N(0,x)e_{B,N}(x)
    \geq \min_{x\in \partial B} g_N(x)\left(\capacity_N(B) - \frac{CR^2(h_N\wedge R)}{N^2}\right).
\end{split}
\end{align*}
Rearranging the above expression and using \eqref{eq:box-cap} along with \eqref{eq:def_fbox}, it follows that
\begin{align}
\begin{split}
    \capacity_N(B)^{-1}&\geq\min_{x\in \partial B} g_N(x)\left(1- \frac{CR^2(h_N\wedge R)}{N^2}\left(\frac{1}{R}  + \frac{K_0(\frac{R\vee h_N}{N})}{h_N }\right)\right)\geq (1-\epsilon)\min_{x\in \partial B} g_N(x),
\end{split}
\end{align}
where the last inequality follows from \eqref{eq:ball_error} by taking $R\leq \frac{1}{\Cr{ball_rn}}N$ with $\Cr{ball_rn}$ large enough. This concludes the proof of the upper bounds in \eqref{eq:ball-cap-const} upon using the lower bounds for $g_N(x)$ from \eqref{eq:green_asymp_ball} in the corresponding regime.

For the lower bounds in  \eqref{eq:ball-cap-const}, by \eqref{eq:last_exit}, \eqref{eq:int_em} and \eqref{eq:int_sum_green}, one finds that
\begin{align}
\begin{split}
        1
    &=\sum_{x\in B}g_N(0,x)e_{B,N}(x)=\sum_{x\in \partial B}g_N(0,x)e_{B,N}(x)
    +\frac{{ 1}+N^{-2}}{N^2+1}\sum_{x\in \text{int} (B)}g_N(0,x)\\
    &\leq \max_{x\in \partial B} g_N(x) \capacity_N(B) +\frac{CR^2(h_N\wedge R)}{N^2}\left(\frac{1}{R}  + \frac{K_0(\frac{R\vee h_N}{N})}{h_N }\right).
\end{split}
\end{align}
The lower bound in \eqref{eq:ball-cap-const} follows  by taking $R\leq \frac{1}{\Cr{ball_rn}}N$ with $\Cr{ball_rn}$ large enough using \eqref{eq:ball_error} and \eqref{eq:green_asymp_ball}.
\end{proof}

\section{Lower bounds and precise asymptotics in the ``thin'' case} \label{sec:lower_bounds}

We now return to the percolation problem discussed in the Introduction. Some of the results will require a more refined understanding of the walk on the slab $\mathbb{S}_N$. These are postponed to the last two sections. The results of Sections~\ref{sec:green-slab} and~\ref{sec:cap-slab} alone already allow for various conclusions, which we discuss in the present section. These include all lower bounds as well as various results in which the radius $R$ of the box to be crossed effectively satisfies $R \ll N$ but is large enough to make $B_R$ sufficiently ``thin.'' The main observation is that  relating the properties that radius and capacity of a given cluster be large becomes remarkably precise in the ``thin'' case. This mechanism is essentially what underlies the following result, which is
 a stronger version of Theorem \ref{thm:arm_asymp}.

\begin{theorem} \label{prop:arm_general}
    For all $\epsilon\in(0,1)$, there exists $\Cl{arm_r}(\epsilon),\Cl{arm_rn}(\epsilon)<\infty$ such that for all $N,R,h_N$ such that $\Cr{arm_r}\leq R\leq \frac{1}{\Cr{arm_rn}}N$, $h_N\leq \frac{1}{\Cr{arm_rn}}N$ and $\frac{\log R}{ R}\leq \frac{\epsilon}{h_N}\log(\frac{N}{R\vee h_N})$,
    \begin{multline} \label{eq:arm_prop}
        \frac{1}{\pi}\arctan\left[\left(\frac{(1+\epsilon)\frac{\pi}{3}g_N(0)h_N}{\log(N/(R\vee h_N ))}-1 \right)^{-\frac{1}{2}}\right] 
        \leq \theta_N^h(R) 
        \leq \frac{1}{\pi}\arctan\left[\left(\frac{(1-\epsilon)\frac{\pi}{3}g_N(0)h_N}{\log(N/(R\vee h_N ))}-1 \right)^{-\frac{1}{2}}\right].
    \end{multline}
\end{theorem}

Theorem \ref{thm:arm_asymp} can be readily obtained from \eqref{eq:arm_prop}, as we now explain.
\begin{proof}[Proof of Theorem \ref{thm:arm_asymp}]
It suffices to note that under the standing assumption of Theorem \ref{thm:arm_asymp}, for any $\epsilon\in(0,1)$ there exists $N_0(\epsilon)\geq \Cr{arm_r}\Cr{C:range}^{-1}$ such that for all $N\geq N_0(\epsilon)$, we have $\Cr{arm_r}\leq R\leq \frac{1}{\Cr{arm_rn}}N$ and $\frac{\log R}{ R}\leq \frac{\epsilon}{h_N} \log(\frac{N}{R\vee h_N})$. Moreover, the assumption 
$R\gg\frac{h_N\log R}{\log(N/(R\vee h_N))}$ (as $N \to \infty$) of Theorem \ref{thm:arm_asymp} implies in particular that $h_N \ll N$ so we can further assume that $h_N\leq \frac{1}{\Cr{arm_rn}}N$ for $N \geq N_0(\varepsilon)$. Thus, \eqref{eq:arm_prop} holds for such $N$, and the asymptotics \eqref{eq:arm_main} follow by letting first $N\to\infty$, then $\varepsilon \downarrow 0$.
\end{proof}

We now give the proofs of Theorem \ref{prop:arm_general} and of the lower bound in \eqref{eq:critical_one_arm}. Let us first collect a few basic facts concerning $\arctan$.
The following calculus exercise will be useful. For $g>0$ consider $x> g^{-1}$ and $I_g(x)\stackrel{\text{def.}}{=}\frac{1}{2\pi\sqrt{g}} \int_x^\infty \frac{1}{t\sqrt{t-g^{-1}}} \,dt$. With the substitution $u^2=t-g^{-1}$,
\begin{equation} \label{eq:arctan_calculus}
   I_g(x)=
   \frac{1}{\pi\sqrt{g}} \int_{\sqrt{x-g^{-1}}}^\infty \frac{1}{u^2+g^{-1}} \,du 
   = \frac{1}{\pi}\left( \frac{\pi}{2}-\arctan(\sqrt{gx-1})\right)
   =\frac{1}{\pi}\arctan\left(\frac{1}{\sqrt{gx-1}}\right).    
\end{equation}

The following elementary estimates involving $\arctan$ will also be useful.
\begin{lemma} \label{lem:one_arm_minus}
    \begin{align} 
    & \label{eq:arctan_bound}
 \text{for all }x\geq 0: \quad   \frac{\pi}{4}(1\wedge x)\leq \arctan(x)\leq \frac{\pi}{2}(1\wedge x),\\    &\label{eq:one_arm_minus}
   \text{for all $x>1$}: \quad 
 \arctan\left({1}/{\sqrt{x}}\, \right)
    \leq 
    \arctan\left({1}/{\sqrt{x-1}} \,\right) 
    \leq 
    2 \sqrt{2} \arctan\left({1}/{\sqrt{x}} \,\right).    \end{align}
\end{lemma}

We postpone the short verification of Lemma \ref{lem:one_arm_minus} to the end of this section, and proceed to prove the main result of this section. It relies on a refined version of the comparison method with the capacity observable employed in \cite[Section 4]{drewitz_critical_2023}.

\begin{proof}[Proof of Theorem \ref{prop:arm_general}]
Let $\mathcal{C} \subset \tSlab$ denote the cluster of the origin in $\{ \varphi \geq 0\}$.
By Theorem~\ref{thm:green_main},(i) and using \eqref{eq:bessel_ub}, it follows that $g_N(0) \leq C \log N$, hence applying \cite[Lemma 3.2,(2)]{drewitz_cluster_2022} with $g_0=C\log N$ and by translation invariance, it follows that the weighted graph $\Slab$ (see below \eqref{def:slab} for the formal definition) satisfies the condition (Cap) introduced in \cite{drewitz_cluster_2022}. In particular, combining Theorems 1.1 and 3.7 from the same reference, we deduce that $\capacity_N(\mathcal{C})$ is finite almost surely and (cf.~\cite[(3.8)]{drewitz_cluster_2022} and \eqref{eq:arctan_calculus})
\begin{equation}\label{eq:caplaw}
\P_N^h(\capacity_N(\mathcal{C}) > x)= I_g(x), \quad \text{for all } x> g^{-1}, \text{ with } g=g_N(0)
\end{equation}
(we refer to \cite[Section 2]{drewitz_cluster_2022} regarding the extension of $\capacity_N(\cdot)$ as introduced in \eqref{eq:cap_var} to subsets of $\tSlab$). Additionally, as we now explain, with $D_R$ and $\ell_R$ as introduced around \eqref{eq:def_box}, one has the chain of inclusions, 
    \begin{equation} \label{eq:cluster_cap_comparison}
        \{\capacity_N(D_R)<  \capacity_N(\mathcal{C})\}\subset\{0 \leftrightarrow\partial  B_R \}
        \subset
        \{(1-\varepsilon) \capacity_N(\ell_R)<  \capacity_N(\mathcal{C})\},
    \end{equation}
valid for all $\varepsilon \in (0,1)$ and $R,N, h_N$ satisfying the assumptions above \eqref{eq:arm_prop}.

The inclusions \eqref{eq:cluster_cap_comparison} constitute a refinement of \cite[Lemma 4.1]{drewitz_critical_2023}.
Below let $\tilde B_R, \tilde D_R$ be the sets obtained from $B_R$, $D_R$ (see \eqref{eq:def_box}) by adding all the cables joining any pair of neighbors in $B_R$, resp.~$D_R$.% \PF{no cables to cemetery. }
%{\red We define $\tilde B(x,R),\tilde D(x,R)\subset \tilde{\mathbb{S}}_N$ to be the set consisting all sites in $B(x,R),D(x,R)$, all the cables joining any pair of neighbors in $B(x,R)\cup\partial B(x,R),D(x,R)\cup\partial D(x,R)$ and all the cables connecting sites in $B(x,R),D(x,R)$ to the cemetery state $\Delta$.}
The first inclusion in \eqref{eq:cluster_cap_comparison} is valid for all $R,N,h(\cdot)$, and immediate: if $\{0 \leftrightarrow\partial  B_R \}$ does not occur, then  $\mathcal{C} \subset { \tilde B_R}\subset { \tilde D_R}$, hence $\capacity_N(\mathcal{C}) \leq \capacity_N(D_R)$ by  monotonicity {and \cite[(2.16)]{drewitz_cluster_2022}}. 
 For the second inclusion, we consider the cases $R>h_N$ and $R\leq h_N$ separately. 
 
 Let $\hat{R}^3:=R$ and $\hat{R}^2:=\lfloor{R}/{2}\rfloor$. On the event $\{0 \leftrightarrow\partial  B_R \}$ one extracts from $\mathcal{C}$ a finite sequence $(x_n^{(3)}: 0\leq n\leq \hat{R}^{3})$ with $x_n^{(3)}=(y_n^{(3)},z_n^{(3)})$ of vertices in $\Slab$ with $|x_n^{(3)}|_{\infty}=n$ when $R\leq h_N$ and similarly $(x_n^{(2)} : 0\leq n\leq \hat{R}^2)$ with
$|y_n^{(2)}|_{\infty}=n$ when $R> h_N$. (Note the existence of the sequence $(x_n^{(2)})$ is justified when $R>h_N$ by the fact that for all $x=(y,z)\in\partial B_R$, one has $|y|_\infty\geq |y|\geq \sqrt{R^{2}-(h_N/2)^{2}}\geq R/2$.) 
Let $\hat{\ell}^{i}_R (\subset \mathcal{C})$ denote the union of all the points in $(x_n^{(i)})$. It's straightforward to see that $\capacity_N(\mathcal{C}) \geq \capacity_N(\hat{\ell}^i_R)$; hence it to prove \eqref{eq:cluster_cap_comparison}, it suffices to show that $\capacity_N(\hat{\ell}^2_R)\geq (1-\epsilon)\capacity_N(\ell_R)$ when $R> h_N$ and $\capacity_N(\hat{\ell}^3_R)\geq (1-\epsilon)\capacity_N(\ell_R)$ when $R\leq h_N$.

 Now for $i\in\{3,2\}$, let $\tau^i:\mathbb{S}_N\mapsto\mathbb{S}_N$ be the projection such that $\tau^i(x_n^{(i)})=(n,0,0)$ for all $0\leq n\leq \hat{R}^i $. It's easy to see that $\tau^i(\hat{\ell}^i_R)=\ell_{\hat{R}^i+1}$ and for all $x^{(i)}=(y^{(i)},z^{(i)}),\bar{x}^{(i)}=(\bar{y}^{(i)},\bar{z}^{(i)}) \in \hat{\ell}^i_R$,
\begin{equation}\label{eq:tau_proj}
    \|\tau^i(x^{(i)})-\tau^i(\bar{x}^{(i)})\|\leq \|x^{(i)}-\bar{x}^{(i)}\| \text{ for } i\in\{3,2\} \text{ and }
    |\tau^2(y^{(2)})-\tau^2(\bar{y}^{(2)}) |\leq |y^{(2)}-\bar{y}^{(2)}|,
\end{equation}
where, with a light abuse of notation, we use $\tau^2(y)$ to denote the first two components of $\tau^2(x)$. 
 By \eqref{eq:tau_proj}, \eqref{eq:green_3d_bound} and \eqref{eq:g2-asymp-unif}, we have that $$\kappa^3\stackrel{\text{def.}}{=}\max_{i\in\{3,2\}}\max_{x,x^\prime \in \hat\ell^i_R}\frac{g_N^3(x,x^\prime)}{g_N^3(\tau(x),\tau(x^\prime))}\leq \frac{\Cr{d3_ub_green}}{\Cr{d3_lb_green}} \frac{\|\tau(x)-\tau(x^\prime)\|\vee1}{\|x-x^\prime\|\vee1}\leq C$$ and similarly $$\kappa^2\stackrel{\text{def.}}{=}\max_{x,x^\prime \in \hat\ell^2_R}\frac{g_N^2(x,x^\prime)}{g_N^2(\tau(x),\tau(x^\prime))}\leq 1+\epsilon$$ for $\Cr{arm_r}\Cr{arm_rn}\leq N,h_N\leq \frac{1}{\Cr{arm_rn}}N$, with $\Cr{arm_r}(\epsilon),\Cr{arm_rn}(\epsilon)$ sufficiently large. 
If $R>h_N$, letting $\nu$ be the uniform measure on $ \ell_{\hat{R}^2+1}$, 
\begin{align}
\begin{split}
    \capacity_N(\hat{\ell}^2_R)&=(\inf_{\mu\in\cP(\hat{\ell}^2_R)} \mathscr{E}^3(\mu)+\mathscr{E}^2(\mu))^{-1}
    \geq (\kappa^3\mathscr{E}^3(\nu)+\kappa^2\mathscr{E}^2(\nu))^{-1}\\
    &\geq \left(C\frac{2\log(R/2)}{ R}+ (1+\epsilon)^2\frac{3}{\pi}\frac{1}{h_N}K_0\left(\frac{R}{2N}\right)\right)^{-1} \geq (1-\epsilon)\capacity_N(\ell_R),
\end{split}
\end{align}
where the third inequality follows from \eqref{eq:3d_tube_unif} and \eqref{eq:2d_line_cr0_unif}; the last inequality follows from the assumption $\frac{\log R}{R}\leq\frac{\epsilon}{h_N}\log(\frac{N}{R\vee h_N})$, \eqref{eq:line-cap-const} (which is in force under our assumptions on $R, N, h_N$), and a change of variable in $\epsilon$.

If $R\leq h_N$, letting $\nu$ be the uniform measure on $ \ell_{\hat{R}^3+1}$, 
\begin{align}
\begin{split}
    \capacity_N(\hat{\ell}^3_R)&=\big(\inf_{\mu\in\cP(\hat{\ell}^3_R)} \mathscr{E}^3(\mu)+\mathscr{E}^2(\mu)\big)^{-1}\\
    &\geq \left((\kappa^3\mathscr{E}^3(\nu)+ (1+\epsilon)\frac{3}{\pi}\frac{1}{h_N}K_0\left(\frac{h_N}{N}\right)\right)^{-1}\\
    &\geq \left(C\frac{\log( R)}{ R}+ (1+\epsilon)\frac{3}{\pi}\frac{1}{h_N}K_0\left(\frac{h_N}{N}\right)\right)^{-1} \geq (1-\epsilon)\capacity_N(\ell_R),
\end{split}
\end{align}
where the second inequality follows from \eqref{eq:g^2-sharpbd}; the third inequality follows from \eqref{eq:3d_tube_unif} and the last inequality follows from the assumption $\frac{\log R}{R}\leq\frac{\epsilon}{h_N}\log(\frac{N}{R\vee h_N})$, \eqref{eq:line-cap-const} and a change of variable in $\epsilon$.

\black

With \eqref{eq:cluster_cap_comparison} now shown, its usefulness hinges on having sufficiently sharp estimates for $\capacity_N(D_R)$ and $\capacity_N(\ell_R)$, which are supplied by the results of Section~\ref{sec:cap-slab}; see, in particular, item~(ii) in each of Propositions~\ref{prop:line_cap} and \ref{prop:ball_cap}. In the ``flat'' regime of parameters $R,N,h_N$ considered here, these asymptotics essentially match to leading order. Indeed, combining the 
  fact that $K_0(x)\sim \log(1/x)$ as $x\to 0$ (see Lemma~\ref{L:Bessel}), \eqref{eq:line-cap-const}, and the first line of \eqref{eq:ball-cap-const}, we get that for $A_R\in\{\ell_R,D_R\}$,    whenever $R,N,h_N$ satisfy the conditions above \eqref{eq:arm_prop} for a given $\varepsilon \in (0,1)$,\begin{equation}\label{eq:arm_line_ball}
        \frac{\pi(1-\epsilon)}{3}\frac{h_N}{\log(N/(R\vee h_N ))}
        \leq \capacity_N(A_R)
        \leq \frac{\pi(1+\epsilon)}{3}\frac{h_N}{\log(N/(R\vee h_N ))}.
    \end{equation}
 The claim \eqref{eq:arm_prop} is a direct consequence of \eqref{eq:cluster_cap_comparison}, \eqref{eq:caplaw}, \eqref{eq:arctan_calculus} and \eqref{eq:arm_line_ball}.
\end{proof}

A slight adaptation of the above method also yields the lower bound in Theorem~\ref{thm:critical_connect}.

\begin{proof}[Proof of the lower bound in \eqref{eq:critical_one_arm}] Recall that the first inclusion in \eqref{eq:cluster_cap_comparison} holds without restrictions on $R,N$ and $h_N$. Thus, it remains valid in the context of Theorem~\ref{thm:critical_connect}. Combining it with \eqref{eq:box-cap} gives
\begin{equation*}
    \{\Cr{ub_box}\Fbox(R)\leq  \capacity_N(\mathcal{C})\}\subset\{0 \leftrightarrow\partial  B_R \},
\end{equation*}
since the event on the left implies that $\capacity_N(\mathcal{C}) \geq \capacity_N(D_R)$.
Hence, applying \eqref{eq:caplaw}, \eqref{eq:arctan_calculus}, we get that
\begin{equation*}
    \theta_{N}^{h}(R) \geq c \arctan\big((\Cr{ub_box}g_N^h(0)\Fbox(R)-1)^{-1/2} \big) %\stackrel{\eqref{eq:one_arm_minus}, \eqref{eq:arctan_bound}}{\geq}\frac{c}{\sqrt{\Cr{ub_box}g_N^h(0)\Fbox(R)}},
\end{equation*}
from which \eqref{eq:critical_one_arm} follows using \eqref{eq:one_arm_minus},\eqref{eq:arctan_bound}.
\end{proof}

It remains to supply the elementary:

\begin{proof}[Proof of Lemma \ref{lem:one_arm_minus}]
The bounds in \eqref{eq:arctan_bound} follow from elementary considerations, upon combining the facts that $\arctan'(0)=1$, $\arctan(1)=\tfrac{\pi}{4}$ and $\arctan(
\cdot) \leq \tfrac{\pi}{2}$ and concavity.
As to \eqref{eq:one_arm_minus}, the lower bound is immediate since $\arctan(\cdot)$ is increasing. For the upper bound, using \eqref{eq:arctan_bound} repeatedly we have that 
\begin{equation*}
    \arctan\left(\frac{1}{\sqrt{x-1}}\right)\leq \frac{\pi}{2}\left(\frac{1}{\sqrt{x-1}} \wedge 1\right)\leq 
    \begin{cases*}
    \frac{\pi}{2} \leq 2\sqrt{2}\arctan\left(\frac{1}{\sqrt{x}}\right), & if $1< x\leq 2$ \\
    \frac{\pi}{2} \sqrt{\frac{2}{x} } = \frac{\pi}{4}\frac{2\sqrt{2}}{\sqrt{x}} \leq 2\sqrt{2}\arctan\left(\frac{1}{\sqrt{x}}\right), &  if $x>2$  
    \end{cases*}.
\end{equation*}
\end{proof}

\begin{remark}[Plateau]\label{R:plateau} From \eqref{eq:caplaw} and the capacity estimates \eqref{eq:line-cap-macro} and \eqref{eq:box-cap-macro} in the regime $h_N \ll N/\log N$, one sees via the comparison between radius and capacity as in \eqref{eq:cluster_cap_comparison} that $\text{cap}_N(\mathcal{C})$ is of order $g h_N$, where $g=g_N(0)$, whenever $\mathcal{C}$ is connected to distance $N$, and in particular scales multiplicatively in $g$. In the delocalized regime $h_N \ll \log N$, the Green's function $g$ scales like $h_N^{-1} \log N$, cf.~\eqref{eq:variance-GFF}, hence the factors of $h_N$ cancel and $\text{cap}_N(\mathcal{C}) \asymp \log N$ uniformly in $N$ on the connection event, yielding estimates uniform in $h_N$ as in the first line of \eqref{critical_one_arm-macro}.

\end{remark}

\section{Killed estimates and Harnack inequality}
\label{eq:sec:kill}

We now extend the bounds on $g_N$ of Theorem~\ref{thm:green_main} to accommodate killing on a suitable set $K \subset \Slab$, giving rise to the killed Green's function $g_N^K$. The effect of the killing on the asymptotics is the content of Proposition~\ref{cor:killed_green}, the main result of this section. Proposition~\ref{cor:killed_green} has several immediate consequences that will be useful in the next section. Namely, an (elliptic) Harnack inequality, see Corollary~\ref{C:harnack}, as well as killed capacity estimates for balls, see Corollary~\ref{cor:killed_cap}.  Using the estimates from Proposition~\ref{cor:killed_green} and with a bit more work, we then prove in Proposition~\ref{lem:cap_walk_space_2} lower deviation bounds on the (killed) capacity of the range of the walk, which are fit for purpose.

Let $K\subset \mathbb{S}_N$ be compact, and $P_x^K$ be the law of $X_{\cdot \wedge H_K}$ under $P_x$ (see above \eqref{def:green} regarding $X$) the canonical law of the continuous-time Markov chain killed upon hitting $K$; see below \eqref{eq:last_exit} regarding $H_K$, the entrance time in $K$. We denote the Green's function killed upon hitting $K$ by
    \begin{equation} \label{eq:gNK}
      g_N^K(x_1,x_2)\stackrel{\text{def.}}{=} \frac{1}{\lambda_{x_1}}E_{x_1}^K\Big[\int_0^{\infty} \1_{\{X_t=x_2\}} \,dt\Big]
      =\frac{1}{\lambda_{x_1}}E_{x_1}\Big[\int_0^{H_K} \1_{\{X_t=x_2\}} \,dt\Big],
    \end{equation}
    so that $g_N= g_N^{\emptyset}$. Note that $ g_N^K(x_1,x_2)$ is no longer translation invariant.
Similarly as in \eqref{eq:equil}-\eqref{def:capacity}, for $A\subset \mathbb{S}_N$, we define the equilibrium measure $e_{A,N}^K(\cdot)$ of $A$ and the capacity $\capacity_N^K(A)$ of $A$ relative to $\mathbb{S}_N\setminus K$, with $P_x^K$ replacing $P_x$ everywhere, or equivalently, $H_K$ replacing $\infty$. The analogue of the variational formula \eqref{eq:cap_var} remains true for $\capacity_N^K(A)$, with $g_N^K$ in place of $g_N$. The following result extends the bounds of Theorem~\ref{thm:green_main},~(i) to allow for a killing on the set $K$.

\begin{proposition} \label{cor:killed_green}
    For all $N\geq 1,r\geq 2$ such that $r\leq \Cr{C:range}N$, $x_0=(y_0,z_0)\in \mathbb{S}_N$ and $K=B(x_0,r)^{\text{c}}$, and all $x_1,x_2,\in B(x_0,r/2)$, the following hold: 
%    \begin{enumerate}[label={(\roman*)}]
 %       \item if $r \geq  2h_N$, we have
          \begin{multline} \label{eq:killed_green_smallh} 
            \frac{\Cl[c]{d3_lb_killed}}{\|x_1-x_2\|\vee 1} +\frac{\Cl[c]{d2_lb_killed}}{h_N}  \log\left(\frac{r}{(|y_1-y_2|\vee h_N) \wedge \tfrac r2 }\right)         \\   \leq g_N^K(x_1,x_2) 
            \leq 
            \frac{\Cl{d3_ub_kileed}}{\|x_1-x_2\|\vee 1}+
              \frac{\Cl{d2_ub_kileed}}{h_N}  \log\left(\frac{r}{(|y_1-y_2|\vee h_N)  \wedge \tfrac r2 }\right). %1_{r \geq 2 h_N}.
          \end{multline}
 %       \item if $r< 2 h_N$, we have
 %         \begin{equation} \label{eq:killed_green_largeh} 
 %           \frac{\Cl[c]{d3_lb_killed}}{\|x_1-x_2\|\vee 1} 
 %           \leq g_N^K(x_1,x_2)
 %           \leq 
 %            \frac{\Cl{d3_ub_kileed}}{\|x_1-x_2\|\vee 1} .
 %         \end{equation}
 %   \end{enumerate}
\end{proposition}
In comparison with \eqref{eq:green_estimate}, and in view of Lemma~\ref{L:Bessel}, the effect of the killing is to localize the two-dimensional effect to scale $r$, whenever it is felt, i.e.~when $r \gtrsim h_N$.
Notice also concerning the argument of the logarithm that in the present case, (and unlike in Theorem~\ref{thm:green_main}, where one always has ${N}{} \gtrsim (|y|\vee h_N)$ by assumption), it may well be that ${r} \ll ({|y_1-y_2|\vee h_N})$.

We postpone the proof of Proposition~\ref{cor:killed_green} for a few lines. We say that $u: \mathbb{S}_N \to \mathbb{R}$ is \emph{harmonic} at $x$ if $Lu(x)=0$, where $Lu(x)=\sum_{x'}\lambda_{x,x'}(u(x')-u(x))- \kappa_x u(x)$. As a consequence of Proposition~\ref{cor:killed_green}, we obtain the following:

\begin{corollary}[Harnack inequality on $\mathbb{S}_N$] \label{C:harnack}
    There exists $\Cl{harnack}<\infty$ such that for all $t\in(0,\frac{1}{2}]$, $R\geq 1$ and any function $u$ that's nonnegative on $\bar{B}_{2R} (=B_{2R} \cup \partial (B_{2R}^c))$ and harmonic on $B_{2R}$,
    \begin{equation} \label{eq:harnack_general}
        u(x)\leq \Cr{harnack}u(x^\prime) \quad\quad \forall x,x^\prime \in B_{tR}.
    \end{equation}
\end{corollary}
\begin{proof}
It follows from Proposition \ref{cor:killed_green} that 
\begin{align*}
\begin{split}
      &\max_{x,x^\prime \in B_{tR}} \max_{w\in \partial B_R}\frac{g^{B_{2 R}}_N(x,w)}{g^{B_{2R}}_N(x^\prime,w)}\leq \frac{C\left(\frac{1}{(1-t)R}+\frac{\log(2/(1-t))}{h_N}  1_{R \geq 2h_N} \right)}{c\left(\frac{1}{R}+\frac{\log 2}{h_N} 1_{R \geq 2h_N} \right)}\leq C, %\\ 
  %    &\max_{x,x^\prime \in B_{tR}} \max_{w\in \partial_{int}B_R}\frac{g^{B_{2 R}}_N(x,w)}{g^{B_{2R}}_N(x^\prime,w)}\leq \frac{C}{1-t}\leq C,
\end{split}
\end{align*}
from which \eqref{eq:harnack} follows by a straightforward adaptation (to incorporate the presence of the killing measure) of the proof of \cite[Lemma A.2]{sznitman_critical_2011}.
\end{proof}

We also record the following bounds on box capacities with killing on $K$, which will be useful later.
\begin{corollary} \label{cor:killed_cap}
    For all $N\geq 1$, $r^\prime\geq r\geq 1$ such that $r^\prime\leq \Cr{C:range}N$, $x_0\in \mathbb{S}_N$ and $K\subseteq B(x_0,2r^\prime)^{\text{c}}$: 
    \begin{enumerate}[label={(\roman*)}]
        \item if $r^\prime \geq h_N$, we have
        \begin{equation}
        c \left(r\wedge \frac{h_N}{\log\big(\frac{2r^\prime}{r\vee h_N}\big)}\right)
        \leq 
        \capacity_N^K(B(x_0,r)) 
        \leq C \left(r\wedge \frac{h_N}{\log\big(\frac{2r^\prime}{r\vee h_N}\big)}\right).
        \end{equation}
        \item if $r^\prime < h_N$, we have 
        \begin{equation}
       c r
        \leq 
        \capacity_N^K(B(x_0,r)) 
        \leq C r.
        \end{equation}
    \end{enumerate}
\end{corollary}

Corollary \ref{cor:killed_cap} is a direct consequence of Proposition \ref{cor:killed_green} by running through the proof of \eqref{eq:box-cap} with the estimates in Proposition \ref{cor:killed_green}.

\begin{proof}[Proof of Proposition~\ref{cor:killed_green}]
We start by doing a projection of $(X_t)_{t\geq0}$ as in \eqref{eq:projection}; some extra care is needed here owing to how the killing set $K$ is projected. Let 
\begin{equation} \label{eq:def_bproj}
    B^{\text{proj}}= B^{\text{proj}}(x_0,r)=\left\{x=(y,z)\in \Z^3: |y-y_0|^2+\big(d_{\Z}(\hat{z}_0,z)\text{ mod }\tfrac{h_N}{2}\big)^2 <r^2\right\}
    \text{ and }
    K^{\text{proj}}=(B^{\text{proj}})^{ c }.
\end{equation}
    With hopefully obvious notation, we proceed with the same decomposition as in \eqref{eq:projection} and write
    \begin{equation}\label{eq:decomp-kill}
        g_N^K(x_1,x_2) = g_N^{K,3}(x_1,x_2)+g_N^{K,2}(x_1,x_2),
    \end{equation}
    where (cf. below~\eqref{eq:projection})
    \begin{align}
    &g_N^{K,3}(x_1,x_2) = \frac{1}{\lambda_{x_1}} \int_0^\infty P_{(y_1,\hat{z}_1)} \big((Y_t,Z_t)=(y_2,\hat{z}_2),t<H_{K^{\text{proj}}}\big) e^{-\frac{t}{N^2}} \,dt, \text{ and}\\
    &  \label{eq:def_g2K}
    g_N^{K,2}(x_1,x_2) = \frac{1}{\lambda_{x_1}} \int_0^\infty \sum_{k\in \mathbb{Z} \setminus\{0\}} P_{(y_1,\hat{z}_1)} \big((Y_t,Z_t)=(y_2,\hat{z}_2+  kh_N),t<H_{K^{\text{proj}}}\big) e^{-\frac{t}{N^2}}  
    \,dt.
    \end{align}
  
        We start with the ``topologically trivial" part $g_N^{K,3}$ and claim that
    \begin{equation} \label{eq:killed_g3}
        \frac{c}{\|x_1-x_2\|\vee 1} \leq g_N^{K,3}(x_1,x_2) \leq \frac{C}{\|x_1-x_2\|\vee 1} .    
    \end{equation}
    The upper bound of \eqref{eq:killed_g3} follows from the upper bound in \eqref{eq:green_3d_bound} since $P \big((Y_t,Z_t)=(y,\hat{z}),t<H_{K^\text{proj}}\big)\leq P \big((Y_t,Z_t)=(y,\hat{z})\big)$ and hence $g_N^{K,3}(x_1,x_2)\leq g_N^{3}(x_1,x_2)$. For the lower bound, note that since $\{x\in \Z^3:|x-x_0|<r\}\subset B^{\text{proj}}$,
    by the bound \cite[Theorem 5.26]{barlow_random_2017} with $\alpha=3,\beta=2$ on the killed heat kernel, and using that $e^{-t/N^2}\geq e^{-(r/N)^2} \geq c$ for all $t \leq r^2$, we obtain that when $\|x_1-x_2\|\geq 1$,
    \begin{align*}
    \begin{split}
        g_N^{K,3}(x_1,x_2) 
        &\geq c \int_{\|x_1-x_2\|}^{r^2} t^{-3/2}e^{-\frac{3\|x_1-x_2\|^2}{2t}} \,dt   
        \geq \frac{c}{\|x_1-x_2\|} \int_{\frac{3\|x_1-x_2\|^2}{2r^2}}^{\frac{3\|x_1-x_2\|}{2}} s^{-1/2}e^{-s} \,ds  \geq \frac{c'}{\|x_1-x_2\|}.
    \end{split}
    \end{align*}
    When $\|x_1-x_2\|=0$, a trivial lower bound is obtained by integrating over $t\in[1,2]$ in the first line of the display above. Overall, this establishes \eqref{eq:killed_g3}.

    We now move onto upper bounding the term $g_N^{K,2}(x_1,x_2)$. For $a,b\in[0,\infty]$, let 
    \begin{equation}\label{eq:killed_integtal}
        I(a,b)= \frac{1}{\lambda_{x_1}}\int_a^b\sum_{k\in \mathbb{Z} \setminus\{0\}} P_{(y_1,\hat{z}_1)} \big((Y_t,Z_t)=(y_2,\hat{z}_2+  kh_N),t<H_{K^\text{proj}}\big) e^{-\frac{t}{N^2}}  
    \,dt.
    \end{equation}
    In view of \eqref{eq:def_bproj}, it is easy to see that $B^{\text{proj}}(x_0,r)\subset B^2(y_0,r)\times \Z $. Hence letting $K^2=B^2(y_0,r)^{\text{c}} \subset \Z^2$ (see around \eqref{eq:def_box} for notation), one has using \eqref{eq:X-decomp} that 
    \begin{equation} \label{eq:smaller_killing}
        I(a,b)\leq \frac{1}{\lambda_{x_1}}\int_a^b
        P^2_{y_1}(Y_t=y_2,t<H_{K^2})
        \sum_{k\in \mathbb{Z} \setminus\{0\}}  \Pz{\hat{z}_1}(Z_t=\hat{z}_2+  kh_N)  
    \,dt.
    \end{equation}
    Now note that by applying the Markov property at time $t/2$, we have that for all $t\geq r^2$,
    \begin{align} \label{eq:killed_hk_large_t}
    \begin{split}
        \Pzd{y_1} (Y_t=y_2,t<H_{K^2})
        &=\sum_{y\in B^2(y_0,r)}
        \Pzd{y_1} (Y_{t/2}=y,t/2<H_{K^2})
        \Pzd{y} (Y_{t/2}=y_2,t/2<H_{K^2}) \\
        & \leq \sup_{y\in B^2(y_0,r)}\Pzd{y} (Y_{t/2}=y_2) \Pzd{y_1} (t/2<H_{K^2})
        \leq \frac{C}{r^2}\exp\left\{-\frac{ct}{r^2}\right\},
    \end{split}
    \end{align}
    where the last inequality follows from bounding the first probability by \cite[(5.18)]{barlow_random_2017} and the fact that $t\geq r^2$ and bounding the second probability by \cite[(2.50)]{lawler_random_2010}. Now by applying, in this order, \eqref{eq:smaller_killing}, the same argument as for \eqref{eq:prop_lclt_approx_z} (with $M=r^2,d=2$; the presence of the additional killing for $Y_t$ is  inconsequential), and the upper bound in \eqref{eq:approx_sum_gaussian}, which is in fact true for all $t>0$, one obtains that
    \begin{align} \label{eq:killed_int_large_t}
    \begin{split}
        I(r^2,\infty)
        \leq&\frac{1}{\lambda_{x_1}}\int_{r^2}^\infty \Pzd{y_1} (Y_t=y_2,t<H_{K^2}) 
        \sum_{k\in \mathbb{Z} \setminus\{0\}}  \Pz{\hat{z}_1}(Z_t=\hat{z}_2+  kh_N)\,dt\\
        \leq 
        &\frac{1}{\lambda_{x_1}}\frac{1}{h_N}\int_{r^2}^\infty \Pzd{y_1} (Y_t=y_2,t<H_{K^2}) 
         \left[\sum_{k\in \mathbb{Z} \setminus \{0\}}  \sqrt{\frac{3 h_N^2}{2\pi t}} e^{ - \frac{3|\hat{z}_2-\hat{z}_1+ kh_N|^2}{2t}}\right] \,dt + \frac{C}{h_N^2 r}\\
         \leq &\frac{1}{\lambda_{x_1}}\frac{1}{h_N}
         \int_{r^2}^\infty  \Pzd{y_1} (Y_t=y_2,t<H_{K^2})  \,dt + \frac{C}{h_N^2 r} \leq \frac{C}{h_N},
    \end{split}
    \end{align}
    where the last inequality follows from \eqref{eq:killed_hk_large_t}.

    If $r \geq 2h_N$, by applying \eqref{eq:prop_lclt_main_hbig} with $d=2,r=1,M=h_N^2$ and \eqref{eq:approx_sum_gaussian}, one finds that
    \begin{align} \label{eq:killed_int_small_t}
    \begin{split}
        I(h_N^2,r^2)\leq&C \int_{h_N^2}^{r^2} \Pzd{y_1} (Y_t=y_2) 
        \sum_{k\in \mathbb{Z} \setminus\{0\}}  \Pz{\hat{z}_1}(Z_t=\hat{z}_2+  kh_N)\,dt\\
        \leq &
        \frac{C}{h_N}\int_{h_N^2}^{r^2}\frac{1}{t}\exp\left\{-3\frac{|y_1-y_2|^2}{2t}\right\}\,dt + \frac{C^\prime}{h_N^3} \\
        \leq &\frac{C}{h_N} 
        \int_{\frac{3|y_1-y_2|^2}{2r^2}}^{\frac{3|y_1-y_2|^2}{2h_N^2}} \frac{e^{-u}}{u}\,du +\frac{C^\prime}{h_N^3} 
        \leq 
        \begin{cases*}
        \frac{C}{h_N} \log\big(\frac{r}{|y_1-y_2|}\big), & if $|y_1-y_2|\geq h_N$ \\
        \frac{C}{h_N} \log\big(\frac{r}{h_N}\big) , &  if $|y_1-y_2|<h_N$  
        \end{cases*},
    \end{split}
    \end{align}
where for the last inequality, we used \eqref{eq:exp_inequality} when $|y_1-y_2|\geq h_N$ and used $e^{-u}<1$ when $|y_1-y_2|< h_N$.
    Additionally, without any condition on $r$ and $h_N$, it's an easy consequence of \eqref{eq:en_prop} with $M=1$ that 
    \begin{equation}\label{eq:killed_int_very_small_t}
       I(0,h_N^2)\leq \frac{C}{h_N\vee |y_1-y_2|} { +\frac{C^\prime}{h_N^2\sqrt{h_N\vee |y_1-y_2|}}}, 
    \end{equation}
    hence combining this with \eqref{eq:killed_int_large_t} and \eqref{eq:killed_int_small_t}, we have that 
    \begin{equation} \label{eq:killed_g2_ub}
        g_N^{K,2}(x_1,x_2) \leq 
        \begin{cases*}
        \frac{C}{h_N}\log\big(\frac{r}{|y_1-y_2|\vee h_N}\big), &  if $r\geq2h_N$ \\
        \frac{C}{h_N} , & if $r< 2h_N$  
        \end{cases*},
    \end{equation}
    which, together with the upper bound on $g_N^{K,3}$ from \eqref{eq:killed_g3} and \eqref{eq:decomp-kill}, concludes the proof of the upper bound in \eqref{eq:killed_green_smallh}. 

    For the lower bound of $g_N^{K,2}$, when $r< \lambda h_N$ for some $\lambda \geq 2$ to be fixed momentarily, we take the trivial bound $g_N^{K,2}(x_1,x_2)\geq 0$ and the lower bound in \eqref{eq:killed_green_smallh} follows via \eqref{eq:killed_g3} and \eqref{eq:decomp-kill}. In the case $r\geq \lambda h_N (\geq 2 h_N)$, we have in particular that $B^2(y_0,\frac{3}{4}r)\times \Z 
    \subset B^{\text{proj}}(x_0,r)$. Abbreviating $K^2= \Z^2 \setminus B^2(y_0,\frac{3}{4}r)$ and $M_{\lambda}= \tfrac\lambda{10} (|y_1-y_2|\vee h_N)^2$, it follows that 
    \begin{align*}
    g_N^{K,2}(x_1,x_2) 
        \geq 
    \frac{1}{\lambda_{x_1}}\int_{M_{\lambda}}^{(\frac{3}{4}r)^2}\Pzd{y_1} (Y_t=y_2,t<H_{K^2}) e^{-\frac{t}{N^2}}
        \sum_{k\in \mathbb{Z} \setminus\{0\}}  \Pz{\hat{z}_1}(Z_t=\hat{z}_2+  kh_N) \,dt.
    \end{align*}
    In addition, by means of a similar adaptation of \eqref{eq:prop_lclt_approx_z} with $d=2$ as in \eqref{eq:killed_int_large_t} above, it follows from the previous display and \eqref{eq:approx_sum_gaussian} that
    \begin{align*}
    \begin{split}
        g_N^{K,2}(x_1,x_2) 
        &\geq 
    \frac{1}{\lambda_{x_1}}\frac{1}{h_N}\int_{ M_{\lambda}}^{(\frac{3}{4}r)^2}\Pzd{y_1} (Y_t=y_2,t<H_{K^2}) e^{-\frac{t}{N^2}}
    \left[\sum_{k\in \mathbb{Z} \setminus \{0\}} h_N (\frac{3}{2\pi t})^{1/2} e^{ - \frac{3|\hat{z}+ kh_N|^2}{2t}}\right]\,dt -  \frac{C}{h_N^3 \sqrt{\lambda}} \\
    &\geq 
    \frac{c}{h_N} \left(\int_{M_{\lambda}}^{(\frac{3}{4}r)^2}\Pzd{y_1} (Y_t=y_2,t<H_{K^2}) e^{-\frac{t}{N^2}}\,dt - \frac{C}{\sqrt{\lambda}}  \right).
    \end{split}
    \end{align*}
    To conclude, we apply \cite[Proposition 5.26]{barlow_random_2017} with $\alpha=\beta=2$ to the probability involving $Y_t$ in the last display, which yields, for all $\lambda \geq 2$,
    \begin{multline*}
        \int_{M_{\lambda}}^{(\frac{3}{4}r)^2} \frac{1}{t} 
        e^{-C \frac{|y_1-y_2|^2}{t}}  \,dt     
        =
       \int_{\frac{C|y_1-y_2|^2}{(3r/4)^2}}^{\frac{10 C |y_1-y_2|^2}{\lambda(|y_1-y_2|\vee h_N)^2}} \frac{e^{-u}}{u}\,du \\
        \geq {e^{-\frac{C'|y_1-y_2|^2}{(|y_1-y_2|\vee h_N)^2}}}\log\left(\frac{10(3r/4)^2}{ \lambda (|y_1-y_2|\vee h_N)^2}\right)
        \geq c \log\left(\frac{3 r}{|y_1-y_2|\vee h_N}\right),
    \end{multline*}
    where we used $e^{-\frac{C|y_1-y_2|^2}{(|y_1-y_2|\vee h_N)^2}}\geq e^{-C}\geq c$ and $r^2/\lambda \geq r$ (since $r \geq \lambda h_N \geq \lambda$) in the last step. Since the last bound is $\geq c$ uniformly in $\lambda$, returning to the previous display and choosing $\lambda \geq 20 \Cr{C:gauss-approx}^2$ large enough, we deduce that $g_N^{K,2}(x_1,x_2) 
        \geq c \log(\tfrac{ r}{|y_1-y_2|\vee h_N})$ for $r \geq \lambda h_N$, thus completing the proof.
\end{proof}

Using the killed estimates from Proposition~\ref{cor:killed_green}, we now prove lower deviations on the (killed) capacity of the range of the walk. Below for $T \geq 0$ (possibly random) we abbreviate $X_{[0,T]}= \{ x \in \mathbb{S}_N: x = X_s \text{ for some } 0 \leq s \leq T\}$ the range of $X$ until time $T$. The following result asserts that when $T$ is the exit time of $ B(x,R)$ for some $R \geq 1$ and  $x\in \mathbb{S}_N$, the capacity of $X_{[0,T]}$ under $P_x$ is proportional to the capacity of $B(x,R)$ with high probability; see e.g.~\cite{MR2819660, drewitz_geometry_2023, drewitz_critical_2023, prevost_first_2024} for results of this flavour in transient setups with polynomial decay of the Green's function.

\begin{proposition} \label{lem:cap_walk_space_2}
    There exists $\Cl[c]{rw_cap}>0$ such that for all $x\in \mathbb{S}_N$, $R\geq 1$ such that $R\leq \Cr{C:range}N$, $2\leq s\leq \frac{R}{2}$ and $K \in \{ B(x,2R)^{\text{c}}, \emptyset \}$, we have, with $T=H_{B(x,R)^{\text{c}}}$,
    \begin{equation}\label{eq:rw_cap_with_killing}
        P_x\left( \capacity_{N}^{K}\big(X_{[0,T]}\big)
        \leq  \frac{\Cr{rw_cap}}{s}\capacity_N^K\big(B(x,R)\big) \right)
        \leq Ce^{-cs}.
    \end{equation}
%    and
%    \begin{equation} \label{eq:rw_cap_without_killing}
%        P_x\left( \capacity_{N}\big(X_{[0,T]}\big)
%        \leq  \frac{\Cr{rw_cap}}{s}\capacity_N\big(B(x,R)\big) \right)
%        \leq Ce^{-cs}.
%    \end{equation}
\end{proposition}

Towards the proof of Proposition~\ref{lem:cap_walk_space_2}, we first isolate the following estimate.

\begin{lemma}
    For all $N\geq 1,r\geq 2$ such that $r\leq \Cr{C:range}N$, the following holds. For all $x_0\in \mathbb{S}_N$, $K=B(x_0,r)^{\text{c}}$,  $x \in B(x_0,r/2)$ and $s\in[0,r/2)$, we have
    \begin{equation} \label{eq:killed_hk_int}
        \int_{s^2}^\infty P_{x}\left(X_t=x, t < H_K\right)\,dt\leq 
        \frac{C}{s\vee 1}+
        \frac{C}{h_N}\log\left(\frac{r}{(s\vee h_N)\wedge \frac{r}{2}}\right).
    \end{equation}
\end{lemma}

\begin{proof}
    Note if $s=0$, $\int_{s^2}^\infty P_{x}\left(X_t=x, t < H_K\right)\,dt=\lambda_xg^K_N(x,x)$ hence \eqref{eq:killed_hk_int} is a directly consequence of Proposition \ref{cor:killed_green}.
    For $s>0$, let $K^{\text{proj}}=B^{\text{proj}}(x_0,r)^{\text{c}}$ as in \eqref{eq:def_bproj} and write
    \begin{align}\label{eq:killed_hk_decomp}
    \begin{split}
    P_{x}\left(X_t=x,  t < H_K\right)= \sum_{k \in \Z}P_{(y,z)}&\left((Y_t,Z_t)=(y,\hat z +kh_N),t<H_{K^{\text{proj}}}\right)e^{-t/N^2}.
    \end{split}  
    \end{align}
    It is clear that by \cite[(5.18)]{barlow_random_2017}, integrating over the term $k=0$ in \eqref{eq:killed_hk_decomp} contributes to the $Cs^{-1}$ in \eqref{eq:killed_hk_int}. The integral of the sum over $k \neq 0$ in \eqref{eq:killed_hk_decomp} can simply be bounded by $g^{K,2}_N(x,x)$ (cf.~\eqref{eq:def_g2K}) when $r<2h_N$ and the result in this case follows from \eqref{eq:killed_g2_ub}.
    When $r\geq 2h_N$, recall the definition of $I(a,b)$ from \eqref{eq:killed_integtal}. By picking $x=x_1=x_2$ in \eqref{eq:killed_integtal} and following \eqref{eq:killed_int_small_t} we have that $I((s\vee h_N)^2,r^2)\leq \frac{C}{h_N}\log(r/(s\vee h_N))$. Combining this with \eqref{eq:killed_int_large_t} and \eqref{eq:killed_int_very_small_t} we thus obtain \eqref{eq:killed_hk_int}.
\end{proof}

\begin{proof}[Proof of Proposition \ref{lem:cap_walk_space_2}]
    We give the proof for $K=B(x,2R)^{\text{c}}$, the case $K=\emptyset$ is similar, and simpler.
    Recall that $(\bar X_n)_{n\geq 0}$ denotes the discrete-time skeleton of $(X_t)_{t\geq 0}$. We may work with $\bar X$ since the set of vertices $x'\in B(x,R)$ visited by $\bar X$ before exiting $B(x,R)$ coincides with $X_{[0,T]}$. 
    
    We will first show that there exists $\Cr{rw_cap}>0$ such that for any $u\in B(x,R)$ and $2\leq s\leq \frac{R}{2}$,
    \begin{equation} \label{eq:rw_cap_main}
        P_u\left( 
        \capacity_{N}^{K}\big(\dX_{[0,R^2/s^2]} \big)\leq  \frac{\Cr{rw_cap}}{s}\capacity_N^K\big(B(x,R)\big) ,
        \dX_{[0,R^2/s^2]} \cap B(x,R)^{\text{c}}=\emptyset
        \right) \leq \frac{1}{2},
    \end{equation}
    where, with a slight abuse of notation, we identify $R^2/s^2$ with its integer part in the sequel.
    By the Markov inequality, the l.h.s.~of \eqref{eq:rw_cap_main} is bounded by 
    \begin{equation} \label{eq:rw_cap_markov}
         \frac{\Cr{rw_cap}}{s}\capacity_N^K\big(B(x,R)\big) \times  E_u\left[ 
        \capacity_{N}^{K}\big(\dX_{[0,R^2/s^2]} \big)^{-1} \1\big\{
        \dX_{[0,R^2/s^2]} \cap B(x,R)^{\text{c}}=\emptyset\big\}
        \right].
    \end{equation}
  %  For $x,x^\prime \in \mathbb{S}_N$, let $\dg_N^K(x,x^\prime)=E_x^K(\sum_{n\geq 0}\1\{\dX_n=x^\prime\})$ be the green's function of the discrete skeleton. 
  By applying the analogue of the variational principle \eqref{eq:cap_var} for $\capacity_N^K(\cdot)$ with
    $\mu(x)=\frac{s^2}{R^2}\sum_{p=0}^{R^2/s^2}\delta_{\dX_p}$ and using the fact that on the event $\{\dX_{[0,R^2/s^2]} \cap B(x,R)^{\text{c}}=\emptyset\}$, $\dX_{[0,R^2/s^2]}$ has the same law under $E_u$ and $E_u^K$, we get that the expectation in the last display is smaller than
    \begin{equation} \label{eq:rw_cap_var_form}
        \frac{s^4}{R^4}\sum_{i,j=0}^{R^2/s^2} \lambda_u E_u^K\left[ g^K_N(\dX_i,\dX_j) \1\{
            \dX_i,\dX_j\in B(x,R)\}
            \right].
    \end{equation}
    Now let $\hat{P}^K_\cdot$ be an independent copy of $P^K_\cdot$ converning the process $\hat{X}$. For $i<j$, with $\dg_N^K(x,x^\prime)=E_x^K(\sum_{n\geq 0}\1\{\dX_n=x^\prime\}) = \lambda_{x'} g_N^K(x,x')$ the Green's function of the discrete skeleton,  we have that,
    \begin{align}\label{eq:get_killed_hk}
    \begin{split}
         \lambda_u &E_u^K\left[ 
        \hat{E}_{\dX_i}^K\left[
        g^K_N(\hat{X}_0,\hat{X}_{j-i}) \1\{
        \hat{X}_{j-i}\in B(x,R)\}
        \right]
        \1\{
        \dX_i \in B(x,R)\}
        \right]\\
        &\leq \sup_{v\in B(x,R)} 
        E^K_v\left(\dg_N^K(v,X_{j-i}) \1\{
        \dX_{j-i}\in B(x,R)\}\right) 
       % =\sup_{v\in B(x,R)}\sum_{v^\prime \in B(x,R)} g_N^K(v,v^\prime) P^K_v(\dX_{j-i}=v^\prime) \\
        %&=\sup_{v\in B(x,R)} \sum_{n=0}^\infty \sum_{v^\prime \in B(x,R)} P^K_{v}(\dX_n=v^\prime) P^K_{v^\prime}(\dX_{j-i}=v)
        =\sup_{v\in B(x,R)} \sum_{n=j-i}^\infty P^K_{v}(\dX_{n}=v),
    \end{split}
    \end{align}
    where the last equality follows upon rewriting the expectation as $\sum_{n=0}^\infty \sum_{v^\prime} P^K_{v}(\dX_n=v^\prime) P^K_{v}(\dX_{j-i}=v^{\prime})$ with $v^\prime$ ranging in $B(x,R)$, using time reversal to exchange $v$ and $v^\prime$ in the last probability and applying the Markov property at time $n$. 
    Now let $N_{t}$ be an independent Poisson random variable with parameter $t$. With \eqref{eq:killed_hk_int} and an easy comparison bound $P_v^K(X_{n}=v)\geq cP_v^K(\dX_{n}=v)$ one has that for all $v\in B(x,R)$ and $i<j$,
    \begin{equation} \label{eq:apply_killed_hk}
        \sum_{n=j-i}^\infty P^K_{v}(\dX_{n}=v)
        \leq C\int_{j-i}^\infty P^K_{v}(X_t=v)\,dt \leq C\left(\frac{1}{\sqrt{j-i}}+\frac{1}{h_N}\log\left(\frac{2R}{(\sqrt{j-i}\vee h_N)\wedge R}\right)\right).
    \end{equation}
    Hence combining Proposition \ref{cor:killed_green} to deal with the on-diagonal case $i=j$, \eqref{eq:get_killed_hk} and \eqref{eq:apply_killed_hk} we get that
    \begin{align}
    \begin{split}
        \eqref{eq:rw_cap_var_form} &\leq 
        C\frac{s^2}{R^2}
        \left(1+  \frac{1}{h_N} \black \log\left(\frac R{h_N} \vee 2 \right)+\sum_{n=1}^{R^2/s^2}\frac{1}{\sqrt{n}}+\frac{1}{h_N}\log\left(\frac{2R}{(\sqrt{n}\vee h_N)\wedge R}\right) 
        \right)\\
        &\leq  C\left( \frac{s}{R}+ \frac{1}{h_N}\log\left(\frac{2R}{((R/s)\vee h_N)\wedge R}\right)\right).
    \end{split}
    \end{align}
    Together with \eqref{eq:rw_cap_markov} and Corollary \ref{cor:killed_cap} we thus obtain that the probability in \eqref{eq:rw_cap_main} is bounded by 
    \begin{align*}
    %    &P_u\left( 
    %    \capacity_{N}^{K}\big(\dX[0,R^2/s^2] \big)\leq  \frac{c}{s}\capacity_N^K\big(B(x,R)\big) ,
    %    \dX[0,R^2/s^2] \cap B(x,R)^{\text{c}}=\emptyset
    %    \right) \\
       % \leq
        &\Cr{rw_cap}\times C \frac{\frac{1}{s}\left(R\wedge \frac{h_N}{\log(2)}\right)}{(R/s)\wedge \frac{h_N}{\log\left(\frac{2R}{((R/s)\vee h_N)\wedge R}\right)}} 
        \leq \Cr{rw_cap}\times C \frac{\left((R/s)\wedge \frac{h_N}{s\log(2)}\right)}{\left((R/s)\wedge \frac{h_N}{\log(2s)}\right)}  \leq  \Cr{rw_cap}\times C ,      
    \end{align*}
    hence by picking $\Cr{rw_cap}$ to be small enough we have proven \eqref{eq:rw_cap_main}.

Let $\pi_u(t)$ refer to the probability in \eqref{eq:rw_cap_main} with both instances of $\dX_{[0,R^2/s^2]}$ replaced by $\dX_{[0,t]}$. To finish, we use \eqref{eq:rw_cap_main} and apply the Markov property iteratively to get
   $$
    \pi_x\big(R^2/s\big) \leq  \pi_x\big(R^2/s^2\big) \times \textstyle \sup_{u} \pi_u\big((s-1)R^2/{s^2}\big) \leq 2^{-s},
$$
with $u$ ranging in $B(x,R)$. This completes the proof since $P_x\left(\dX[0,R^2/s] \cap B(x,R)^{\text{c}}\neq\emptyset \right)\leq  C\exp{\{-c s\}}$ as follows from \cite[(2.51)]{lawler_random_2010}. 
\end{proof}

\section{Upper bounds} \label{sec:percolation}

In this section, we show the upper bound in Theorem~\ref{thm:critical_connect}, thus completing the proof of \eqref{eq:critical_one_arm}; for the complementary lower bound, see the end of Section~\ref{sec:lower_bounds}. In certain sub-regimes of parameters $h_N, R, N$, several alternative approaches are possible. To be expedient, we present here an argument which works uniformly for \textit{all} choices of $h_N, R$ and $N$ satisfying the standing assumptions of Theorem~\ref{thm:critical_connect}; recall these entail that $R \leq \Cr{C:range}N$ for an arbitrary (large) finite constant $\Cr{C:range}$ and that $h_N=\floor{h(N)}$ for some non-negative increasing function $h(\cdot)$ with $1 \leq h(t) \leq t$.
These assumptions are tacitly assumed in the sequel. In particular, the proof below works equally well in the extreme cases $h_N=1$ (2-dimensional) and $h_N=N$ (3-dimensional), as well as any intermediate choice.

\medskip
To complete the proof of \eqref{eq:critical_one_arm} we will show the following result, which is effectively a stronger version of the upper bound. Recall $\Fbox(\cdot)$ from \eqref{eq:def_fbox}, $\capacity_N(\cdot)$ from the beginning of Section~\ref{sec:cap-slab} and let $\mathcal{C}$ denote the cluster of $0$ in $\{ \varphi \geq 0\}$.

\begin{theorem}\label{prop:one_arm}
    For any $N,R\geq 1$ such that $R\leq \Cr{C:range}N$ and $s\leq c$, we have
    \begin{equation} \label{eq:one_arm_exp}
        \P_N\left(0 \leftrightarrow\partial B_R,\capacity_N(\mathcal{C}\cap B_R)<s \Fbox(R)\right) \leq \arctan 
        \left[
        \left(\left(1\vee\frac{\log R}{h_N}\right)\Fbox( R)\right)^{-\frac{1}{2}} e^{-\frac{c}{s}}
        \right].
    \end{equation}
\end{theorem} 

The functional form of the bound on the right-hand side of \eqref{eq:one_arm_exp} (involving $\arctan$) naturally comes out of the proof. It is essentially owed to the fact that the probability on the left-hand side of \eqref{eq:one_arm_exp} involves a deviation for the capacity, which eventually makes it appear, cf.~\eqref{eq:caplaw} and \eqref{eq:arctan_calculus}. The specific form of the upper bound in \eqref{eq:one_arm_exp} is remarkable owing to the lower bound from Section~\ref{sec:lower_bounds}, which has the same form; cf.~also Theorem~\ref{prop:arm_general}, or Theorem~\ref{thm:arm_asymp}, which both exhibit this phenomenon, albeit for restricted ranges of parameters.

\medskip

Before proving Theorem~\ref{prop:one_arm}, let us show how it implies the upper bound in \eqref{eq:critical_one_arm}. This relies on the following estimate.

\begin{lemma}\label{lem:arctan_killed_unkilled}
For any $h(\cdot)$ (see above \eqref{def:slab}) and all $N,R\geq 1$, we have that
\begin{equation} \label{eq:arctan_killed_unkilled}
    \arctan \left[
    \left(\left(1\vee\frac{\log R}{h_N}\right)\Fbox( R)\right)^{-\frac{1}{2}} 
    \right] \leq C \left( \left(1\vee\frac{\log N}{h_N}\right)\Fbox( R) \right)^{-\frac{1}{2}}.
\end{equation}
\end{lemma}

\begin{proof}[Proof of Lemma \ref{lem:arctan_killed_unkilled}]
When $R\leq \sqrt{N}$ and $\frac{\log N}{h_N}> 1$, we have in particular that $\frac{R\vee h_N }{N} \leq N^{-1/2}$ and hence by \eqref{eq:def_fbox} we can bound the expression in brackets on the r.h.s~of \eqref{eq:arctan_killed_unkilled} as
\begin{equation*}
    \left(1\vee\frac{\log N}{h_N}\right)\Fbox(R)  \leq C \left(\frac{h_N}{\frac{1}{2}\log N} \vee  \frac{\log N}{\frac{1}{2}\log N}\right) \leq  C'
\end{equation*}
so the r.h.s.~of \eqref{eq:arctan_killed_unkilled} is $ \geq c$
which concludes the proof in this case since $|\arctan(x)|<\frac{\pi}{2}$ for all $x>0$. 
Now assuming $R>\sqrt{N}$ or $\frac{\log N}{h_N} \leq 1$. By \eqref{eq:arctan_bound} we have that the l.h.s.~of \eqref{eq:arctan_killed_unkilled} is bounded by $ \frac{\pi}{2} \big(\big(1\vee\frac{\log R}{h_N}\big)\Fbox( R)\big)^{-{1}/{2}}$. Hence it suffices to show that $\big(1\vee\frac{\log N}{h_N}\big)^{1/2}\big(1\vee\frac{\log R}{h_N}\big)^{-1/2}\leq C$.
Note that this is trivially true if $\frac{\log N}{h_N}\leq 1$. Now assuming $R>\sqrt{N}$ and $\frac{\log N}{h_N}> 1$, one has that in this case
\begin{equation*}
   \left(1\vee\frac{\log N}{h_N}\right)^{1/2}\left(1\vee\frac{\log R}{h_N}\right)^{-1/2}\leq  \left(\log N/\log(R)\right)^{1/2}\leq \sqrt{2}. 
\end{equation*}
\end{proof}

\begin{proof}[Proof of the upper bound in \eqref{eq:critical_one_arm}]
Fix $s>0$ such that Theorem \ref{prop:one_arm} holds. One has
\begin{equation} \label{eq:ub_two_terms}
    \theta_N^h(R) \leq 
    \P_N\left(0 \leftrightarrow\partial B_R, \capacity_N(\mathcal{C}\cap B_R)<s \Fbox(R)\right) +
    \P_N\left(s \Fbox(R) \leq \capacity_N(\mathcal{C})\leq \infty \right).
\end{equation}
By \eqref{eq:caplaw}, \eqref{eq:arctan_calculus}, \eqref{eq:one_arm_minus} and \eqref{eq:arctan_bound} the second term on the r.h.s.~of \eqref{eq:ub_two_terms} is bounded by $C(g_N(0)\Fbox(R))^{-1/2}$. By \eqref{eq:one_arm_exp} and \eqref{eq:arctan_killed_unkilled}, so is the first term, as \eqref{eq:green_estimate} implies that $ g_N(0) \leq C(1\vee \tfrac{\log N}{h_N}) $.
\end{proof}

We now prepare the ground for the proof of Theorem~\ref{prop:one_arm}. We first introduce a way to approximate $\mathbb{S}_N$ at a given scale $L\geq 1$.
 Let $\pi:\Z/h_N\Z \to \{0,1,\dots,h_N-1\}$ be the canonical map that picks the unique representative of each congruence class in $\{0,1,\dots,h_N-1\}$ .
    For all $L\geq 1$, let $m=\max(1,\lfloor h_N/L \rfloor)$, $r=h_N-mL$ the remainder of $h_N$ modulo $L$, and for $j\in\{0,\dots,m-1\}$, let $s_j=\sum^j_{i= 0 }(L+\1\{i\leq r\})-1$. These choices imply that $s_j \in \{0,1,\dots,h_N-1\} $ with spacing in $\{L,L+1\}$ and $s_{m-1}=h_{N}-1$. Now consider
    \begin{equation} \label{eq:Lambda_def}
    \Lambda_{\circ}(L)=\{z\in (\Z/h_N\Z): \pi(z) \in\{s_0,\dots,s_{m-1}\}\} \quad \text{and}\quad \Lambda(L)=L\Z^2\times   \Lambda_{\circ}(L).
    \end{equation}

\begin{lemma} \label{lem:renorm_la}
       For all $N,R \geq 1$, $L> 3$,
    \begin{align} &\label{eq:renorm_coverage}
       \textstyle \bigcup_{x \in \Lambda(L)} B(x,L) = \mathbb{S}_N,  \\ 
     &\label{eq:renorm_separation}
        B(x, L/2)\cap B(x^\prime, L/2)=\emptyset , \text{ for all } x\neq x^\prime \in \Lambda(L),\\
    &\label{eq:renorm_count}
        |\Lambda(L)\cap B(x,LR)| \leq CR^2 \times \textstyle \left(\frac{h_N\wedge (LR)}{L}\vee 1\right),  \text{ for all } x \in \Lambda(L).
    \end{align}
\end{lemma}

\begin{proof}
    For $\Z^2$, consider each square $y +( [\tfrac L2, \tfrac L2)\cap\Z)^2 $ anchored at $ y\in L\Z^2$. The distance from a lattice point $y^\prime$ in any such square to the closest anchor $y$ is at most $|y-y^\prime|^2\leq L^2/2$. For the vertical direction, it's easy to see that for any $z^\prime\in (\Z/h_N\Z)$, $\min_{z\in \Lambda_{\circ}(L)} |\hat z-\hat z^\prime|^2 \leq (L+1)^2/4$. Hence combining the previous two observations and the fact that $L>3$ we have that for any $x^\prime \in \mathbb{S}_N$, $\min_{x\in \Lambda(L)}\|x-x^\prime\|\leq \min_{x\in \Lambda(L)}\sqrt{ |y-y^\prime|^2 +|\hat z-\hat z^\prime|^2}<L$, which concludes the proof of \eqref{eq:renorm_coverage}. 
    
    Since $|y-y^\prime|\geq L$ for all $y\neq y^\prime \in L\Z^2$ and $d_{\Z/h_N\Z}(z,z^\prime)\geq L$  for all $z\neq z^\prime\in \Lambda_{\circ}(L)$,   \eqref{eq:renorm_separation} follows directly from the definition of $\Lambda(L)$  in \eqref{eq:Lambda_def}. Finally, for \eqref{eq:renorm_count}, we use the observation that $|\Lambda(L)\cap B(x^\prime,LR)|\leq|L\Z^2\cap B^{2}(y^\prime,LR)|\times |\Lambda_{\circ}(L)\cap B_{(\Z/h_N\Z)}(z^\prime,LR)|$. It's easy to see that there exists $C<\infty$ such that $|L\Z^2\cap B^{2}(y^\prime,LR)|\leq CR^2$. For the vertical direction, if $L\geq h_N$ then $|\Lambda_{\circ}(L)\cap B_{(\Z/h_N\Z)}(z^\prime,LR)|=1$; if $L< h_N$, then $|\Lambda_{\circ}(L)\cap B_{(\Z/h_N\Z)}(z^\prime,LR)|\leq 2\frac{h_N\wedge(LR)}{L}+1$.
\end{proof}

The following result, tailored to our purposes, follows from the Harnack inequality proved in Corollary~\ref{C:harnack} and a chaining argument using $\Lambda(\cdot)$ above. Recall $g_N^K$ from \eqref{eq:gNK}.

\begin{lemma}
    For all $N\geq 1$, $C \leq R\leq \Cr{C:range}N$, $x\in B_R^{\text{c}}$, $1>\bfa >\frac{C'}{R}$ and $K\subset B_{(1-\bfa)R}$,
    \begin{equation} \label{eq:harnack}
        \sup_{x^\prime \in\partial B(x,\bfa R/2)} g_N^K(x,x^\prime)\leq C\inf_{x^\prime \in\partial B(x,\bfa R/2)} g_N^K(x,x^\prime).
    \end{equation}

\end{lemma}

\begin{proof}
    Since for all $x\in B^{\text{c}}_R$ and $u\in 
    \partial B(x,\bfa R/2)$, $g^K(x,\cdot)$ is nonnegative and harmonic in $B(u,\bfa R/4)$, it follows from \eqref{eq:harnack_general} that 
    \begin{equation} \label{eq:harnack_ball}
        \sup_{v\in B(u,\bfa R/4)}g_N^K(x,v)\leq C'\inf_{v\in B(u,\bfa R/4)}g_N^K(x,v).
    \end{equation}
    We now proceed with a chaining argument and let $\{u_1,u_2,\dots ,u_n\}\subset \Lambda(\bfa R/4)$ be a collection of sites such that $n\leq  C$, $\partial B(x,\bfa R/2) \subset \bigcup_{i=1}^n B(u_i,\bfa R/4)$,  $B(u_1,\bfa R/4)\cap B(u_n,\bfa R/4)\neq \emptyset$, and $B(u_i,\bfa R/4)\cap B(u_{i+1},\bfa R/4)\neq \emptyset$ for all $i\in\{1,\dots,n-1\}$. Note the existence of such collection of sites follows from \eqref{eq:renorm_coverage} and \eqref{eq:renorm_count}. Hence for all $x_1,x_2 \in \partial B(x,\bfa R/2)$, there exists $i,j\in\{1,\dots,n\}$ such that $x_1\in B(u_i,\bfa R/4)$ and $x_2 \in B(u_j,\bfa R/4)$, and we may assume $i\leq j$ without loss of generality. Combining this and \eqref{eq:harnack_ball}, we have that
    \begin{equation*}
        \frac{g_N^K(x,x_1)}{g_N^K(x,x_2)}\leq 
        \frac{\sup_{v\in B(u_i,\bfa R/4)}g_N^K(x,v)}{\inf_{v^\prime \in B(u_j,\bfa R/4)}g_N^K(x,v^\prime)}
        \leq \prod_{m=i}^{j-1}\frac{\sup_{v\in B(u_m,\bfa R/4)}g_N^K(x,v)}{\inf_{v^\prime \in B(u_{m+1},\bfa R/4)}g_N^K(x,v^\prime)}\leq C.
    \end{equation*}

\end{proof}

We now borrow a notion of ``good obstacle'' $\mathcal{O}$ from \cite{drewitz_arm_2023}; see also \cite{RI-III, GRS24.1} for related concepts in the context of random interlacements. One point requiring slight care
is that the capacity lower bound inherent to the definition of a good obstacle set involves killing, which cannot be dispensed with (unlike in \cite{drewitz_arm_2023}) owing to its effect within our setup (cf.~for instance Proposition~\ref{cor:killed_green} and Theorem~\ref{thm:green_main}).

%One such $\mathcal{O}$ will be later shown to be hard

%\capacity_N(\mathcal{C}\cap B_R)

We say that $\pi=(x_i)_{1\leq i\leq M}$ is a
path in $\Lambda(L)$ from $0$ to $K\subset \mathbb{S}_N$ if $x_1=0$, $x_M\in K$, $x_i\in\Lambda(L)$ for all $1\leq i\leq M$, and for each $1\leq i\leq M-1$, there exist $x\in B(x_i,L)$ and $y\in B(x_{i+1},L)$ such that $x$ and $y$ are neighbours in $\mathbb{S}_N$. We call a set $\lO \subset \mathbb{S}_N$ a $(L,R,n,\kappa)$-good obstacle if for all paths $\pi$ in $\Lambda(L)$ from $0$ to $B_R^\text{c}$, there exists a set $A\subset \text{range}(\pi\cap B_R)$ with $|A|\geq n$ such that $\capacity_N^{B(x^\prime,2L)^{\text{c}}}(\lO\cap B(x^\prime,L))\geq \kappa$ for all $x^\prime \in A$.

Recall that $P_{x}^{\lO}$ denotes the law of the random walk with killing on $\lO$; see Section~\ref{eq:sec:kill} for notation. As we now breifly explain, by adapting the argument of \cite[Lemma 2.1]{drewitz_arm_2023}, one finds that $\Cl{large_ball}<\infty$ such that, for all $N,R, L\geq 1$ with $ \Cr{C:range}N \geq R \geq C$, $R\geq  L $, all $\frac{1}{2}>\bfa >\frac{C}{R}$, $\kappa >0$, integer $n\geq 1$, $x\in B_{\Cr{large_ball}R}^{\text{c}}$ and for any $(L,R,n,\kappa)$-good obstacle set $\lO\subset  B_{(1-\bfa)R}$,
    \begin{equation} \label{eq:no_hit_Obis}
        P_0^{\lO}\left(H_x<\infty\right) \leq  e^{-c\frac{\kappa n}{L\wedge h_N}}  \sup_{x^\prime \in \partial B(x,\bfa R/2)}  P_{x^\prime}^{\lO} ( H_x< \infty)
        ;
    \end{equation}
indeed \eqref{eq:no_hit_Obis} arises by looking at subsequent boxes centered in $\Lambda(L)$ visited by $X$ on its way to $x$, constituting $\pi$. Whenever $X$ enters a box $B=B(y,L)$ with center $y$ that lies in $A$, the good obstacle $\mathcal{O}$ has a significant presence, i.e.~one gets that $P_{y^\prime}(H_{\lO} < H_{B(y,2L)^\text{c}}) \geq    \frac{c\kappa }{L\wedge h_N}
$ for all $y^{\prime} \in B$ by combining the capacity lower bound inherent to $\mathcal{O}$ and \eqref{cor:killed_green}. Applying the Markov property multiple times and using that $|A| \geq n$ thus produces the desired exponential pre-factor, and \eqref{eq:no_hit_Obis} quickly follows.

The bound \eqref{eq:no_hit_Obis} quantifies that a good obstacle set $\mathcal{O}$ is hard to avoid for the random walk. In previous works, e.g.~\cite{drewitz_arm_2023, drewitz_critical_2024}, which assumed polynomial decay of the Green's function with some exponent $\nu$, the term $P_{x^\prime}^{\lO} ( H_x< \infty)$ on the right-hand side of \eqref{eq:no_hit_Obis} was never problematic (and simply bounded by $CR^{-\nu}$), essentially because the killing by $\mathcal{O}$ is no longer felt. Within the present setup (which permits recurrent behaviour at scale $R$), this  is more subtle. The estimates that we will need to deal with this issue are the content of the following lemma.

\begin{lemma} \label{lem:nohit_obstacle} Under the assumptions above \eqref{eq:no_hit_Obis}, and if $\frac{\kappa n}{L\wedge h_N}\geq \Cl{avoid_cost}$, with $B= B(x,\bfa R/2)$,
    \begin{equation} \label{eq:no_hit_O}
     \sup_{x^\prime \in \partial B}  P_{x^\prime}^{\lO} ( H_x< \infty) \leq C
        \inf_{x^\prime\in B} \frac{g^{\lO\cup\{0\}}_N(x^\prime,x)}{g^{\lO}_N(x,x)}.
            \end{equation}
       %  Let $\Cl[c]{scale_rl}\in(0,\frac{1}{48})$. For all $N,R\geq 1$ such that $R\leq \Cr{C:range}N$, $1>\bfa >\frac{C}{R}$, $2\bfd<\bfe\leq \frac{\bfa}{4}$ and $K\subseteq \Ann^{\bfa}_R$, one has that for all $x^\prime \in \Ann_R^{\bfd,\bfe}$ and 
      Moreover, for some $\Cl[c]{scale_rl}\in(0,\frac{1}{48})$ and for any set $\lO\subset  B_{(1-\bfa)R}$, if $x\in \partial B_{\Cr{large_ball}R}$ then%and $x^\prime \in  B_{R}\setminus B_{(1-\bfb)R}$ with $\bfb\leq \frac{\bfa}{4}$,
    \begin{equation} \label{eq:Ade_out}
       \inf_{x^\prime} P_{x^\prime}^{\lO}\left(H_x<\infty\right) \geq
        ce^{-C\alpha^h_N} \times
        \left(1\wedge \frac{h_N}{\log(\bfa R)}\right)\inf_{u\in B^x} g_N^{\lO}(u,x) ,
    \end{equation}
    where $\alpha^h_N := (\frac{1}{\bfa})^2\times(\frac{R\wedge h_N}{\bfa R})$, $B^x\coloneqq B({x,\Cr{scale_rl}\bfa R})$ and the infimum is over $x^\prime \in  B_{R}\setminus B_{(1-\frac{\bfa}{4})R}$.
\end{lemma}

\begin{proof} One knows that $ P_{x^\prime}^{\lO} ( H_x< \infty) =  {g^{\lO}_N(x^\prime,x)}/{g^{\lO}_N(x,x)}  $ as a direct consequence of \eqref{eq:last_exit} with $A=\{ x\}$ (more precisely, its analogue for the process $P_{x^\prime}^{\lO}$, which involves $g^{\lO}$). We show that
    \begin{equation} \label{eq:no_zero}
        g^{\lO}_N(x^\prime,x) \leq Cg^{\lO\cup \{0\}}_N(x^\prime,x), \quad \forall x^\prime \in \partial B .
    \end{equation}
    If \eqref{eq:no_zero} holds, then \eqref{eq:harnack} yields that $g^{\lO\cup \{0\}}_N(x^\prime,x)
        \leq C\inf_{z \in \partial  B}g^{\lO\cup \{0\}}_N(z,x)$ and the infimum can be taken over $z \in B$ instead  by application of the maximum principle (see, e.g.~\cite[Theorem 1.37]{barlow_random_2017}) since $g^{\lO\cup \{0\}}_N(\cdot,x)$ is harmonic on $B(x,\bfa R/2)\setminus\{x\}$ and bounded (by $ g^{\lO\cup \{0\}}_N(x,x)$). This gives \eqref{eq:no_hit_O}.
        
        To obtain \eqref{eq:no_zero}, one writes $g^{\lO\cup \{0\}}_N(x^\prime,x) =g^{\lO}_N(x^\prime,x) - g_N^{\lO}(0,x^\prime )P_0(H_x<H_{\lO}).$
    Since one has that $g^{\lO}_N(x^\prime,x)/g^{\lO}_N(x^\prime,x^\prime)=P_{x^\prime}^{\lO}(H_x<\infty)$ and $g_N^{\lO}(0,x^\prime )/g^{\lO}_N(x^\prime,x^\prime)=P_{0}^{\lO}(H_{x^\prime}<\infty)$, it suffices to show $P_0(H_x<H_{\lO}) \leq \frac{1}{2}P_{x^\prime}^{\lO}(H_x<\infty)$. By applying the strong Markov property at time $H_{B}$ and using \eqref{eq:harnack} for the function $P_{\cdot}^{\lO}(H_x<\infty)$ on $\partial B$ we get that
    \begin{equation*}
        P_0(H_x<H_{\lO}) \leq \sup_{u \in \partial B}P_{u}^{\lO}(H_x<\infty)P_{0}^{\lO}(H_{B}<\infty)
        \leq CP_{x^\prime}^{\lO}(H_x<\infty) e^{-c\frac{\kappa n}{L\wedge h_N}},
    \end{equation*}
    where we used the same argument as described below \eqref{eq:no_hit_Obis} to get the exponential term in the last inequality. Hence by choosing $\frac{\kappa n}{L\wedge h_N}\geq \Cr{avoid_cost}$ with $\Cr{avoid_cost}$ large enough we obtain \eqref{eq:no_zero}.

As to \eqref{eq:Ade_out}, assume $\Cr{scale_rl}<\frac{1}{12}$ and let $B_v= B(v,\tfrac13\Cr{scale_rl} \bfa R)$, $K_v=B(v,\frac{1}{8} \bfa R)^{\text{c}}$ and note that by Proposition \ref{cor:killed_green} and Corollary \ref{cor:killed_cap}, we have that for all $u,v,\in \mathbb{S}_N$ such that $u\neq v$ and $\|u-v\|\leq \Cr{scale_rl}\bfa R$,
    \begin{equation} \label{eq:step_nohit_inner_annulus}
        P_u\left(H_{B_v} < H_{K_v}\right) \geq \capacity^{K_v}_N(B_v) \times \inf_{v^\prime \in B_v }g_N^{K_v}(u,v^\prime) \geq c.
    \end{equation}
   We utilise (\ref{eq:step_nohit_inner_annulus}) to make a connection from $x^\prime$ to $x$. 
    Let $x^\prime \in  B_{R}\setminus B_{(1-\frac{\bfa}{4})R}$. Then there exists, for some integer $p^\prime\geq1$, a nearest-neighbour path  $\pi=(x_i)_{1\leq i \leq p^\prime} \subset \mathbb{S}_N$ such that $x_1=x^\prime, x_{p^\prime}=\bar{x}$ where $\bar{x}$ is a vertex in $B_x$ and  $x_i \notin B_{(1-\frac{\bfa}{4})R}$ for all $1\leq i\leq p^\prime$.   
    Hence, $\dist(K_{x_i}^{ \text{c}} , B_{(1-\bfa)R})\geq R(\bfa-\frac{\bfa}{4})-\bfa R/8\geq \frac{5}{8}\bfa R$ for all $1\leq i\leq p^\prime$. In particular,  $K_{x_i}^{ \text{c}}\cap B_{(1-\bfa)R}=\emptyset$ for all $i$.
    Now let $v_1=x^\prime$ and for each $k\geq 1$, define recursively $v_{k+1}$ as the first vertex in $\Lambda(\lfloor \tfrac13\Cr{scale_rl}\bfa R \rfloor)$ such that $B_{v_{k+1}}$ is visited by $\pi$ after exiting $B_{v_{k}}$. We denote by $p$ the integer such that $x\in B(v_p,\tfrac23\Cr{scale_rl}\bfa R)$.  
    In view of \eqref{eq:renorm_count}, we can assume that $p\leq C\alpha^h_N$, hence by combining \eqref{eq:step_nohit_inner_annulus} and the Markov property we have that
   $ P_{x^\prime}^{\lO}\left(H_x<\infty\right) $ is bounded from below by $\exp\{-C\alpha^h_N\}
        \inf_{u} P_u^{\lO}(H_x<\infty)$, where the infimum ranges over $u\in B_{v_p}$.  We may now conclude since by Corollary \ref{cor:killed_cap} applied with $r=1$, $r'=c \bfa R$, and abbreviating $D\coloneqq  B({x,\Cr{scale_rl}\bfa R})$,
    \begin{align*}
    \begin{split}
       \inf_{u\in B_{v_p}} P_u^{\lO}(H_x<\infty) 
       &\geq \capacity_N^{K_{v_p}}(\{x\})\inf_{u\in D}  
       g_N^{\lO}(u,x) 
       \geq c \left(1\wedge \frac{h_N}{\log(\bfa R)}\right)\inf_{u\in D} g_N^{\lO}(u,x)  .     
    \end{split}
    \end{align*}
\end{proof}

The remainder of this section is concerned with the proof of Theorem~\ref{prop:one_arm}. The argument is modelled after the proof of a similar result from \cite{drewitz_critical_2024}, and we refer to it whenever possible; see also \cite{cai_one-arm_2024, werner2025switchingidentitycablegraphloop} for possible alternative routes. The results of \cite{drewitz_critical_2024} apply to a class ``low-dimensional'' transient graphs with polynomial volume growth, and polynomial decay of the Green's function. These conditions are too restrictive for our purposes. The following lemma will eventually follow by application of \cite[Proposition 2.1]{drewitz_critical_2024}, itself an extension of \cite[Proposition 5.2]{lupu_loop_2016}. Recall that $\mathcal{C}$ denotes the cluster of $0$ in $\{ \varphi \geq 0\}$.

%For $R\geq 1$, we now introduce the following annuli for $\bfa \in(0,1),\bfb\in(\bfa,1)$,
%\begin{equation*}
%    \Ann_R^{\bfa,\bfb} \stackrel{\text{def.}}{=}  B_{(1-\bfa)R}\setminus B_{(1-\bfb)R} \quad \text{and} \quad \Ann_R^{\bfa} \stackrel{\text{def.}}{=}  B_{(1-\bfa)R}.
%\end{equation*}

\begin{lemma} \label{lem:hit_annulus}
    For all $N,R\geq 1$ such that $R\leq \Cr{C:range}N$,
    $\frac{2}{3}>\bfb>2\bfa >\frac{C}{R}$
    , $\frac{\bfa}{4} \geq \bfe> 2\bfd$, $s>0$ and $\lO\subset \Ann^{\bfa,\bfb}_R \coloneqq B_{(1-\bfa)R}\setminus B_{(1-\bfb)R}$, one has    \begin{multline} \label{eq:hit_annulus}
       \P_N^{\lO}  \big(\capacity_N(\mathcal{C}\cap \Ann^{\bfd,\bfe}_R)\geq s \Fbox((\bfe-\bfd)R) \big)  
       \leq
     \arctan\left[C \left(s\big(1\vee\tfrac{\log R}{h_N}\big)\Fbox(\bfd R)\right)^{-\frac{1}{2}} \times \rho
        \right] ,
    \end{multline}
    where $\rho \coloneqq \inf_{x \in \partial \hat{B}} \frac{g^\lO_N(x,x)e^{C\alpha^h_N }P^\lO_0(H_x<\infty)}{\inf_{x^\prime \in B^x} g_N^{\lO\cup \{0\}}(x^\prime,x)}$ with $B^x=B({x,\Cr{scale_rl}\bfa R})$, $\hat{B}= B_{\Cr{large_ball}R}$. 
\end{lemma}

\begin{proof}[Proof of Lemma \ref{lem:hit_annulus}]
With a view towards applying \cite[Proposition 2.1]{drewitz_critical_2024}, which regards the tail of $  g_{N}^{\lO\cup \{0\}}(x)-g_{N}^{\lO\cup \mathcal{C}}(x)
$, we need to find a good (i.e.~as large as possible) lower bound on this random variable on the event in \eqref{eq:hit_annulus}. We start with the following random walk estimate. Let $x\in \partial \hat{B}= \partial B_{\Cr{large_ball}R}$ (w.l.o.g.~assume $\Cr{large_ball} \geq 2$), $\lO_0= \lO\cup \{0\}$ so Proposition \ref{cor:killed_green} yields 
    \begin{equation}\label{eq:green-lb-pf-final}
        g^{\lO_{0}}_N(x,x)\geq g^{(B_{ R})^\text{c}}_N(0,0) \geq c\left(1\vee({\log( R)}/{h_N})\right).
    \end{equation}
Now suppose $K \subset \Ann^{\bfd,\bfe}_R$ is such that $K$ is compact, has finitely many components and $\capacity_N(K)\geq s\Fbox((\bfe-\bfd)R)$.
    First note that for all $x^\prime \in K$, by \eqref{eq:green-lb-pf-final} and \eqref{eq:Ade_out}, we have that
    \begin{align*}
    \begin{split}
     g^{{\lO_{0}}}_N(x,x^\prime)
        =g^{{\lO_{0}}}_N(x,x)P_{x^\prime}^{{\lO_{0}}}(H_x<\infty )  
        &\geq c\exp\{-C\alpha^h_N\}\inf_{u\in B^x} g_N^{\lO_{0}}(u,x)  .   
    \end{split}
    \end{align*}
    Hence for all $x\in \partial B_{\Cr{large_ball}R}$, by a last exit decomposition and using the lower bound on $\capacity_N(K)$,    \begin{align} \label{eq:hit_outer_ann}
    \begin{split}
        P_x^{{\lO_{0}}} (H_{K}<\infty) 
        &%\geq s\Fbox((\bfe-\bfd)R) \inf_{x^\prime \in K^\prime} g^{{\lO_{0}}}_N(x,x^\prime)
         \geq cs
        \Fbox(\bfd R)\exp\{-c^\prime\alpha^h_N\}\inf_{u\in B^x} g_N^{\lO_{0}}(u,x)  .
    \end{split}
    \end{align}  On the event $\{\capacity_N(\mathcal{C}\cap \Ann^{\bfd,\bfe}_R)\geq s \Fbox((\bfe-\bfd)R)\}$, one thus finds the following lower bound for all $x\in \partial \hat{B}$ by applying %\cite[(2.6)]{drewitz_critical_2024} to the graph $\mathbb{S}_N\setminus K$,
    firs the Markov property at time $H_{\mathcal{C}}$, abbreviating $T= H_{\mathcal{C}\cap \Ann^{\bfd,\bfe}_R}$
    \begin{multline*}
        g_{N}^{\lO_{0}}({ x,x})-g_{N}^{\lO\cup \mathcal{C}}(x,x)
        \!=\! E_x^{\lO_{0}}\left[g_N^{\lO_{0}}(X_{H_{\mathcal{C}}},x)\1\{H_{\mathcal{C}}<\infty\}\right] 
        \!\geq\! c E_x^{\lO_{0}}\left[ g_N^{\lO_{0}}(x,x)P_{X_{T}}^{\lO_{0}}(H_x<\infty) \1\{T<\infty\}\right] \\
        \geq 
        ce^{-C\alpha^h_N}\inf_{u\in B^x} g_N^{\lO_{0}}(u,x)
        \times P_x^{\lO_{0}}(T< \infty)
        \geq cs
        \Fbox(\bfd R)
        e^{-C^\prime\alpha^h_N}
        \left(\inf_{u\in B^x} g_N^{\lO_{0}}(u,x)\right)^2
        ,
    \end{multline*}
    where the first inequality follows by last-exit decomposition, the second inequality follows from \eqref{eq:green-lb-pf-final} and \eqref{eq:Ade_out} and the last inequality follows from \eqref{eq:hit_outer_ann}. Combining this with \cite[Proposition 2.1]{drewitz_critical_2024} applied to the metric graph $\tSlab\setminus \lO$, we obtain that $   \P_N^{\lO}  \big(\capacity_N(\mathcal{C}\cap \Ann^{\bfd,\bfe}_R)\geq s \Fbox((\bfe-\bfd)R) \big)  
$ is bounded by
    \begin{multline*}
 %       &\P_N^K\left(\capacity_N(\mathcal{C}\cap \Ann^{\bfd,\bfe}_R)\geq s \Fbox((\bfe-\bfd)R) \right)\\
   %     \leq 
         \P_N^{\lO}\left(
        g_{N}^{\lO_{0}}(x,x)-g_{N}^{\lO\cup \mathcal{C}}(x,x)\geq 
        cs
        \Fbox(\bfd R)e^{-C^\prime\alpha^h_N }\big( \textstyle\inf_{u\in B^x} g_N^{\lO_{0}}(u,x)\big)^{2}
        \right)\\
        =\frac{1}{\pi}\arctan\left(\left(g^{\lO}_N(0,0) cs
        \Fbox(\bfd R) \right)^{-\frac{1}{2}}\frac{g^{\lO}_N(0,x)e^{C\alpha^h_N }}{\textstyle\inf_{u\in B^x} g_N^{\lO_{0}}(u,x)} \right) 
    \end{multline*}
    We may conclude upon noting $g^{\lO}_N(0,x)= g^{\lO}_N(x,x)P^{\lO}_0(H_x<\infty)$
    and using \eqref{eq:green-lb-pf-final} and a similar estimate yielding $g^{\lO}_N(0,0) \geq c(1\vee({\log (R)}/{h_N}))$ .% or write without arctan
%    \begin{align*}
%    \begin{split}
%        c\P_N^K&\left(\capacity_N(\mathcal{C}\cap \Ann^{\bfd,\bfe}_R)\geq s \Fbox((\bfe-\bfd)R) \right)\\
%        &\leq C\left(\left(\sqrt{\frac{\left(1\vee\frac{\log(\bfa R)}{h_N}\right)\left((\bfa R)\wedge h_N\right)^2}{cs\Fbox(\bfd R)}}e^{C\alpha^h_N }\inf_{x\in\partial B_R}P^K_0(H_x<\infty)\right)\vee 1\right) .       
%    \end{split}
%    \end{align*}
\end{proof}

\begin{lemma} \label{lem:pre_recursion}
    For all $N\geq 1$, $C \leq R\leq \Cr{C:range}N$, $\frac{2}{3}>b>2\bfa $, $ 2\bfd<\bfe\leq \frac{\bfa}{4}$ and $s,t>0$,
    \begin{align} \label{eq:pre_recursion}
    \begin{split}
        &\P_N \left( \capacity_N(\mathcal{C}\cap \Ann^{\bfa,\bfb}_R)\leq t \Fbox((\bfb-\bfa)R),\capacity_N(\mathcal{C}\cap \Ann^{\bfd,\bfe}_R) \geq s \Fbox((\bfe-\bfd)R) \right) \\
        &\leq \arctan 
        \left[
        \psi_s e^{C^\prime\alpha^h_N-\Cl[c]{exp_const}t^{-1}}
        \right], \quad \psi_s \equiv C \left(s\left(1\vee\tfrac{\log R}{h_N}\right)\Fbox(\bfd R)\right)^{-\frac{1}{2}} .
    \end{split}
    \end{align}
\end{lemma}

\begin{proof}
    We use the isomorphism \cite{lupu_loop_2016} with the loop soup $\cL$ at intensity $1/2$ on the metric graph $\tSlab$. Recall $\cL$ is a Poisson point process of Markovian loops on $\tSlab$ defined under an auxiliary probability $\Q_N$ with intensity measure $\frac{1}{2}\mu_N$. Let $\cC= \cC(\cL)$ be defined as the empty set with probability $1/2$ (under $\Q_N$), or otherwise denotes the cluster of $0$ in $\cL$. By the isomorphism $\cC$ has the same law under $\Q_N$ and $\P_N$.

    Let $\cL^{\text{big}}\subset \cL$ be obtained from $\cL=\sum_i \delta_{\gamma_i}$ by retaining only \emph{big} loops in the annulus $\Ann^{\bfa,\bfb}_R$, i.e. loops $\gamma_i$ which satisfy $\capacity_N(\text{range}(\gamma_i)\cap \mathbb{S}_N)>t\Fbox((\bfb-\bfa)R)$ and for which $\text{range}(\gamma_i)\cap \mathbb{S}_N\subset \Ann^{\bfa,\bfb}_R$. Then on the event $\capacity_N(\cC\cap\Ann^{\bfa,\bfb}_R)\leq t\Fbox((\bfb-\bfa)R)$, $\mathcal{C}$ has the same law under $\P_N$ as $\cC(\cL\setminus\cL^{\text{big}})$. Let $\lO$ be the intersection of $\mathbb{S}_N$ and of the range of all the loops comprising $\cL^{\text{big}}$. For $\delta>0$ to be fixed later, let
    \begin{equation}\label{eq:L-ell}
        L\stackrel{\text{def.}}{=} t(\bfb-\bfa)R/\delta, \text{ and } 
        \ell \stackrel{\text{def.}}{=} \floor{{(\bfb-\bfa)R}/{5L}}
        = \floor{{\delta}/{5t}}.
    \end{equation}
    Note that we can w.l.o.g.~assume that $t\leq c(\delta)$ and $\delta\leq c^\prime$ (and hence $\ell,L\geq 1$) since otherwise the statement is either trivial or the claim follows easily from \eqref{eq:arctan_calculus}, \eqref{eq:caplaw}, \eqref{eq:one_arm_minus} and $g_N(0)\geq c(1\vee\frac{\log R}{h_N})$. Let $\bfG$ be the event that $\lO$ is a $(L,R,\ell/2,\delta(L\wedge h_N))$-good obstacle set. Now by \eqref{eq:hit_annulus}, the restriction property for the loop soup (\cite[Theorem 6.1]{fitzsimmons_markovian_2014}) and the isomorphism on $\tSlab \setminus \lO$, we have that the intersection of $\bfG$ with the 
 event in the first line of \eqref{eq:pre_recursion} has $\Q_N$-probability bounded by  \begin{equation}\label{eq:final1}
   % \begin{split}
      \E^{\Q_N}\left[
        \P_N^{\lO}
        \left( 
        \capacity_N(\mathcal{C} \cap \Ann^{\bfd,\bfe}_R) 
        \geq s \Fbox((\bfe-\bfd)R)
        \right)
        \1_{\bfG}\right] 
         \leq 
        \E^{\Q_N}\big[
       \arctan\left( \psi_s  \rho
        \right) 
        \1_{\bfG}\big],
    \end{equation}
   We now bound $\rho$, as defined in Lemma~\ref{lem:hit_annulus}, which is random and depends on $\lO$. Since $\lO$ is a $(L,R,\ell/2,\delta(L\wedge h_N))$-good obstacle set on the event $\bfG$, the bound \eqref{eq:no_hit_Obis} applies to the hitting probability $P^\lO_0(H_x<\infty)$ appearing as part of $\rho$. Combined with  \eqref{eq:no_hit_O}, and since $\frac{(\ell/2)\delta (L\wedge h_N) }{L\wedge h_N}\geq \frac{\delta^2}{5t}$ and we can pick $t\leq c(\delta)$ small enough such that $\frac{\delta^2}{5t}\geq C$, it yields that $\rho \leq e^{C\alpha^h_N-c\delta^2/t}$. Plugging this into the previous display gives the desired bound \eqref{eq:pre_recursion} on the event $\bfG$.

It remains to control the target event on $\bfG^{\text{c}}$, in the course of which we will also pick $\delta$.
    Let us denote by $\cP$ the set of tuples $\tau=(x_1,\dots,x_\ell)$ such that for all $i\in\{1,\dots,\ell\}$, $x_i\in \Lambda(L)$, $B(x_i,L)\subset \Ann^{a,b}_R$ $x_{(i+1)\wedge\ell}\in B(x_i,5L)$, and for all $i\neq  j\in\{1,\dots,\ell-1\}$, $B(x_i,L)\cap B(x_j,L)=\emptyset $. 
    From now on consider $\delta < \frac{1}{32}$. We fix an $i\in \{1,\dots,\ell\}$, and consider the set $\tilde{\Lambda}_i=B(x_i,L)\cap\Lambda(\ceil{32\delta L})$.
    Note that by definition of $\Lambda(\ceil{32\delta L}))$, see \eqref{eq:renorm_separation}, the boxes $B(x,16\delta L)$ are disjoint for $x\in \tilde{\Lambda}_i$. We then define the set $\Lambda_i\subset \tilde{\Lambda}_i$ by only keeping the vertices $x$ such that $B(x,16\delta L)\subset B(x_i,L)$.

    Given $i\in\{1,\dots,\ell\}$ and $x\in \Lambda_i$, a loop $\gamma$ is called $(i,x)$-good if  $\text{range}(\gamma)\subset B(x,16\delta L)$ and both,
    \begin{align}\label{eq:godd2cap}
        &\capacity_N(\text{range}(\gamma)\cap \mathbb{S}_N)>t\Fbox((\bfb-\bfa)R) , \quad
       \capacity_N^{B(x_i,2L)^{\text{c}}}(\text{range}( \gamma)\cap \mathbb{S}_N) \geq  \delta(L\wedge h_N).
       %\\
    \end{align}
    We write $D_\tau$ for the set of $i\in\{1,\dots ,\ell\}$ such that there exists a loop in $\cL^\text{big}$ that is $(i,x)$-good for some $x\in \Lambda_i$. It is easy to see that any path $\pi$ in $\Lambda(L)$ from $0$ to $B_R^{\text{c}}$ contains a tuple $\tau\in\cP$. Hence on $\bfG^{\text{c}}$, there must exist a tuple $\tau\in\cP$ such that $|D_\tau|< \ell/2$. 
    Since $B(x,16\delta L)$ are disjoint by definition for all $x\in\Lambda_i$, the events $E_{i,x}=\{\exists \gamma \in \cL^{\text{big}}:\gamma \text{ is }(i,x)\text{-good}\}$ are i.i.d. Hence by letting $p(\delta)=\inf_i \Q_N( \bigcup_{ x\in \Lambda_i} E_{i,x})$ and noting that $|\Lambda_i|\geq c\frac{1}{\delta^2}\times (\frac{L\wedge h_N}{\delta L}\vee 1)\geq \frac{c}{\delta^2}$ (which follows from Lemma \ref{lem:renorm_la}), we have that,
        \begin{equation} \label{eq:binom_dom}
        1-p(\delta) 
        \leq \sup_i 
        \left(
        \sup_{x\in\Lambda_i}
        \Q_N 
        \left(
        \nexists \gamma\in\cL^{\text{big}}, \gamma \text{ is }(i,x)\text{-good}
        \right)
        \right)^{\frac{c}{\delta^2}}.
    \end{equation}
    We write $Z^\gamma$ for the discrete skeleton of the loop $\gamma$ (see around \cite[(3.5)]{drewitz_arm_2023} for the precise definition of $Z^\gamma$). Let $K=B(x,16\delta L)^\text{c}$. For $i\in\{1,\dots ,\ell\}$ and $x\in\Lambda_i$, owing to \eqref{eq:godd2cap} we have that,
    \begin{align*} %\label{eq:calc_binom_dom}
    \begin{split}
       &\mu_N 
       \left(
       \gamma \text{ is } (i,x)\text{-good}
       \right) 
        \geq 
       \mu_N 
       \left(
       \tB(x,\delta L) \xleftrightarrow{\gamma} \tB(x,4\delta L),
       \gamma \text{ is }(i,x)\text{-good}
       \right) 
       \\ &\geq 
       \inf_{x^\prime \in\partial B(x, 4\delta L)}P_{x^\prime}^{K}
       \left(
        \begin{array}{l}
        H_{B(x,\delta L)}<\infty,\capacity_N\!\left(Z^\gamma_{[0,H_{B(x^\prime,2\delta L)^{\text{c}}}]}\right) 
        \geq t \Fbox((\bfb-\bfa)R), \\[6pt]
        \capacity_N^{B(x_i,2L)^{\text{c}}}\!\left(Z^\gamma_{[0,H_{B(x^\prime,2\delta L)^{\text{c}}}]}\right) 
        \geq \delta(L\wedge h_N)
        \end{array}
       \right) 
       \\ &\geq
       c\inf_{x^\prime \in\partial B(x, 4\delta L)}
       P_{x^\prime}^{K}
       \left(
        \begin{array}{l}
        \capacity_N\!\left(Z^\gamma_{[0,H_{B(x^\prime,2\delta L)^{\text{c}}}]}\right) 
        \geq t \Fbox((\bfb-\bfa)R), \\[6pt]
        \capacity_N^{B(x_i,2L)^{\text{c}}}\!\left(Z^\gamma_{[0,H_{B(x^\prime,2\delta L)^{\text{c}}}]}\right) 
        \geq \delta(L\wedge h_N)
        \end{array}
       \right)
      \inf_{u}
       P_{u}^{K}
       \left(
       H_{B(x,\delta L)}<\infty
       \right) ,
    \end{split}
    \end{align*}
with $u \in\partial B(x^\prime,2\delta L)$ in the last infimum.
    The first probability involving the two capacities the previous line is uniformly bounded away from $0$ by the fact that $t\leq c(\delta)$, $\delta<c^\prime$, a union bound, Proposition \ref{lem:cap_walk_space_2} and Corollary~\ref{cor:killed_cap}. For the second probability, it follows from Proposition \ref{cor:killed_green} and Corollary \ref{cor:killed_cap} that for all $x^\prime \in\partial B(x, 4\delta L),u \in\partial B(x^\prime,2\delta L)$,
    \begin{equation}\label{eq:loop_return}
        P_{u}^{K}
        \left(H_{B(x,\delta L)}<\infty \right)
        \geq \capacity_N^{K}(B(x,\delta L)) \inf_{
        u',        v}
        g^K_N(u',v)\geq c\frac{(\delta L)\wedge h_N}{(\delta L)\wedge h_N}\geq c,
    \end{equation}
    with the infimum ranging over $u' \in \partial B(x^\prime,2\delta L)$ and $ v \in B(x,\delta L)$.    Feeding (\ref{eq:loop_return}) into the previous display and combining with (\ref{eq:binom_dom}), it is clear that $1-p(\delta)\leq \exp\{-\frac{c}{\delta^2}\}$. Additionally, by \eqref{eq:renorm_count}, we have $|\cP|\leq \ell C^\ell \leq \tilde{C}^{\ell}$. Hence we can now deduce by taking $\delta=c$ for a small enough constant $c>0$ and a union bound that,
    \begin{equation*}
        \Q_N(\bfG^\text{c})\leq C^\ell \sup_{\tau\in\cP} \Q_N(|D_\tau|<\ell/2) \leq 2^\ell (1-p(\delta))^{\ell/2} \leq \exp\{-c\ell\}.
    \end{equation*}
    Therefore using \eqref{eq:arctan_calculus}, \eqref{eq:caplaw}, \eqref{eq:one_arm_minus} and the bound $g_N^{\lO}(0)\geq c(1\vee\frac{\log R}{h_N})$ we get
    \begin{equation}\label{eq:final2}
   % \begin{split}
     %   \Q_N
        %\left( \capacity_N(\cC\cap \Ann^{\bfa,\bfb}_R)
        %\leq t \Fbox((\bfb-\bfa)R),
        %\capacity_N(\cC \cap \Ann^{\bfd,\bfe}_R) 
        %\geq s \Fbox((\bfe-\bfd)R) 
        %,\bfG^{\text{c}} 
        %\right) \\
        %&\leq 
         \E^{\Q_N}\Big[
        \P_N^{\lO}
        \left( 
        \capacity_N(\mathcal{C} \cap \Ann^{\bfd,\bfe}_R) 
        \geq s
        \Fbox((\bfe-\bfd)R)
        \right)
        \1_{\bfG^{\text{c}}}\Big] 
        \leq \arctan 
        \left(\psi_s 
        \right)e^{-c^\prime \ell} 
        \leq \arctan 
        \left(\psi_s 
         e^{-c/t}\right)
%    \end{split}
    \end{equation}
    where we also used the inequality $\arctan(\eta x)\geq \eta\arctan(x)$ for all $x>0,\eta\in(0,1)$ and substituted \eqref{eq:L-ell} for $\ell$ in the last line. Together, the bound obtained below \eqref{eq:final1}  and \eqref{eq:final2} yield \eqref{eq:pre_recursion}.
 %   Since there exists $c>0$ such that $\log(f(R) R) \geq c\log R$ for all $R\geq \Cr{arm_R}$ and $f(R) >1/\log R$, we have that $\frac{\log(\bfd R)}{h_N}\geq c\frac{\log(R)}{h_N}$ and may conclude the proof of \eqref{eq:pre_recursion}.
\end{proof}

Now we are ready to complete the proof of Theorem \ref{prop:one_arm} using the following recursive relation. We will refer to \cite{drewitz_critical_2024} whenever adaptations to our setup only incur straightforward modifications. Let
\begin{multline} \label{eq:def_recursion}
    s^{\bfa,\bfb}_{R,\epsilon}
    \stackrel{\text{def.}}{=} \sup
    \Bigl\{
    s\geq  0: \text{for all } t\leq s, \
    \P_N
    \left(
    0 \leftrightarrow\partial B_R, \,
    \capacity_N(\mathcal{C} \cap \Ann^{\bfa,\bfb}_R) \leq t\Fbox((\bfb-\bfa)R)
    \right) \\
    \leq
\arctan 
        \Big[
        \left(\left(1\vee\tfrac{\log R}{h_N}\right)\Fbox( R)\right)^{-\frac{1}{2}} e^{-\frac{\epsilon}{t}}
        \Big]
    \Bigr\}  .
\end{multline}

\begin{proposition}
    There exist $\Cl[c]{recursion}$ such that with $s^{\cdot}_{R}=s^{\cdot}_{R,\Cr{recursion}}$, one has for all $N\geq 1$, all $R\geq C$ such that $R\leq \Cr{C:range}N$, $\frac{2}{3}>\bfb>2\bfa $, $2\bfd<\bfe\leq \frac{\bfa}{4}$, %\WZ{deleted:$\frac{2}{\log R}\leq 2\bfd$ - not used}
    \begin{equation} \label{eq:rec-s}
      s^{\bfa,\bfb}_{R}\geq c
      \Big(
      \bfa^3 \wedge 
      \log
      \big(\tfrac{1}{s^{\bfd,\bfe}_{R}\vee\bfd}
      \big)^{-1}
      \Big).
    \end{equation}
\end{proposition}
\begin{proof}
    Applying Lemma \ref{lem:pre_recursion} and  \eqref{eq:def_recursion}, one obtains upon distinguishing whether or not $ \capacity_N(\mathcal{C} \cap \Ann^{\bfd,\bfe}_R) \leq s\Fbox((\bfd-\bfe)R)$ below that for all $t>0,s\leq s^{\bfd,\bfe}_{R,\epsilon}$,
    \begin{multline*}
        \P_N
        \left(
        0 \leftrightarrow\partial B_R, \,
        \capacity_N(\mathcal{C} \cap \Ann^{\bfa,\bfb}_R) \leq t\Fbox((\bfb-\bfa)R)
        \right) \\
         \leq  \arctan 
        \Big[
        \left(\left(1\vee\tfrac{\log R}{h_N}\right)\Fbox( R)\right)^{-{1}/{2}} e^{-\frac{\epsilon}{t}}
        \Big]
        + 
       \arctan 
        \left[
         \psi_{ s} \, e^{C\alpha^h_N-\frac{\Cr{exp_const}}{t}}
        \right].    
    \end{multline*}
     One readily checks that for $\bfd\in(0,1)$, the inequality $\bfd\Fbox(R) %= (\bfd R)\wedge \frac{h_N}{\bfd K_0\left(\frac{R\vee h_N}{N}\right)}\leq (\bfd R)\wedge \frac{h_N}{ K_0\left(\frac{(\bfd R)\vee h_N}{N}\right)}
     \leq C \Fbox(\bfd R)$ holds.
    Using this, the fact that $\psi_s$ is proportional to $s^{-1/2}$ (see \eqref{eq:pre_recursion}), that $\alpha^h_N \leq \bfa^{-3}$ (see below \eqref{eq:Ade_out}), and taking $s=\exp\{-\Cr{exp_const}t^{-1}\}$ and $\epsilon=\frac{\Cr{exp_const}}{4}$, the claim \eqref{eq:rec-s} readily follows, see the proof of \cite[Proposition 3.1]{drewitz_critical_2024} for a similar argument; in particular, the condition $s\leq s^{\bfd,\bfe}_{R,\epsilon}$ needed for the previous display to hold implies that $t \leq c
      \log
      \big({1}/{s^{\bfd,\bfe}_{R}}
      \big)^{-1}$.    
\end{proof}

\begin{proof}[Proof of Theorem \ref{prop:one_arm}]
    In view of \eqref{eq:def_recursion}, we will first show that for $\bfa_0=\frac{4}{\log R}$ and $\bfb_0=\frac{1}{\log(\log(R)\vee 1)}$,
    \begin{equation} \label{eq:recursion_step0}
        s^{\bfa_0,\bfb_0}_R \geq 
    \begin{cases*}
      c, & if $\frac{(\bfb_0-\bfa_0)R}{\log((\bfb_0-\bfa_0)R\vee 2)} \geq \frac{h_N}{K_0\left({[((\bfb_0-\bfa_0)R)\vee h_N]}/{N}\right)}\equiv \gamma_N$ \\
      R^{-1}, &  otherwise 
    \end{cases*}
    .
    \end{equation}
    When $\frac{(\bfb_0-\bfa_0)R}{\log((\bfb_0-\bfa_0)R\vee 2)} \geq \gamma_N$, it follows from \eqref{eq:line-cap} (say with $\varepsilon=\tfrac12$) and \eqref{eq:cluster_cap_comparison} that $\{0 \leftrightarrow\partial  B_R\}$ implies the event $\{\capacity_N(\mathcal{C} \cap \Ann^{\bfa_0,\bfb_0}_R) \geq \tfrac{\Cr{lb_line}}{2} \Fbox((\bfb_0-\bfa_0)R)\}$. It's therefore trivially true in view of \eqref{eq:def_recursion} that $s^{\bfa_0,\bfb_0}_R\geq \Cr{lb_line}/2$. For the latter case in \eqref{eq:recursion_step0}, since $R^{-1}\Fbox((\bfb_0-\bfa_0)R) \leq \bfb_0 \leq C \,\capacity_N(\{0\})$ for $R\geq C$ and $ \{0 \leftrightarrow\partial  B_R\}\subset \{\capacity_N(\mathcal{C} \cap \Ann^{\bfa_0,\bfb_0}_R) \geq \capacity_N(\{0\}) \}$, we conclude the proof of \eqref{eq:recursion_step0}.

    Note \eqref{eq:recursion_step0} already implies \eqref{eq:one_arm_exp} when $\frac{(\bfb_0-\bfa_0)R}{\log((\bfb_0-\bfa_0)R\vee 2)} \geq \gamma_N$. Otherwise one defines recursively $\log_0(R)=\log(R)$, $\log_{k+1}(R)=\log(\log_k(R))\vee 1$ and lets 
$a_k={4}/{\log_k(R)}$, $b_k={4}/{\log_{k+1}(R)}$ for all $k \geq 0$. The claim \eqref{eq:one_arm_exp} now follows from the same argument as in the proof of \cite[Theorem 1.2]{drewitz_critical_2024}.
    \end{proof}

We conclude with a few comments.
\begin{remark}\label{R:final}

\begin{enumerate}[label={\arabic*)}]
\item (Wedges). \label{R:final-w} The threshold $h_N=\log N$ appearing in \eqref{eq:variance-GFF}, \eqref{critical_one_arm-macro} (see also Fig.~\ref{F:plateau}), is closely related to that identified by Lyons in \cite[Section 6]{zbMATH03804598} as characterising the recurrence/transience threshold for sub-graphs of $\Z^3$ with vertex set
$$
\mathscr{W}^h
= \{ x=(y,z) \in \mathbb{Z}^{2+1}: |z| \leq h (|y|)\}
$$
(by \cite{zbMATH03804598} the graph $\mathscr{W}^h$ is transient if and only if $\sum_N \frac{1}{N h_N}<\infty$). It would be interesting to derive analogues of our results for these graphs, as well as similar results for other percolation or spin models of interest (with $2$ playing the role of the lower-critical dimension of the problem).
\item (Scaling). A natural question is to try to assess the possible types of scaling (and hyper-scaling) laws of the problem on $\mathbb{S}_N$. For instance, as a consequence of \eqref{eq:cluster_cap_comparison} and the analogue of \eqref{eq:caplaw} for $\{\varphi \geq a\}$, together with the results of Section~\ref{sec:cap-slab}, one  obtains a Gaussian decay (in $a$) for the probability to connect $0$ to $\partial B_N$ in $\{\varphi \geq a\}$, $a > 0$.
One may then seek to refine these bounds, in a manner similar to what has been done in \cite{goswami_radius_2022} for Gaussian free field level sets on $\Z^d$, $d \geq 3$, in \cite{drewitz_arm_2023} for the  corresponding problem on the cables, where more precise results can be obtained, in \cite{MR4749810} for a more general class of Gaussian fields, and in \cite{prevost_first_2024,GRS24.1,goswami2025sharp} for the vacant set of random interlacements. In the present case this will inevitably lead to scaling behaviour beyond these familiar regimes. We will return to this elsewhere \cite{rz25b}.
\end{enumerate}

\end{remark}

\noindent\textbf{Acknowledgements.} The research of WZ is supported by the CDT in Mathematics of Random Systems and by the  President's PhD Scholarship scheme at Imperial College, London. The research of PFR is supported by the European Research Council (ERC) under the European Union’s Horizon Europe research and innovation program (grant agreement No 101171046).
\appendix

\section{Heat kernel approximation}\label{A:HK}
In this appendix, we collect several results which allow us to approximate various functionals involving heat kernels of random walks by their
corresponding Brownian counterparts with sufficiently small error. Results in this section hold for general $d\geq2$. 
 For $d\geq 2$ and $r\in(0,1]$, we denote by $P^{d,r}_y$ the canonical law of the continuous-time simple random walk on $\Z^d$ with jump rate $r>0$ starting at $y\in\Z^d$  and write $Y=(Y_t)_{t \geq 0}$ for the corresponding canonical process. That is,~if $(\tilde{Y}_n)_{n \geq 0}$ is the corresponding discrete skeleton, which is the discrete-time simple symmetric random walk on $\Z^d$, then $Y_t=\tilde{Y}_{n_t}$  with $(n_t)_{t \geq 0}$ a rate $r$ Poisson process on $\R_+$, independent of $(\tilde{Y}_{t})_{t \geq 0}$. We let
\begin{equation}\label{eq:BM-kernel}
    p^{d,r}(y,t) \stackrel{\text{def.}}{=}\Big(\frac{d}{2\pi rt}\Big)^{d/2} \exp\Big\{-\frac{d|y|^2}{2rt}\Big\},
\end{equation}
the transition density of a $d$-dimensional Brownian motion with variance $r/d$ at time one. We write $c,C,c^\prime,C^\prime$ etc.~in the sequel for constants that may depend on $d$ and $r$. 
We start by isolating the following bound, which shows that it's unlikely for $Y$ to move a large distance when $t$ is small.

\begin{lemma}
    For all $y\in \Z^d$, $M\geq 1$ and $t>0$  such that $|y|\geq t/M$, we have that 
    \begin{equation} \label{eq:exit_time}
        P^{d,r}_0\big(Y_t=y\big) \leq   Ce^{-c {|y|}/{M}}. 
    \end{equation}
\end{lemma}
\begin{proof}
    Let $(\bar Y_n)_{n\geq 0}$ be the discrete-time skeleton of $(Y_t)_{t\geq 0}$. Let $N_t$ be a Poisson random variable with parameter $t$ independent of $(\bar Y_n)_{n\geq 0}$ and let $\bar\tau_{s}={\min\{n\geq 0:|\bar Y_n|\geq s\}}$ for $s \geq 0$.  Applying  \cite[(2.50)]{lawler_random_2010} with the choice $r\coloneqq \frac{CM}{m}$, where $m \coloneqq \lfloor |y | \rfloor $, one obtains that for all $y \neq 0$,
$$
 P^{d,r}_0\big(Y_t=y, \, N_t \leq rm^2 \big) \leq  P(\bar\tau_m\leq r m^2) \leq Ce^{-c/r}, 
$$    
which contributes to the bound in \eqref{eq:exit_time}. Using the standard Poisson tail bound $P(N_t\geq t(1+u))\leq \exp\{-tu(\log(1+u)-1)\}$ for $u\geq 0$, 
one readily deduces for all $y$ such that $|y|\geq t/M$ upon choosing $C$ large enough that $P(N_t > rm^2) =P(N_t > CMm) \leq C'\exp\{-c^\prime Mm\} $. Overall, \eqref{eq:exit_time} follows. 
\end{proof}

The following results rely on the use of an appropriate local limit theorem.

\begin{proposition} \label{prop:lclt_approx_y}
For all $d\geq 2$, $r\in(0,1]$, $M\geq 0$ and $y \in \Z^d \setminus \{0\}$, 
\begin{equation} \label{eq:lclt_approx_y}
    \int_{M}^\infty \big\lvert P^{d,r}_0 (Y_t=y)
    -p^{d,r}(y,t) \big\rvert  \,dt
    \leq \frac{C}{( M\vee|y|)^d}.
\end{equation}
Moreover, \eqref{eq:lclt_approx_y} remains true for all $M \geq 1$ if $y=0$.
\end{proposition}

\begin{proof}
Using \eqref{eq:BM-kernel} and the substitution $s=\frac{d|y|^2}{2rt}$, we have that
\begin{align} \label{eq:g3_hk_tsmall}
\begin{split}
    \int_0^{|y|} p^{d,r}(y,t)  \,dt 
    \leq \frac{C}{|y|^{d-2}}\int_{\frac{d|y|}{2r}}^\infty  s^{(d-4)/2} e^{-s} \,ds 
    \leq Ce^{-c|y|}.
\end{split}
\end{align}
Hence by combining (\ref{eq:exit_time}) with $M=1$ with (\ref{eq:g3_hk_tsmall}) and crudely bounding the difference below by a sum, we obtain that for all $y \in \Z^d$, 
\begin{equation} \label{eq:g3_t_small}
  \int_{0}^{|y|} \big\lvert P^{d,r}_0 (Y_t=y)
    -p^{d,r}(y,t) \big\rvert  \,dt  \leq Ce^{-c|y|}.
 \end{equation}
    Moreover, it follows from taking $k=d+3$ in \cite[(2.9)]{lawler_random_2010} that for $y \neq 0$ or $M \geq 1$,
  \begin{multline} \label{eq:g3_tbig}
 \int_{M \vee |y|}^\infty \big\lvert P^{d,r}_0 (Y_t=y)
    -p^{d,r}(y,t) \big\rvert  \,dt
  \leq \int_{M \vee |y|}^\infty
  \frac{C}{t^{\frac{d+2}2}} \left[ 
    \Big( \frac{|y|^{d+3}}{t^{(d+3)/2}}+1\Big) 
    e^{- \frac{c|y|^2}{ t }}
    +\frac{1}{t^{d/2}}\right]\,dt \\
 \leq \int_{M\vee|y|}^{(M\vee |y|)^2}
\frac{C}{t^{\frac{d+2}2}}       \Big( \frac{2 |y|}{\sqrt{t}} \Big)^{d+3} e^{- \frac{c|y|^2}{ t }} 
        \,dt    
      + C'\int_{(M\vee |y|)^2}^\infty
      \frac{dt }{t^{\frac{d+2}2}}
    + C \int_{(M\vee |y|)}^\infty
      \frac{dt}{t^{d+1}}. 
    %  \leq \frac{C}{(M\vee|y|)^d}.
    \end{multline} 
    The second and third terms in the second line of \eqref{eq:g3_tbig} are readily seen to give a contribution bounded by $C(M\vee|y|)^{-d}$, and so is the first one, after substituting $t=(M \vee |y|)^2/s$ and replacing the (finite) upper integral end by $+\infty$, which yields a bound of the form $C(M \vee |y|)^{-d} \int_1^{\infty} s^{d + 1/2}e^{-cs} ds = C' (M \vee |y|)^{-d}$. Together with \eqref{eq:g3_t_small}, \eqref{eq:g3_claim} follows, and the addendum concerning $y=0$ follows from \eqref{eq:g3_tbig} alone.
\end{proof}

The following estimates are tailored to our purposes.

\begin{proposition} \label{prop:lclt_approx_z}
Let $(X_t)_{t\geq0}$ be the canonical process with law $P\equiv P_0^{d+1,1}$ and let $(Y_t)_{t\geq0}\subset \Z^{d}$ and $(Z_t)_{t\geq0}\subset \Z$ be the first {$d$} components of $(X_t)_{t\geq0}$ and last component of $(X_t)_{t\geq0}$ respectively. We have that for all $M>0$, $y\in \Z^d$ and $z\in \Z$ such that $|z|\leq h/2$,
  \begin{equation} \label{eq:prop_lclt_approx_z}
    \int_{M}^\infty P(Y_t=y) 
     \sum_{k \in \Z \setminus \{0\}} \Big\lvert p^{1,\frac{1}{d+1}}(z+ k h,t)- P(Z_t=z+ kh)\Big\rvert \,dt 
     \leq \frac{C}{h^2M^{\frac{d-1}{2}}}.
  \end{equation}
  If instead $(X_t)_{t\geq0}$ has law $Q\equiv P_0^{d+1,r}$ and $(Y_t, Z_t)_{t\geq0}\subset \Z^{d} \times \Z$ are as above, then for all $M>0$, $y\in \Z^d$ and $z\in \Z$ such that $|z|\leq h/2$,
\begin{align}  \label{eq:prop_lclt_main_hbig}
    \begin{split}
        &\quad \sum_{k \in \Z \setminus \{0\}} \int_{M}^{\infty} 
        \left\lvert Q \big((Y_t,Z_t)=(y,z + kh)\big) 
        - p^{d+1,r}\big((y,z+ kh),t\big)  \right\rvert \, dt 
    \leq  \frac{C  }{h^2M^{\frac{d-1}{2}}}.
    \end{split}
\end{align}
\end{proposition}

\begin{proof}
First note that $Y$ has law $P_0^{d,\frac{d}{d+1}}$ and $Z$ has law $P_0^{1,\frac{1}{d+1}}$ under $P$. Using $|z|\leq h/2$ and \cite[(2.4)]{lawler_random_2010} (which also holds in continuous time, see the end of \cite[page 24]{lawler_random_2010}), we have that there exists $C$ such that for all $t>0$ and $k \in \Z\setminus \{0\}$,
  \begin{equation*}
    \big\lvert P_0^{1,\frac{1}{d+1}}(Z_t=z+ kh)
    - p^{1,\frac{1}{d+1}}(z+ k h,t) \big\rvert 
    \leq {C}t^{-1/2}(k-\tfrac{1}{2})^{-2}h^{-2} .
  \end{equation*}
Hence, the left-hand side of (\ref{eq:prop_lclt_approx_z}) is bounded above by
$    \frac{C}{h^2}\int_{M}^{\infty} \frac{1}{t^{1/2}} P (Y_t=y) \,dt 
    \leq \frac{C}{h^2}\int_{M}^{\infty} \frac{1}{t^{(d+1)/2}} \,dt \leq {C}{h^{-2} M^{\frac{1-d}{2}}}$, where the last inequality follows from \cite[(5.18)]{barlow_random_2017}. This yields \eqref{eq:prop_lclt_approx_z}. As to \eqref{eq:prop_lclt_main_hbig}, since $|z|\leq h/2$ and by \cite[(2.4)]{lawler_random_2010} (which also works in continuous time, see the end of \cite[Page 24]{lawler_random_2010}), 
  \begin{equation*}
    \big\lvert Q \big((Y_t,Z_t)=(y,z +\sigma kh)\big) 
        - p^{d+1,r}\big((y,z+\sigma kh),t\big) 
    \big\rvert 
    \leq {C}{t^{-(d+1)/2}(k-1/2)^{-2}h^{-2}}, \quad t >0.
  \end{equation*}
The conclusion follows similarly as with \eqref{eq:prop_lclt_approx_z}.
\end{proof}

\section{Properties of $K_0$}\label{A:K_0}

Recall that $K_0(\cdot)$ denotes the zeroth-order modified Bessel function of the second kind, defined in \eqref{eq:K_0}. The following estimate is frequently used. 

\begin{lemma} \label{L:Bessel}For all $t>0$,
   \begin{equation} \label{eq:bessel_ub}
        \log(1/t) \leq K_0(t)\leq (\log(1/t)\vee 0)+C.
    \end{equation}
    In particular, $K_0(t)\sim\log({1}/{t})$ as $t\to 0^+$.
\end{lemma}

\begin{proof} 
    By \cite[(8.447.3), (8.447.1) and (8.365.4)]{gradshteyn_table_2014}, we have that for all $t>0$,
    \begin{equation} \label{eq:bessel_series}
        K_0(t)=\log(2/t)I_0(t)+\sum_{n=0}^\infty\frac{(t/2)^{2n}}{(n!)^2}\psi(n+1),
    \end{equation}
    where $I_0(t)=\sum_{n=0}^\infty\frac{(t/2)^{2n}}{(n!)^2}$ and $\psi(n+1)=-\gamma+\sum_{j=1}^n\frac{1}{j}$, with $\gamma$ denoting Euler's constant. 
    The lower bound in \eqref{eq:bessel_ub} is trivially true when $t\geq 1$ since $\log(1/t)\leq 0$ and $K_0(t)\geq 0$ in view of \eqref{eq:K_0}. When $t\in(0,1)$, singling out the $n=0$ term in the series for $I_0(t)$ and \eqref{eq:bessel_series}, we get
    \begin{equation} \label{eq:bessel_rearrange}
        K_0(t)=\log(1/t) +
        \left(\log(2)-\gamma + 
        (\log(2/t)-\gamma)\sum_{n=1}^\infty\frac{(t/2)^{2n}}{(n!)^2}
        +\sum_{n=1}^\infty\frac{(t/2)^{2n}}{(n!)^2} \sum_{j=1}^n\frac{1}{j}
        \right).    
    \end{equation}
    This concludes the proof of the lower bound in \eqref{eq:bessel_ub} since $\log(2/t)-\gamma\geq \log(2)-\gamma >0$ (the expression $\log$ refers to the natural logarithm), hence the expression inside the bracket in \eqref{eq:bessel_rearrange} is positive.

    For the upper bound in \eqref{eq:bessel_ub}, 
    when $t\in(0,1]$, using \eqref{eq:bessel_series}, $I_0(t)\leq I_0(1)\leq 2$, and that the second term in \eqref{eq:bessel_series} is bounded, one obtains $K_0(t)\leq \log(1/t)+C$. When $t \geq 1$, by \eqref{eq:K_0} and the inequality $\cosh(t)\geq t$ we have that $K_0(t)\leq \int_0^\infty e^{-s}\,ds=1$. \end{proof}

 \begin{proof}[Proof of Lemma~\ref{prop:bessel_approx}]
First note that (\ref{eq:bessel_approx2}) is plainly true when $h\geq 2^k$ since all the argument of $K_0(\cdot)$ are $\tfrac{h}{N}$ in this case. Now assume $h < 2^k$. By (\ref{eq:bessel_ub}), we have for $h < 2^k$ that
    \begin{multline*}
     \sum_{j=1}^{\floor{\log_2(h)}}  (2^j)^2  K_0\left({h}/{N}\right) \leq  C h^2  K_0\left({h}/{N}\right)  \\
    \leq C h^2  + C(2^k)^2 \bigg(\Big(\log\big({N}/{2^k}\big)%\times
      + \big({h}/{2^k}\big)^2\log\big({2^k}/{h}\big)\Big)\vee 0\bigg) 
     \leq  C \,(2^k)^2 \left(K_0\big({2^k}/{N}\big)\vee 1\right).
    \end{multline*} For the rest of the series, let $k^\prime$ denote the largest integer $j$ in $[ \floor{\log_2(h)}+1, k]$  such that $\log(N/2^{j})>0$ when it exists, and set $k^\prime = \floor{\log_2(h)}$ otherwise.
    We then have, for all $j \leq k'$, 
    \begin{equation*}
        \log(N/2^{j})=\log(N/2^{k^\prime}) +(k^\prime -j)\log(2)
        \leq C2^{k^\prime-j}\log(N/2^{k^\prime}).
    \end{equation*}
    It then follows from the above inequality and \eqref{eq:bessel_ub} that
    \begin{align*}
    \begin{split}
        \sum_{j=\floor{\log_2(h)}+1}^{k}  (2^j)^2  K_0\big({2^j}/{N}\big)  
        &\leq C\, (2^k)^2  + C\sum_{j=\ceil{\log_2(h)}+1}^{k^\prime}  (2^j)^2  \times2^{k^\prime-j}\log(N/2^{k^\prime}) \\
        &\leq C (2^k)^2 + C(2^{k^\prime})^2 \log(N/2^{k^\prime}),
    \end{split}
    \end{align*}
    which concludes of proof of \eqref{eq:bessel_approx2} if $k^\prime =k$. If not, then $k' < k$ and by the definition of $k^\prime$ we know that $\log(N/2^{k^\prime +1})<0$, hence
    \begin{equation*}
        (2^{k^\prime})^2 \log(N/2^{k^\prime})=
        (2^{k})^2 \times\big( 4^{k^\prime -k}\log(N/2^{k^\prime+1 }) +4^{k^\prime -k}\log(2)\big)\leq C(2^{k})^2.
    \end{equation*}
    Overall \eqref{eq:bessel_approx2} follows.
    
We now move on to showing \eqref{eq:K_0_bound} and \eqref{eq:K_0_asymp}. Stirling's estimate gives, for every $R\geq 1$, $\log(R!)\geq R\log(R)-R$. Hence for $R<N$, we have that, with the sum over $k$ ranging from $1$ to $R$ in the sequel,
\begin{equation} \label{eq:sum_log_ub}
    \sum_{ k }\log(N/k) =R\log(N)-\log(R!) \leq R\log(N/R)+R.
\end{equation}
Together, \eqref{eq:bessel_ub} and \eqref{eq:sum_log_ub} yield that
\begin{align*}
\begin{split}
        \sum_{k} K_0\left({k}/{N}\right)
        &\leq CR+   \sum_{k}\left(\log\left({N}/{k}\right)\vee 0\right)
        \leq CR+ (R\wedge N)\log\left(N/(R\wedge N)\right)+ (R\wedge N),
\end{split}
\end{align*}
which concludes of proof of \eqref{eq:K_0_bound} upon applying \eqref{eq:bessel_ub} once more. For \eqref{eq:K_0_asymp}, note that for all $\epsilon\in(0,1)$, there exists $\Cr{N_big}(\epsilon)$ large enough such that whenever $\frac{R}{N}\leq \frac{1}{\Cr{N_big}}$, one has $R\wedge N=R$ and $C'R\leq \epsilon R\log(N/R)$. Combining this with the last display and \eqref{eq:bessel_ub} we thus deduce that 
\begin{equation*}
      \sum_{k}K_0\left({k}/{N}\right)
    \leq (1+\epsilon)R\log(N/R)\leq (1+\epsilon)RK_0(R/N).
\end{equation*}
\end{proof}

\bibliography{bibliography}
\bibliographystyle{abbrv}

\end{document}